\documentclass[11pt]{article}

\usepackage[margin=1.0in]{geometry}
\usepackage[utf8]{inputenc}            % allow utf-8 input
\usepackage[T1]{fontenc}               % use 8-bit T1 fonts
\usepackage[dvipsnames]{xcolor}
\usepackage[
    colorlinks=true,
    linkcolor=RoyalPurple,   % Color for cross-references, e.g., \ref, \eqref
    citecolor=Blue,     % Color for citations, e.g., \cite
    urlcolor=Violet     % Color for URLs
]{hyperref}
\usepackage[mathscr]{euscript}
\usepackage[T1]{fontenc}
\usepackage{multirow}
\usepackage{graphicx} % Required for inserting images

\usepackage{url}                        % simple URL typesetting
\usepackage{booktabs}                   % professional-quality tables
\usepackage{amsfonts}                   % blackboard math symbols
\usepackage{nicefrac}                   % compact symbols for 1/2, etc.
\usepackage{bbm}
\usepackage{microtype}                  % microtypography
\usepackage{algorithm}
\usepackage{algpseudocode}
\usepackage[dvipsnames,table]{xcolor} % colored text
\usepackage{lmodern}
\usepackage{comment} %comment environment
\usepackage{tikz}
\usepackage[numbers, sort]{natbib}
\usepackage{amsmath, amssymb, amsfonts}
\usepackage{amsthm, thmtools, thm-restate}
\usepackage{mathtools, nccmath, extarrows}
\usepackage[shortlabels]{enumitem} %resume enumerate
\usepackage{subcaption}
\usepackage{bm}
\usepackage{scalerel}
\usepackage{wrapfig}

\usepackage{stackengine}
\mathtoolsset{showonlyrefs} % Only add number equations if they are referenced

\stackMath

\DeclareFontEncoding{LS1}{}{}
\DeclareFontSubstitution{LS1}{stix}{m}{n}

\makeatletter
  \newcommand*\textmathversion{\csname textmv@\math@version\endcsname}
  \newcommand*\textmv@normal{m}
  \newcommand*\textmv@bold{b}
\makeatother

\newtheorem{theorem}{Theorem}[section]
\newtheorem{proposition}[theorem]{Proposition}
\newtheorem{lemma}[theorem]{Lemma}
\newtheorem{corollary}[theorem]{Corollary}
\newtheorem{claim}[theorem]{Claim}

\theoremstyle{definition}

\newtheorem{assumption}{Assumption}

\newtheorem{example}[theorem]{Example}

\usepackage[most]{tcolorbox}
\usepackage{varwidth}
 \newtcbox{\mybox}{colback=blue!5,
	colframe=blue!30!black, center, enhanced, varwidth upper}
\definecolor{pil}{RGB}{200, 0, 55}
\newtcbtheorem[use counter=theorem, number within=section]{THM}{Theorem}%
{	colframe=blue!30!gray, fonttitle=\bfseries,
	colback=blue!5,  fontupper=\itshape, top=1mm,bottom=1mm, boxrule=0.5pt}{Theorem}

\newcommand{\pil}[1]{{\color{pil}{[P: #1]}}}

\usepackage{bm}
\usepackage[most]{tcolorbox}

\newcommand{\cD}{\mathcal{D}}

\newcommand{\cF}{\mathcal{F}}

\newcommand{\cH}{\mathcal{H}}

\newcommand{\cL}{\mathcal{L}}
\newcommand{\cM}{\mathcal{M}}

\newcommand{\cP}{\mathcal{P}}

\newcommand{\cS}{{\mathcal{S}}}

\newcommand{\N}{\mathbf{N}}
\newcommand{\RR}{{\mathbf R}}
\newcommand{\R}{\RR}
\newcommand{\NN}{{\mathbf N}}

\newcommand{\xs}{x^\star}

\newcommand{\Lxj}{L^{x}_{tj}}
\newcommand{\Lyj}{L^{y}_{tj}}
\newcommand{\Lxjd}{L^{x}_{\delta j}}
\newcommand{\Lyjd}{L^{y}_{\delta j}}
\newcommand{\Lxjdo}{\overline{L^{x}_{\delta j}}}
\newcommand{\Lyjdo}{\overline{L^{y}_{\delta j}}}

\newcommand{\cDl}{\cD_{\delta}}
\newcommand{\indexM}{{\scriptscriptstyle \cM}}
\newcommand{\lambdamin}{\operatorname{\lambda_{min}}}
\newcommand{\lambdamax}{\operatorname{\lambda_{max}}}

\newcommand{\dist}{\operatorname{dist}}
\newcommand{\range}{\operatorname{range}}

\newcommand{\argmin}{\operatornamewithlimits{argmin}}

\newcommand{\dom}{\operatorname{dom}}

\newcommand{\epi}{\operatorname{epi}}

\newcommand{\prox}{\operatorname{prox}}
\newcommand{\cl}{\operatorname{cl}}

\DeclarePairedDelimiter{\dotp}{\langle}{\rangle}
\def\shortdisplay{\setlength{\abovedisplayskip}{5pt}%
	\setlength{\belowdisplayskip}{5pt}%
	\setlength{\abovedisplayshortskip}{2pt}%
	\setlength{\belowdisplayshortskip}{2pt}}

\let\oldselectfont\selectfont
\def\selectfont{\oldselectfont\shortdisplay}

\definecolor{block-gray}{gray}{0.93}
\newtcolorbox{myquote}{colback=block-gray,
	,grow to right by=-10mm,grow to left by=-10mm,
	boxrule=0pt,boxsep=0pt
}

\newcommand{\sub}{\cF}
\newcommand{\zs}{z^\star}
\newcommand{\zstilde}{\underline{z}^\star}
\newcommand{\ys}{y^\star}
\newcommand{\sol}{\mathcal{S}^\star}
\newcommand{\sols}{\mathcal{S}^\star_L}

\newcommand{\ut}{\underline{t}}
\newcommand{\ot}{\overline{t}}

\newcommand{\ballt}[2]{\mathbf{B}_{#2}(#1)}
\newcommand{\ballP}[2]{\mathbf{B}^P_{#2}(#1)}

\newcommand{\cballt}[2]{\overline{\mathbf{B}}_{#2}(#1)}
\newcommand{\cballP}[2]{\overline{\mathbf{B}}^P_{#2}(#1)}
\newcommand{\ball}{\mathbf{B}}
\newcommand{\cball}{\overline{\mathbf{B}}}

\newcommand{\indexB}{B}
\newcommand{\indexBs}{B_a}
\newcommand{\indexBw}{B_d}
\newcommand{\indexN}{N}

\newcommand{\rate}{\frac{1}{2}}
\newcommand{\irate}{2}

\newcommand{\cPo}{\overline{\cP}}
\newcommand{\cPe}{\cP_{\varepsilon}}
\newcommand{\cSn}{\mathcal{S}^n}
\newcommand{\cSnp}{\mathcal{S}^n_+}
\newcommand{\cSnpp}{\mathcal{S}^n_{++}}
\newcommand{\bRR}{\overline{\RR}}
\newcommand{\lambdaminp}{\operatorname{\lambda^+_{min}}}

\newcommand{\step}{p}

\newcommand{\wh}[1]{\widehat{#1}}
\newcommand{\YY}{\mathcal{Z}}
\newcommand{\YN}{\YY_0}
\newcommand{\Bt}{\mathcal{Z}_t}
\newcommand{\eps}{\varepsilon}
\DeclareMathOperator{\relint}{relint}

\newcommand{\Lxt}[1]{\overline{L_{j}^x}(#1)}
\newcommand{\Lyt}[1]{\overline{L_{j}^y}(#1)}
\newcommand{\Lx}{\overline{L_{j}^x}}
\newcommand{\Ly}{\overline{L_{j}^y}}

\newtheorem{innercustomthm}{Assumption}
\newenvironment{customassumption}[1]
{\renewcommand\theinnercustomthm{#1}\innercustomthm}
{\endinnercustomthm}

\newcommand{\proj}{{\rm proj}}

\begin{document}

% \title{Active Set Identification and Rapid Local Convergence for Primal-Dual Degenerate Problems}
\title{Active set identification and rapid convergence for degenerate primal-dual problems}
% \title{Degenerate primal-dual problems: \\ active set identification and rapid convergence}
%\title{PDHG convergence for degenerate partially smooth problems}
\author{
  Mateo D\'\i az\thanks{Department of Applied Mathematics and Statistics and Mathematical Institute for Data Science, Johns Hopkins University, Baltimore, MD 21218, USA. MD was partially supported by NSF awards CCF 2442615 and DMS 2502377.}
\and Pedro Izquierdo Lehmann\footnotemark[1]
\and Haihao Lu\thanks{Sloan School of Management, MIT, Cambridge, MA 02142, USA. HL is partially supported by AFOSR Grant No. FA9550-24-1-0051 and ONR Grant No. N000142412735.} \and Jinwen Yang\thanks{Department of Statistics, University of Chicago, Chicago, IL 60637, USA. JY is partially supported by AFOSR Grant No. FA9550-24-1-0051.}}

\date{}

\maketitle

\begin{abstract}
Primal-dual methods for solving convex optimization problems with functional constraints often exhibit a distinct two-stage behavior. Initially, they converge towards a solution at a sublinear rate. 
%Then, after a certain point, all iterates identify the set of active constraints, and the convergence enters a faster local linear regime. 
Then, after a certain point, the method identifies the active set---in the sense that all subsequent iterates share the same active constraints---and the convergence enters a faster local linear regime.
Theory characterizing this phenomenon spans over three decades. However, most existing work only guarantees eventual identification of the active set and relies heavily on nondegeneracy conditions, such as strict complementarity, which often fail to hold in practice. We characterize mild conditions on the problem geometry and the algorithm under which this phenomenon provably occurs. Our guarantees are entirely nonasymptotic and, importantly, do not rely on strict complementarity. Our framework encompasses several widely-used algorithms, including the proximal point method, the primal-dual hybrid gradient method, the alternating direction method of multipliers, and the extragradient method. %, and optimistic gradient method.

\end{abstract}
%\tableofcontents

% All the content of the paper goes here, this will update both the arxiv and neurips templates.
\if 0
- Intro

    - Contributions
    - Related Work
    - Outline
    - Notation
\fi
\section{Introduction}
\label{sec:intro}
% \todo{Fix ball notation. \checkmark}\\
% \todo{Indexing $\tilde x_{k}.$ \checkmark}\\
% \todo{Make sure everything looks nice with $r=1/2.$ \checkmark}\\
% \todo{Decide what to do about $\Lxj.$ \checkmark}\\ 
% \todo{Get rid of periods at the end of each section title. \checkmark}\\
% \todo{Fix $\indexM$ everywhere. \checkmark}\\
% \todo{Get rid of $\theta$ everywhere in the body paper. \checkmark}\\
% \todo{Make sure $\|A||_2$ and $\|A\|_{\rm{op}}$ are consistent \checkmark}\\

\noindent We study convex optimization problems of the form 
\begin{equation}
\label{eq:primalProblem-intro}
\min_{x\in\R^n}  f(x)\qquad
    \text{s.t.}  \qquad g_j(x)\leq 0 \quad \text{for all }j \in \{1, \dots, m\}\, ,
\end{equation} 
where $f\colon \RR^n \to \RR$ and $g_j\colon \RR^n \to \RR$ are convex functions. 
% convex--concave minimax problems of the form  
% \begin{equation}
% \label{eq:minimaxProblem}
% \min_{x\in\RR^n}\max_{y\in\RR^{m}_+} \cL(x,y) 
% \qquad \text{with} \qquad 
% \cL(x, y) = f(x) + \langle y, G(x) \rangle,
% \end{equation}  
Primal-dual first-order methods, which simultaneously solve \eqref{eq:primalProblem-intro} and its dual, often display a marked two-stage behavior: initially, their iterates converge sublinearly towards a solution; then, after a finite number of iterations, the iterates identify the set of active constraints and their convergence switches to linear.  
This phenomenon is ubiquitous in practice, and it is exhibited by several first-order methods. To illustrate this phenomenon, Figure~\ref{fig:intro-example} shows the performance of three popular algorithms---the primal-dual hybrid gradient method (PDHG), the alternating direction method of multipliers (ADMM), and the extragradient method (EGM)---on a simple two-dimensional quadratic program with four constraints; see details in Appendix~\ref{app:intro}.  %All converge to a degenerate solution. 
% $
%     \min_{x\in \RR^2} c^Tx+x^TQx
%     \text{ s.t. } Ax\leq b \in \RR^6
% $
% where
% $$
%     c=\begin{bmatrix}
%         0 \\ -1
%     \end{bmatrix},\quad
%     Q = UDU^\top,\quad
%     D=\begin{bmatrix}
%         1 & 0\\
%         0 & 0
%     \end{bmatrix},\quad
%     U \text{ orth.},\quad
%     A=\begin{bmatrix}
%         1 & 1/\kappa\\
%         -1 & 1/\kappa\\
%         0 & 1\\
%         -\alpha & 1
%     \end{bmatrix},\quad
%     b=\begin{bmatrix}
%          1 \\ 1 \\ \kappa-\delta \\ \kappa-\delta\left(1-\frac{\alpha}{\kappa}\right)
%     \end{bmatrix},\quad
%    \delta,\kappa>0,\ \alpha\in[0,\kappa]\,.
% $$ 

 %   - Primal-dual algorithms often exhibit a two stage behavior  To illustrate this behavior let us consider a simple quadratic programming problem. 

    There is extensive work on this phenomenon, from algorithm-specific analyses \cite{lee2012manifold, liang2015activity, liang2017activity, liang2017local, liang2018local} to geometric accounts of the structure that drives it for broad families of algorithms~\cite{burke1988identification, wright1993identifiable, lemarechal2000lagrangian, lewis2002active, hare2007identifying, lewis2022partial}. Although technical, the core idea is simple and directly relevant here, so we briefly recall it. It is convenient to reformulate \eqref{eq:primalProblem-intro} as an equivalent minimax problem \begin{equation}
\label{eq:minimaxProblem}
\min_{x\in\RR^n}\max_{y\in\RR^{m}} \cL(x,y)
\qquad \text{with} \qquad 
\cL(x, y) = f(x) - \iota_{\RR_{+}^{m}}(y) + \langle y, G(x) \rangle
\end{equation}  
here $\iota_{\RR^{m}_{+}}$ denotes the indicator of the nonnegative orthant and $G\colon \RR^n \rightarrow \RR^m$ is the map with $i$th component $(G(x))_i = g_i(x).$
    In turn, a pair $z^\star=(x^\star,y^\star)$ solves \eqref{eq:minimaxProblem} if, and only if,
\[
0 \in \cF(z^\star) 
\qquad \text{with} \qquad 
\cF(z):=
\begin{bmatrix}
\partial_x \cL(z)\\[2pt]
\partial_y \left(-\cL(z)\right)
\end{bmatrix},
\]
where $\partial_x \cL$ and $\partial_y (-\cL)$ denote convex subdifferentials in $x$ and $y$, respectively. Suppose we had an iterative algorithm that generates a sequence $z^k=(x^k,y^k)$ for which there exists a sequence of saddle subdifferentials $\xi^k\in\cF(z^k)$ satisfying
\begin{equation}\label{eq:convergence-intro}
(z^k,\xi^k)\to (z^\star,0)\qquad\text{as}\qquad k\to\infty.
\end{equation}
\begin{wrapfigure}{r}{0.4\textwidth}
  \begin{center}
      \vspace{-1cm}
    \includegraphics[width=0.4\textwidth]{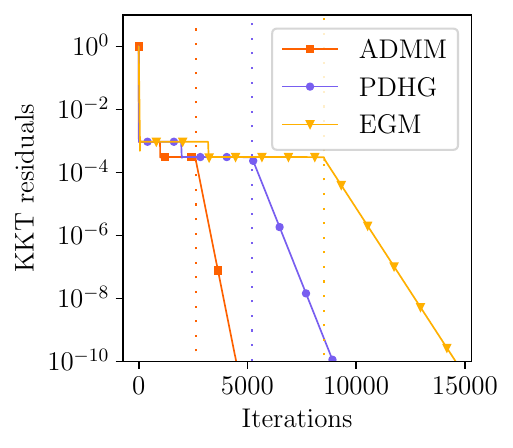}
    \vspace{-1cm}
  \end{center}
  \caption{Distance to solution versus iteration count for several algorithms applied to a degenerate QP. The vertical dotted lines represent the last iteration at which the active set changed.}%\hnote{Change this example to one with a different convergence property.}}
  \label{fig:intro-example}
\end{wrapfigure}
Many methods satisfy this requirement, including the proximal point method (PPM), PDHG, ADMM, and EGM.
Under strict complementarity, the primal-dual space decomposes locally into a manifold $\cM$ passing through $z^\star$---encoding the active constraints---and its complement $\cM^{\rm{c}}$. On $\cM$, the function $\cL$ is smooth; off $\cM$, the subdifferentials are bounded away from zero. Hence, small subdifferentials can arise only on $\cM.$ Consequently, if $\xi^k\to 0$, the iterates must eventually enter $\cM$, yielding finite identification.  Once on $\cM$, the algorithm effectively solves a smooth problem on a manifold; under a local error bound, the algorithm exhibits linear convergence.

Although impressive in scope, this body of work suffers from two limitations. First, existing results rely on the strict complementarity of the limit solution, which can only be verified a posteriori. In practice, algorithms frequently converge to degenerate solutions that violate strict complementarity---as occurs for all algorithms on the problem in Figure~\ref{fig:intro-example}. Second, most existing results guarantee only eventual active set identification without explicit finite-time bounds. These limitations motivate the central question of this work.

\mybox{\it\centering What problem and algorithmic properties enable nonasymptotic, finite-time identification and subsequent linear convergence without requiring strict complementarity?}

To answer this question, we identify a set of mild properties that covers a broad range of algorithms and problems.  Algorithmically, we require: $(i)$ certain structure on the dual updates, ensuring feasibility, and $(ii)$ convergence to a solution, in the sense of \eqref{eq:convergence-intro}, with a sublinear decay of the saddle subgradient norm. Once again, these conditions hold for major primal-dual algorithms, such as PPM, PDHG, ADMM, and EGM. On the problem side, we require $G(\cdot)$ being Lipschitz and an error bound on the saddle subdifferential: there exists $\alpha>0$ such that
\begin{equation}\label{eq:metric-subregularity-intro}
        \alpha\dist_2(z,\sol)\leq \dist_2(0,\sub(z))\qquad\text{for all } z\in\cD\,,
\end{equation}
where $\sol$ is the solution set and $\cD$ is a set containing all iterates of the algorithm. This inequality is known as \emph{metric subregularity} and has been widely studied in variational analysis \cite{ioffe2000metric, rockafellar1998variational}. It generalizes quadratic-growth conditions \cite{drusvyatskiy2013tilt,drusvyatskiy2013second} and holds broadly; for instance, it holds for linear programming (LP) problems where \(\alpha\) can be bounded via the Hoffman constant \cite{hoffman2003approximate, applegate2023faster}.

To state our nonasymptotic bounds, we need to introduce three key ingredients. The first ingredient is the correct notion of an identifiable set. When strict complementarity does not hold, we can still identify a set $\cM$, but unlike before, this set is no longer a manifold; instead, it is a union of manifolds. In particular, assuming the algorithm converges to $z^\star,$ we define the set
\begin{equation}
    \cM =\Big\{(x,y)\in\RR^n \times \RR^m_+\mid\,G(x)_\indexN<0,\,\,y_{\indexN}= 0,\, \text{ and  }\,y_{\indexBs} > 0\big\}
\end{equation}
with $\indexN := \{i \in [m] : g_i(x^\star) < 0\}$ the set of \emph{non-active} constratints and $\indexBs := \{i \in [m] : g_i(x^\star) = 0\, \text{ and }\, y_i > 0\}$ the set of \emph{active} constraints. When strict complementary holds, $ \indexN \cup \indexBs = [m]$, which ensures  $\cM$ is a manifold (in fact, a subspace). When it does not hold, the set of \emph{degenerate} constraints $\indexBw :=[m]\setminus(\indexN \cup \indexBs) = \{i \in [m] : g_i(x^\star) = 0\, \text{ and }\, y_i = 0\}$ is non-empty and $\cM$ can be described as a union of manifolds indexed by subsets of $\indexBw$, %Then 
\begin{equation}
    \cM =\bigcup_{T\subseteq \indexBw} \Big\{(x,y)\in\RR^n \times \RR^m_+\mid\,G(x)_\indexN<0,\,\,y_{\indexN}= 0\, , y_{T}=0\, ,y_{\indexBw\setminus T}>0 \text{ and  }\,y_{\indexBs} > 0\big\}.
\end{equation}
%is a union of exponentially many manifolds.

% But in general, it is not a manifold. %When it does not hold, % In the degenerate case,
% $\cM$ fails to be a manifold.  %distinguish whether the dual variables of weakly active constraints, i.e., $g_i(x^\star) = 0$ with $y_i = 0$, are zero or not.

The second ingredient is the \emph{radius of active-set stability}, defined as the smallest perturbation
%within $\YY:=\RR^n\times\RR^m_+$ 
that flips the sign of some entry of $G(x)_{\indexN}$ or $y_{\indexBs}$; formally,
\begin{equation}
\delta := \sup\Big\{\, t>0 \mid G(x)_{\indexN}<0 \ \text{and}\ y_{\indexBs}>0 \ \text{for all } (x,y)\in \ballt{\zs}{t}\Big\}. % \ballt{\zs}{t}\cap\YY\Big\}.
\end{equation}
Intuitively, $\delta$ measures how deeply $z^\star$ lies in $\cM$; larger $\delta$ means a wider margin by which the active inequalities are satisfied. In turn, once the iterates enter $\ballt{\zs}{\delta}$, they remain there and, after a few further steps, land on $\cM$. %Hence, the radius $\delta$ controls the speed at which identification happens; a larger ball yields faster identification.  %that the number of iterations needed to identify to manifold $\cM$ is equal to
\mybox{ We show that\,  $\left(\max\{1, \frac{1}{\alpha}\}\frac{8 \dist\left(z^0, \sol\right)}{\delta} \right)^{2} + \texttt{const} $ \,\, iterations suffice to identify $\cM$,}
%where $\alpha$ is the modulus in \eqref{eq:metric-subregularity-intro}
\noindent where $\alpha$ is the metric subregularity modulus in \eqref{eq:metric-subregularity-intro} and \texttt{const} is a small constant depending on algorithmic tuning. The quadratic term captures the time to reach the $\delta$-ball; the additive constant accounts for the final steps to lock onto $\cM$.

The third and final ingredient is a notion of `restricted' metric subregularity as opposed to the `global' notion. It is well known that when~\eqref{eq:metric-subregularity-intro} holds, convergence of first-order algorithms gets boosted from sublinear to linear; paralleling what happens with gradient descent on strongly-convex smooth function. However, the rate depends on the modulus $\alpha$, and when $\alpha$ is small, the linear rate can be impractically slow. Crucially, once identification kicks in,  it suffices to enforce metric subregularity in a neighborhood of the identifiable set: replace $\cD$ in \eqref{eq:metric-subregularity-intro} by $\cM$ intersected with a $\delta$-ball around the solution.  Let $\alpha_{\indexM}$ denote the corresponding restricted modulus; this constant governs the post-identification linear rate.
\mybox{After $\cM$ is identified, $\mathcal{O}\left(\frac{1}{\alpha_{\indexM}^2}\log\left(\frac{1}{\alpha_{\indexM}\varepsilon}\right)\right)$ iterations suffice to find a $z$ with $\dist(z, \sol) \leq \varepsilon.$}
\noindent Importantly, $\alpha_{\indexM}$ can be orders of magnitude larger than the global metric subregularity modulus and this convergence rate holds even when $\cM$ fails to be a manifold, thereby explaining the second-phase speed-up and yielding nonasymptotic guarantees without requiring strict complementarity.
 %Although a bit more technical, we also show that the constant $\alpha$ appearing in the identifiability bound can also be improved.
% In a nutshell, our work explains the nonasymptotic behavior of primal-dual first-order methods applied to~\eqref{eq:primalProblem-intro} without the need to impose strict complementarity.

\paragraph*{Outline.} The rest of this section is devoted to related work. Section~\ref{sec:prelims} briefly summarizes the necessary background. In Section~\ref{sec:assumptions}, we describe the problem and algorithmic classes that we consider, and in Section~\ref{section:algorithms}, we verify that several popular algorithms fall within the algorithmic class we study. Section~\ref{sec:guarantees} presents our general guarantees. Section~\ref{sec:experiments} closes the paper with numerical experiments supporting our theory. Lengthy, technical proofs are deferred to the appendix.

\subsection*{Related literature}

Our work is closely related to that of the last two authors \cite{lu2024geometry}, who studied the central question of this work for the particular case of PDHG applied to linear programs (LPs). Our answer builds on their ideas but requires substantial changes.  Two main obstacles to this generalization are: $(i)$ the LP analysis is based on polyhedral geometry, which is unavailable for general convex problems with functional constraints; and $(ii)$ \cite{lu2024geometry} leverages the explicit PDHG updates, whereas here we distill the basic algorithmic conditions that still ensure the two-stage behavior. The second obstacle necessitates more delicate arguments to treat algorithms such as ADMM, which lack several of the structural properties enjoyed by PDHG.

\paragraph{Convex-concave primal-dual algorithms.}
Convex-concave saddle-point problems and their associated primal-dual algorithms have been extensively studied for decades. The saddle-point problems are also studied as instances of general variational inequalities. In 1976, Rockafellar’s seminal work introduced the proximal point method (PPM)~\cite{rockafellar1976monotone} for solving monotone variational inequalities, while Korpelevich proposed the extragradient method (EGM) for solving convex–concave saddle-point problems~\cite{korpelevich1976extragradient}. In~\cite{nemirovski2004prox}, Nemirovski showed that EGM, as a mirror-prox instance, can be viewed as an approximation of PPM. For saddle-point problems with bilinear interaction terms, algorithmic development has been especially rich: the primal–dual hybrid gradient (PDHG) method~\cite{zhu2008efficient,chambolle2011first,chambolle2016ergodic} and the alternating direction method of multipliers (ADMM)~\cite{glowinski1975approximation,boyd2011distributed} are widely used in practice. More recently, it has been established that PDHG and ADMM can also be interpreted as approximations of PPM~\cite{he2012convergence,lu2023unified}.

\paragraph{Finite time identification.} Finite-time identification of active sets refers to an algorithm’s capability to identify the underlying manifold or active constraints in finite iterations~\cite{burke1988identification,wright1993identifiable, hare2004identifying}. This behavior is often analyzed under the framework of partial smoothness~\cite{hare2004identifying,lewis2013partial,lewis2022partial} and through the closely related $\mathcal{VU}$-decomposition perspective developed by~\cite{lemarechal2000lagrangian, mifflin2000mathcalvu}. Roughly speaking, a function is partly smooth relative to a manifold if it is smooth along the manifold while being sharply nonsmooth in directions transverse to it. This structure, combined with strict complementarity, enables characterizations of identification and the subsequent fast local convergence of algorithms. For example, manifold identification for dual averaging has been established in~\cite{lee2012manifold}, while forward–backward splitting methods have been shown to achieve finite-time identification under similar assumptions~\cite{liang2014local,liang2017activity}. In the context of primal-dual methods, finite-time active-set identification and local convergence of PDHG and ADMM are analyzed in~\cite{liang2017local,liang2018local, applegate2023faster, xiong2024accessible}. Recent work \cite{applegate2024infeasibility} showed that finite time identification and fast convergence also occur for infeasible linear programming problems, albeit with respect to an auxiliary feasible problem that characterizes the direction in which the iterates diverge. However, these guarantees hinge on {nondegeneracy} of the limiting solution---an assumption that is difficult to certify a priori and is frequently violated in applications~\cite{fadili2018sensitivity,fadili2019model}. 
To our knowledge, there are two notable exceptions.  First, the line of work initiated by \cite{wright2002modifying, wright2003constraint,oberlin2006active}, which develops modified sequential quadratic programming schemes designed to recover two-stage behavior even in the presence of degeneracy. In contrast, our analysis applies to standard first-order methods without any degeneracy-handling modifications. Second, the work  \cite{fadili2018sensitivity} develops a sensitivity and identification theory for (degenerate) mirror-stratifiable convex functions. Our work, instead, is not concerned with sensitivity and does not rely on mirror-stratifiability.

\paragraph{Metric subregularity and growth conditions.}
Metric subregularity imposes a linear error bound on the saddle subdifferential~\cite{ioffe1979necessary,dontchev2004regularity,dontchev2009implicit}, serving as a unifying regularity condition across diverse problem classes. Originally introduced in the early works of Robinson~\cite{robinson1981some}, metric subregularity is closely related to notions such as calmness and error bounds~\cite{henrion2002calmness,drusvyatskiy2018error}. Many structured problems, including those involving piecewise linear–quadratic models such as Lasso and support vector machines, naturally satisfy subregularity on compact domains~\cite{robinson1981some,zheng2014metric}. Motivated by these applications, recent research has studied how first-order methods behave under this assumption. In convex minimization, metric subregularity of the subdifferential has been shown to be equivalent to quadratic growth conditions, leading to linear convergence~\cite{drusvyatskiy2018error}. Similar developments extend to saddle-point and primal–dual settings, such as for PDHG and ADMM~\cite{yuan2020discerning,lu2022infimal}.

\section{Preliminaries}\label{sec:prelims}

In this section, we review the notation and necessary background in linear algebra and convex analysis. We defer the interested reader to the monographs \cite{borwein2006convex, rockafellar1997convex}. %, ryu2022large, beck2017first}.
%\paragraph{Linear Algebra.}
We use the symbols $\NN$ and $\RR$ to denote the set of natural (without zero) and real numbers, respectively. Further, we use $\bRR$ to denote $\RR \cup \{+\infty\}.$ We denote the set of nonnegative reals as $\RR_{+}.$ We endow $\RR^{n}$ with the standard dot product $\dotp{x,y} = x^{\top} y$ and its induced norm $\|x\|_{2}=\sqrt{\dotp{x, x}}$. The symbols $\ballt{ z}{t}$ and $\cballt{z}{t}$ denote the open and closed ball of radius $t$ centered at \(z\), respectively, and we label $\ball :=\ballt{0}{1}$ and $\cball := \cballt{0}{1}$. We use $\cSn$ to denote the set of $n \times n$ symmetric matrices. For a given $M \in \cSn$ we denote its eigenvalues by $\lambda_{1}(M) \geq \dots \geq \lambda_{n}(M)$. We write \(\lambdamax(M)\) for the largest eigenvalue of $M$ and \(\lambdaminp(M)\) for the smallest nonzero eigenvalue of $M$.
We use
$\cSnp=\{M\in\cSn:\lambda_{n}(M)\geq 0\}$ and $\cSnpp=\{M \in\cSn:\lambda_{n}(M)>0\}$
to denote positive semidefinite (PSD) and positive definite (PD) matrices, respectively. The symbol $\kappa(M)=\lambda_1(M)/\lambda_n(M)$ denotes the condition number of $M$.  The symbol $M^{\dagger}$ denotes the pseudoinverse of $M$.
% The symbols $\cSnp$ and $\cSnpp$ denote the set of positive semidefinite and positive definite matrices, repectively; i.e., $M \in \cSn$ such that $\lambda_{n}(M) \geq 0$ or with strict inequality. The set of matrices for which this inequality holds strictly is denoted by $\cSnpp.$
Any $M \in \cSnp$ induces a semi-inner product $\dotp{x, y}_{M} = x^{\top}M y$ and a semi-norm $\|x\|_{M} = \sqrt{\dotp{x, x}_{M}}$ that induces a pseudodistance. The symbol $\cball_t^M(z)$ denotes the set of points whose \(M\)-pseudodistance to \(z\) is less than or equal to \(t\). Similarly, the symbol $\ball_t^M(z)$ denotes the set of points whose \(M\)-pseudodistance to \(z\) is strictly less than \(t\). Abusing notation, we define the $M$-pseudo-distance from a point $x$ to a set $Q \subseteq \RR^n$ via
$$
\dist_{M}(x, Q) := \inf_{y \in Q} \| x - y \|_M. $$
When $M$ corresponds to the identity, we drop the subindex and simply write $\dist$. We define the Euclidean projection onto a closed set $Q$ as
$$\proj(x) := \argmin_{y \in Q} \| x - y \|_2. $$
% For an arbitrary matrix $M \in \RR^{n \times m},$ we use $\sigma_{1}(M)\geq \dots\geq\sigma_{\min\{m,n\}}(M)$ to denote its singular values.

%\paragraph{Convex Analysis.}
    % In this paper we will focus on convex, closed and proper (CPP) functions, as they are amenable to analysis and to algorithm implementation. A function \(f:\RR^n\to\R\cup\{+ \infty\}\) can be characterized as CPP by looking at its {\em epigraph}, which defined as
    % \[
        % \epi(f)=\{(x,t)\in\RR^{d+1}:f(x)\leq t\}
    % \]
% When the epigraph of \(f\) is closed, we say \(f\) is closed. When it is nonempty, we say \(f\) is proper. And when it is convex, we say \(f\) is convex.
Consider a function $f \colon \RR^{n} \to \bRR$. We let $\epi(f)=\{(x,t)\in\RR^{d+1}:f(x)\leq t\}$ be the epigraph of $f$. The function is proper if the epigraph is nonempty. Analogously, we say that $f$ is closed (resp. convex) if its epigraph is closed (resp. convex). Given a convex, closed set $Q \subseteq \RR^{n},$ we define its indicator function as \begin{equation}
  \iota_{Q}(x) := \begin{cases}0 & \text{if }x \in Q, \\
  +\infty & \text{otherwise.}\end{cases}
  \end{equation}
For a closed, convex, proper function $f\colon \RR^{d}\rightarrow \bRR$ and a point $x\in \RR^{n}$, the \emph{convex subdifferential} of $f$ at $x$, denoted by $\partial f(x)$, corresponds to the set of vectors $g$ satisfying
 \[
         f(z)\geq f(x)+\langle g,z-x\rangle \qquad \text{for all} \ z \in \RR^{m}.
      \]
    % A relation, set-valued mapping, multi-valued function or operator \(R\) on a set \(S\) is a subset of \(S\times S\). If \(R(x)\equiv \{y\in S:(x,y)\in R\}\) is a singleton or empty for all \(x\in S\), then \(R\) is {\em single-valued} and represents a function with domain \(\{x\in S:R(x)\neq\emptyset\}\).
    % A relation that is of special interest in convex analysis the {\em subdifferential} of a function \(f:\RR^n\to\R\cup\{+ \infty\}\):
    % \[
        % \partial f(x)=\{g\in R^{n}:\forall z\in \RR^n\, f(x)\geq f(x)+\langle g,z-x\rangle\}
    % \]
      % When \(f\) is CCP, \(\partial f(x)\) is a nonempty closed and convex set for any \(x\in\dom (f)\).
      Similarly, for a function \(\cL\colon\RR^{n}\times \RR^{m}\to\RR\cup\{+\infty\}\) which is convex in its first component and concave in its second component, the \emph{saddle subdifferential} at a point $(x, y) \in \RR^{n}\times \RR^{m}$ is given by
 \begin{equation}\label{eq:saddle-subdifferential}        \cF (x,y)=\begin{bmatrix}\partial_x\cL(x,y)\\\partial_y(-\cL(x,y))\end{bmatrix},
    \end{equation}
    where, for any fixed \(y\in\RR^m\) we use  \(\partial_x\cL(x,y)\) to denote the subdifferential of the function \(\cL(\cdot,y)\) at \(x\) and an analogous definition follows for $\partial_{y}$. %When \(\cL\) is CCCP \(F(z)\) is a nonempty closed and convex set for any \(z\in\RR^{n+m}\).

\section{Setting and algorithms}\label{sec:assumptions}

In this section, we formalize the setting we study, state assumptions, and present a general algorithmic template. Rather than focusing on a single method, we analyze a meta-algorithm satisfying mild conditions. At the end of this section, we show that many popular methods fit this template.

%In this section, we formally introduce the setting and assumptions we cover as well as a general templae of the algorithms we consider. Instead of studying a concrete method, we consider a meta algorithm satisfying certain assumptions, we shall see in Section~\ref{section:algorithms} that many popular methods fit our template.
\label{section:settings}
\subsection{Problem class}
\label{section:problem}
Our departing point is a pair of primal-dual problems of the form
% where \(f\colon\R^n\to\R\) is a convex function and $G\colon\R^n\to\R^m$ as a place holder for $G(\cdot) = \left(g_1(\cdot), \dots, g_m(\cdot)\right)^\top$ where $g_j\colon \RR^d \rightarrow \RR$ is a convex function for all $j \in [m]$. This formulation is equivalent to a pair of primal and dual optimization problems given by
\begin{equation}
\label{eq:primalProblem}
p^\star = \begin{cases}\displaystyle\min_{x\in\R^n} & f(x)\\
    \text{s.t.} &  g_j(x)\leq 0 \quad \text{for all }j \in \{1, \dots, m\}\, ,
    \end{cases} \qquad \text{and} \qquad     q^\star =\begin{cases}\displaystyle\max_{y\in\R^m} & h(y)\\
    \text{s.t.} & y\geq 0\, ,
    \end{cases}
\end{equation} 
where the functions $f$ and $g_{j}$ are assumed to be convex, and \(h(y) = \min_{x \in \RR^{n}} f(x) + \dotp{y, G(x)}\) with $G$ the map whose entries are given by $G(x)_{i} = g_{i}(x)$. This formulation subsumes linear and quadratic programming and extends well beyond these classes.
%While the dual problem is defined in terms of as
%\begin{equation}
%\label{eq:dualProblem}
%\end{equation}
%where \(h\colon\RR^m\to\R\) is given by $h(y) = \min_{x \in \RR^d} \cL(x,y)$. %For instance, defining \(h(y)=\min_{x\in\RR^n}\cL(x,y)\) we have that \(h=f_D+\iota_{C}\), where \(C=\{y\in\RR^m:G_D(y)\leq 0,\, y\geq 0\}\). 
% Furthermore, from the max-min formulation for the dual problem and the form of \(\cL\) we have the dual constraint \(y\geq 0\), from which we can assume without loss of generality that \((G_D(y))_k=-y_k\) for all \(k\in [m]\). 
%Our assumptions over the elements involved in the above problems are shown in Assumption~\ref{assumption:problem}.
Primal-dual optimal solutions \(\sol\) are given by the set of pairs \((x,y)\in\R^{n+m}\) satisfying
\begin{equation}
\label{eq:primalDualSolutions}
    \begin{aligned}
            f(x) -h(y) &\leq 0 && \text{(Zero duality gap)}\\
            G(x) &\leq 0 && \text{(Primal feasibility)}\\
            y&\geq 0 && \text{(Dual feasibility)}\,. 
    \end{aligned}
\end{equation}
As mentioned in the introduction~\eqref{eq:primalProblem} is equivalent to the minimax problem \eqref{eq:minimaxProblem}. For any \(z\in\RR^{n+m}\) let \(\sub(z)\subseteq \RR^{n+m}\) denote the saddle subdifferential~\eqref{eq:saddle-subdifferential} of the constrained Lagrangian function \(\cL(x,y)=f(x)-\iota_{\R^m_+}(y) + \dotp{y, G(x)}\).
% , that is
% \[
%     \sub(z)=\begin{bmatrix}
%         \partial f(x)+\displaystyle\sum_{j=1}^my_j\partial g_j(x)\\
%         -G(x)+N_{\RR_+^m}(y)
%     \end{bmatrix}
%     %\quad\text{where}\quad J_G(x)=\begin{bmatrix}\partial g_1(x) & \dots & \partial g_m(x)\end{bmatrix}
% \]
By the saddle point theorem for convex optimization~\cite{rockafellar1997convex}, when the functions \(f\) and \(g_j\) are convex we have \(\sub^{-1}(0)=\sol\). 

We will impose a standard set of assumptions on the primal-dual problem.

\begin{assumption}
\label{assumption:problem}
    Problem \eqref{eq:minimaxProblem} satisfies the following two conditions.
    \begin{enumerate}
        \item{\bf (Convexity)} The functions \(f\) and \(g_j\) are convex for all $j \in [m]$.
        \item {\bf (Existence of solutions)} The set of primal-dual solutions \(\sol\) is nonempty.
    \end{enumerate}
\end{assumption}
These conditions are standard. The existence of primal-dual solutions is equivalent to strong duality \(p^\star=d^\star\) with primal-dual attainment, which is implied by constraint qualification conditions such as the Slater condition.
\begin{assumption}[{\bf Metric sub-regularity}]\label{assumption:metric-subregularity}   For any radius \(t>0\) there exists \(\alpha>0\) such that
     \[
        \alpha\dist_2(z,\sol)\leq \dist_2(0,\sub(z))\qquad\text{for all } z\in\sol+ t \ball\,.
     \]
\end{assumption}

 The assumption asserts an error bound for the saddle subdifferential---known as metric subregularity in the variational analysis literature~\cite{ioffe2000metric, rockafellar1998variational}---that is intimately related to quadratic-growth behavior and the stability of solutions \cite{drusvyatskiy2013second, drusvyatskiy2013tilt}. The assumption holds for a broad class of primal-dual problems; in particular, all LP problems satisfy it~\cite{hoffman2003approximate, applegate2023faster}.
%Metric subregularity is tightly linked to error bounds \cite{hoffman2003approximate} and generalizes quadratic-growth-type conditions \cite{drusvyatskiy2013tilt,drusvyatskiy2013second}. It holds broadly; for instance, it holds for linear programs where \(\alpha\) can be bounded via the Hoffman constant \cite{applegate2023faster}.

%one is not standard but ubiquitous in convex optimization: it is equivalent to uniform local quadratic growth at the primal-dual solution set. 
%For instance, any quadratic programming problem satisfies the latter condition. Furthermore, in Appendix~\ref{appendix:proof-relaxedAssumption} we show that if \(\sol\) is bounded, then the uniformity of the condition can be relaxed: just point-wise metric sub-regularity over \(\sol\) suffices. 

\subsection{Algorithmic template}
    \label{section:metaAlgorithm}
    In this section, we introduce the meta-algorithm we analyze. To solve \eqref{eq:primalProblem}, the meta-algorithm maintains two sequences: the main iterates \(z^k=(x^k,y^k)\) and the auxiliary iterates \(\tilde z^k=(\tilde x^k,\tilde y^k)\). That is, the method is initialized at some \(z^0\)
    % \(z^0= (x^{0}, y^{0}) = \tilde z^{0}\)
   % (potentially equal) \(z^0\) and \(\tilde z^{0}\)
    and  updates
%we an iterative method starting from a fixed initial iterate \(z^0\) = \(x^0, y^0\). %We allow the algorithm to have intermediate iterates; we consider a generic class of primal-dual algorithms for primal-dual problems with iterate update
    \begin{equation}
    \label{eq:algorithm}
        z^{k+1},\tilde{z}^{k+1}\leftarrow \texttt{PrimalDualStep}(z^{k}). 
    \end{equation}
    In a nutshell, \(\tilde z^k\) is an intermediate update that we will use to describe our assumptions; yet for several algorithms it is trivially equal to $z^k$. We also suppose that each algorithm comes equipped with a PD matrix $P\in \cS_{++}^{n + m}$ and its associated norm $\|z\|_P:= \sqrt{z^\top P z}$, which dictates a natural geometry to measure progress of the algorithm. %To motivate our meta-algorithm and assumptions, the following section introduces several examples matching our template.
    
    %As an illustrative example, consider the primal-dual hybrid gradient method (PDHG) \cite{chambolle2011first}, which applied to \eqref{eq:minimaxProblem} with linear constraints --- i.e., $G(x) = Ax-b$ for a matrix $A\in\RR^{ m\times n}$ and a vector $b\in\RR^{m}$---reduces to
%\begin{equation}\label{eq:fake-pdhg}
%              x^{k+1}\leftarrow\prox_{\eta f}(x^k-\eta A^\top y^{k}), \quad  y^{k+1}\leftarrow \proj_{\RR^m_+}(y^k+\eta A\tilde{x}^{k}),  \quad\text{and}\quad \tilde{x}^{k+1}\leftarrow 2x^{k+1}-x^k,
 %   \end{equation}
 %\begin{equation}\label{eq:fake-pdhg}
 %             x^{k+1}\leftarrow\prox_{\eta f}(x^k-\eta A^\top y^{k}),  \quad \tilde{x}^{k+1}\leftarrow 2x^{k+1}-x^k,
%\quad\text{and}\quad  y^{k+1}\leftarrow \proj_{\RR^m_+}(y^k+\eta A\tilde{x}^{k+1}),     \end{equation}
 %   where $\eta > 0$ is the stepsize and we trivially let $\tilde y^{{k+1}} = y^{k+1}$. For this update, $\tilde x^k$ is simply an intermediate extrapolation update. In turn, when $\eta< \|A\|_{\rm{op}}^{-1} $, one can show \cite{lu2022infimal} that
 %   \begin{equation}
 %   \label{eq:pdhg-rate} 
 %   \|z^{k+1} - z^k\|_P \leq \frac{\dist_P(z^0, \cS^\star)}{\sqrt{k}}\qquad \text{with} \qquad P = \begin{pmatrix}
 %                   \frac{1}{\eta}I & -A^T\\
 %                   -A & \frac{1}{\eta} I
 %   \end{pmatrix} \, .
 % \end{equation}
            
Next, we introduce three assumptions on the meta-algorithm \eqref{eq:algorithm}. Assumption~\ref{assumption:asymptoticIdentification} is necessary for eventual identification; Assumption~\ref{assumption:rapidLocalConvergence} is required for local linear convergence; and Assumption~\ref{assumption:finiteTimeIdentification}, in tandem with the first two, delivers nonasymptotic identification. We shall see in Section~\ref{section:algorithms} that all of these assumptions hold for several popular algorithms. % such as the proximal point method(PPM), PDHG, the alternating direction method of multipliers (ADMM), the extragradient method (EG). All assumptions depends on the PSD matrix \(P\).
For intuition, the reader might take \(P=I\), in which case \(\|\cdot\|_{P}=\|\cdot\|_{2}\); this choice is realized by a number of methods.
     \begin{assumption}%[Basic identification]
     \label{assumption:asymptoticIdentification}
        There exists a positive definite matrix $P\in \cSnpp$ such that the following three hold.
        \begin{itemize}
            \item[$(i)$] (\textbf{Convergence}) For any initial iterate \(z^0\in\RR^{n+m}\) there exists \(\zs\in\sol\) such that 
            \[
                \max\left\{\|z^k-\zs\|_P,\|\tilde{z}^k-\zs\|_P\right\}\to 0 \quad \text{as} \quad k \rightarrow \infty\,.
            \] 
            %In particular, there exists \(R>0\) such that \(z^k,\tilde{z}^k\in\ballP{\zs}{R}\) for all \(k\in\NN_0\).
          \item[$(ii)$] (\textbf{Dual update}) %Denote \(z^{k+1}=(x^k,y^k)\), \(\tilde{z}^{k+1}=(\tilde{x}^k,\tilde{y}^k)\) as primal-dual variables. Then
                There exists \(\eta>0\) ({stepsize}) such that dual update takes the form
            \[
                y^{k+1}=\proj_{\RR^{m}_+}(y^k+\eta G(\tilde{x}^{k+1}))\,.
            \]
            %Further, we have $\tilde{z}^k \in \YY :=\RR^n\times\RR^m_+.$
            \item[\((iii)\)] ({\bf $P$-Lipschitzness}) For each \(t>0\) and $j \in [m]$ there exist $\Lxj, \Lyj > 0$ such that  %in $\ballP{\zs}{t}$ 
            $$
            g_j(x) - g_j(x') \leq \Lxj\|z-z'\|_P \qquad \text{and} \qquad |y_j - y_j'| \leq \Lyj \|z - z'\|_P
            $$ 
            for $z = (x, y), z' = (x', y') \in\ballP{\zs}{t}.$ 
            %\YY\cap\ballP{\zs}{t}.$
            \end{itemize}
     \end{assumption}
     %Assumption~\ref{assumption:asymptoticIdentification} is needed for asymptotic identification (Theorem~\ref{theorem:asymptoticIdentification}). 
     %\pil{Say why we mention the range of $P$ here too.} \MD{No, here let's just delete it. This Assumption is modified in the appendix, Assumption 4 isn't.}
     Let us comment on these conditions. The first condition essentially states that the algorithm converges to an optimal solution point-wisely. The second condition is ubiquitous in primal-dual algorithms tackling \eqref{eq:primalProblem} as we shall see in the following section. The third condition, although not standard, holds automatically provided the functions \(g_j\) are Lipschitz continuous since in Euclidean spaces all norms are equivalent. We state it in terms of the $P$-norm since it is a natural norm to state our guarantees. In the appendix, we state a weaker, more technical assumption that only requires $P$ to be positive semidefinite, which is crucial to cover ADMM, whose associated $P$ is singular. In what follows, $P$ denotes the same matrix as in Assumption~\ref{assumption:asymptoticIdentification}. We use the range of $P$ in the statement of the next assumption, which under Assumption~\ref{assumption:asymptoticIdentification} is trivially the whole space (since $P$ is positive definite); we keep this form to remain compatible with the semidefinite (singular) case handled in the appendix. % to analyze %Although $P$ is PD for most methods, for ADMM this matrix singular, and so to accomodate to this case, we adopt this slightly more delicate assumption. %, e.g., convex.
     %The following assumption is needed for linear convergence (Theorems~\ref{theorem:linearConvergence} and~\ref{theorem:linearConvergenceFast}). 
     \begin{assumption}%[Linear Convergence]
     \label{assumption:rapidLocalConvergence}
        The following two hold.
        \begin{itemize}
            \item[\((i)\)] (\textbf{Subdifferential sublinear rate}) There exists \(\gamma>0\) such that
            \begin{align*}
                \dist_{P^{\dagger}}(0,\sub(z^k)\cap\range(P)) \leq \frac{\gamma\dist_P(z^0,\sol)}{\sqrt{k}} \qquad \text{for any \(k\in\NN\)}\,.
            \end{align*}
             \item[\((ii)\)] (\textbf{Closedness to solutions}) There exists \(\tau>0\) such that \(\dist_2(z^k,\sol)\leq \tau\) for all \(k\in\NN\).
        \end{itemize}
     \end{assumption}

The first assumption states essentially that the subgradient $\sub(z^k)$ converges to $0$ at a sublinear rate, and the second assumption states the iterates stay in a bounded region from the optimal solution set. These conditions are common among first-order methods. In turn, a simple consequence of this assumption and metric subregularity is linear convergence. The next result formalizes this statement. This is a well-known guarantee, we include it for completeness; its proof appears in Appendix~\ref{appendix:proof-linearConvergence}.
     \begin{proposition}
         \label{theorem:linearConvergence}
          Suppose Assumptions~\ref{assumption:problem}, \ref{assumption:metric-subregularity}, and~\ref{assumption:rapidLocalConvergence} hold. Let $z^{k}$ be the \(k\)th iterate generated by update \eqref{eq:algorithm}. Then,
          \begin{equation}
    \label{eq:alphaG}
    0<\alpha_G:=\inf_{z\in\cD}\frac{\dist_2(0,\sub(z))}{\dist_2(z,\sol)} \qquad \text{with} \qquad 
    \cD=\sol+\tau\ball\,.
    %\cD=\left( \sol+\tau\ball \right)\cap\YY\,.
    \end{equation}
          Further, for any $k \in \NN$ we have
        \[
           \dist_2(z^k,\sol)\leq \sqrt{e}\nu \exp\left(-\rate \frac{k}{\left\lceil e\nu^{\irate}\right\rceil}\right)\dist_{2}(z^0,\sol)\quad\text{where}\quad \nu=\frac{\gamma\lambdamax(P)}{\alpha_G}.
        \]
        \end{proposition}
     The rate depends on the global metric subregularity modulus in \eqref{eq:alphaG}. In practice, this constant is often small, especially for badly conditioned problems, yielding impractically slow convergence. Thus, it is common to observe an first stage of active set identification; the next assumption is key in deriving bounds on this stage.
     
      %The following assumption, together with Assumptions~\ref{assumption:asymptoticIdentification},~\ref{assumption:rapidLocalConvergence}, is needed for finite time identification (Theorems~\ref{theorem:finiteTimeIdentification} and~\ref{theorem:finiteTimeIdentificationLinear}).
    \begin{assumption}%[Finite Time Identification]
    \label{assumption:finiteTimeIdentification} The following two conditions hold.
    \begin{itemize}
        \item[$(i)$] (\textbf{Sublinear rate}) There exists \(\gamma>0\) such that the following inequality holds:
        \begin{align*}
            \max\left\{\|z^{k+1}-z^{k}\|_P,\,\|z^k-\tilde{z}^k\|_P\right\} \leq \frac{\gamma\dist_P(z^0,\sol)}{\sqrt{k}}\qquad\text{for all}\quad k\in\NN_0\,.
        \end{align*}
        % \jnote{Is this $\gamma$ essentially the $\gamma$ in the previous assumption?}
        
        \item[$(ii)$] ({\bf Star non-expansiveness})
        % ({\bf Non-espansiveness towards solutions})
        For any $\zstilde\in\sol$ the following inequality holds
        \[
            \|\zs-\zstilde\|_P\leq \|z^{k}-\zstilde\|_P \qquad \text{for all}\quad  k\in\NN_0\, .
        \]
    \end{itemize}
    \end{assumption}
    The next section shows that all these requirements are satisfied by several algorithms. %As alluded to in \eqref{eq:pdhg-rate}, standard results \cite{lu2022infimal, ryu2022large} show that PDHG satisfies this assumption. Indeed, the sublinear decay of successive differences follows from the so-called firm nonexpansiveness of its update operator. The same argument extends to many other iterative schemes; in Section~\ref{section:algorithms} we will come back to more  examples.

\subsection{Instantiations of the meta-algorithm}%\pil{Instances?}}
\label{section:algorithms}
In this section, we show that four classic algorithms for solving minimax problems satisfy the assumptions defining the meta-algorithm described in Section~\ref{section:metaAlgorithm}. The proofs of all results in this section are deferred to Appendix~\ref{appendix:algorithms}.

\paragraph{Proximal Point Method (PPM).} Fix $\eta > 0,$ PPM~\cite{rockafellar1976monotone} solves~\eqref{eq:minimaxProblem} by iteratively updating
\begin{equation}\label{alg:PPM}
    \displaystyle(x^{k+1},y^{k+1})\leftarrow \arg\min_{x\in\RR^n}\max_{y\in\RR^m_+}f(x)+\langle y,G(x)\rangle + \frac{1}{2\eta}\|x-x^k\|_2^2-\frac{1}{2\eta}\|y-y^k\|_2^2 \ .
\end{equation}
For this method, we trivially take $\tilde z_k = z_k$. We highlight that solving \eqref{alg:PPM} might be just as hard as solving the original problem~\eqref{eq:primalProblem}, and so, in most situations, this is not a practical algorithm. However, it serves as a clean canonical baseline for our framework. The next result shows that PPM satisfies all our assumptions; we defer its proof to Appendix~\ref{appendix:PPM}.
\begin{proposition}
\label{prop:ppm}
     Fix any stepsize $\eta > 0$ and set the auxiliary iterates to \(\tilde{z}^k=z^k\). Then, PPM satisfies Assumptions~\ref{assumption:asymptoticIdentification},~\ref{assumption:rapidLocalConvergence}, and~\ref{assumption:finiteTimeIdentification} with $P=\eta^{-1}I$ and $\gamma=1$.
\end{proposition}

\paragraph{Primal Dual Hybrid Gradient (PDHG).} Assume the constraints in~\eqref{eq:primalProblem} are affine, namely $G(x)=Ax - b$ for $A\in\RR^{m\times n}$ and $b\in\RR^m$. Fix $\eta>0$,  PDHG~\cite{chambolle2011first} solves the corresponding minimax problems \eqref{eq:minimaxProblem} by iteratively updating
\begin{equation*}
    \begin{aligned}
        & x^{k+1}\leftarrow\prox_{\eta f}(x^k-\eta A^\top y^{k})\\
        & \tilde{x}^{k+1}\leftarrow 2x^{k+1}-x^k\\
        & y^{k+1}\leftarrow \proj_{\RR^m_+}(y^k+\eta A\tilde{x}^{k+1})\ . %\\
       % & \tilde{x}^{k+1} \leftarrow y^{k+1}\ .
    \end{aligned}
\end{equation*}
PDHG is used extensively for inverse problems arising in imaging \cite{benning2018modern, fan2019brief} and large-scale linear programming \cite{applegate2021practical, applegate2025pdlp}. 
The next proposition shows that PDHG satisfies all our assumptions; the proof is deferred to Appendix~\ref{appendix:pdhg}.
\begin{proposition}
\label{prop:pdhg}
Fix a stepsize $\eta > 0$ satisfying \(\eta<\|A\|_{\rm{op}}^{-1}\) and set the dual auxiliary iterates to be \(\tilde{y}^k=y^{k}\). Then, PDHG satisfies Assumptions~\ref{assumption:asymptoticIdentification},~\ref{assumption:rapidLocalConvergence}, and~\ref{assumption:finiteTimeIdentification} with $P=\begin{bmatrix}
        \frac{1}{\eta}I & -A^\top\\
        -A & \frac{1}{\eta} I
    \end{bmatrix}$ and $\gamma=1$.
\end{proposition}

\paragraph{Alternating Direction Method of Multipliers (ADMM).} As with PDHG, suppose that the constraints are affine \(G(x)=Ax-b\), \(A\in\RR^{m\times n}\) and \(b\in\RR^m\). Fix $\eta >0$, ADMM~\cite{glowinski1975approximation,boyd2011distributed} solves \eqref{eq:minimaxProblem} by iteratively updating
\begin{equation*}
    \begin{aligned}
        & u^{k+1}\leftarrow \argmin_{u\in\R^m}\left(\iota_{\RR^{m}_+}(u)+\langle y^k,u\rangle+\frac{\eta}{2}\|Ax^{k}+u-b\|^2)\right)\\
        & y^{k+1}\leftarrow y^{k}+\eta(Ax^{k}-b+u^{k+1})\\
        & x^{k+1}\leftarrow \argmin_{x\in\R^n}\left(f(x)+\langle y^{k+1},Ax\rangle +\frac{\eta}{2}\|Ax+u^{k+1}-b\|^2\right) \ .
    \end{aligned}
\end{equation*}
ADMM underpins widely used quadratic and convex programming solvers such as OSQP and SCS~\cite{stellato2020osqp, ocpb:16}. The analysis of ADMM is more subtle. Indeed, unlike the other algorithms we consider, ADMM does not satisfy %satisfy $P$-Lipschitzness (i.e., part (iii) of 
Assumption \ref{assumption:asymptoticIdentification} with a strictly positive definite $P$. Nevertheless, it does satisfy a weaker version of it with $P$ positive semidefinite, namely Assumption~\ref{assumption:asymptoticIdentificationWeak}. %, which we include in . 
As alluded to before, this version is more technical, and so we deferred it to the appendix. Nonetheless, all the convergence guarantees we establish hold with either assumption. The next proposition shows that ADMM satisfies this slightly modified set of assumptions; the proof is deferred to Appendix~\ref{appendix:ADMM}.%, with the difference that the Lipschitz bounds are now required only for the indices \(\indexN\). 
%See Appendix~\ref{appendix:algorithms} for more detail on Assumption~\ref{assumption:asymptoticIdentificationWeak}. 
\begin{proposition}
\label{prop:admm}
Fix any stepsize $\eta> 0$ and set the iterates of ADMM to be $z^k= \tilde{z}^k = (x^k,y^k)$. Assume there exists \(\tau>0\) such that \(\dist_2(z^k,\sol)\leq \tau\) for all \(k\in\NN\). Then, ADMM satisfies Assumptions~\ref{assumption:asymptoticIdentificationWeak},~\ref{assumption:rapidLocalConvergence}, and~\ref{assumption:finiteTimeIdentification} with $P=\begin{bmatrix}
\eta A^\top A & A^\top\\
A & \frac{1}{\eta}I
\end{bmatrix}$ and $\gamma=1$.
\end{proposition}
This proposition supposes that Assumption~\ref{assumption:rapidLocalConvergence} \((ii)\) holds true. This is the case, for instance, when \(A\) has full rank or when the solution set \(\sol\) is bounded~\cite{boyd2011distributed}.%\pil{This result is not explicit in the previous reference, but follows from a simple argument using the results of the reference}. The authors are unaware if this assumption holds in general.

\paragraph{Extragradient Method (EGM).} Assume the Lagrangian function \(\bar{\cL}(x,y)=f(x)+\langle y, G(x)\rangle\), with \(f\) and \(G\) as in \eqref{eq:primalProblem-intro}, is $L$-smooth, i.e., differentiable with $L$-Lipschitz gradients. Fix $\eta>0$, EGM~\cite{korpelevich1976extragradient} solves~\eqref{eq:minimaxProblem} by iteratively updating
\begin{equation}
    \begin{aligned}
        & \tilde{x}^{k+1}\leftarrow x^k-\eta(\nabla f(x^k)+J_G(x^k)^\top y^k)\\
        & \tilde{y}^{k+1}\leftarrow \proj_{\R_+^m}(y^k+\eta G(x^k))\\
        & x^{k+1}\leftarrow x^k-\eta(\nabla  f(\tilde{x}^{k+1})+J_G(\tilde{x}^{k+1})^\top\tilde{y}^{k+1})\\
        & y^{k+1}\leftarrow \proj_{\R_+^m}(y^k+\eta G(\tilde{x}^{k+1})) \ ,
    \end{aligned}
\end{equation}
EGM adds an extrapolation step to the vanilla gradient descent-ascent method to ensure convergence. In turn, it can be EGM interpreted as an approximation of PPM~\cite{nemirovski2004prox}. EGM is used across a broad range of modern minimax and saddle-point applications, including games, machine learning, and imaging \cite{quoc2008extragradient, 
chavdarova2019reducing,lou2025accurate}. The next proposition shows that EGM satisfies all our assumption; the proof is deferred to Appendix~\ref{appendix:EGM}.

\begin{proposition}
\label{prop:EGM}
     Fix $\eta < 1/L$. Then, EGM satisfies Assumptions~\ref{assumption:asymptoticIdentification},~\ref{assumption:rapidLocalConvergence}, and~\ref{assumption:finiteTimeIdentification} with $P=I$ and $\gamma=\frac{3}{\sqrt{1-(\eta L)^2}}$.
\end{proposition}

\if 0  
- Guarantees 
    - Finite identification
    - Local rapid convergence
\fi

\section{Guarantees}
\label{sec:guarantees}
In this section, we present our main theoretical results. Section~\ref{section:manifoldIdentification} provides our identification guarantees. Section~\ref{section:localRapidConvergence} shows that the local geometry of the problem around the limit solution is better conditioned than the full problem, which yield faster linear convergence. Section~\ref{section:constantComparison} compares the metric subregularity moduli associated to the global and the local problems.
\subsection{Finite time identification}
\label{section:manifoldIdentification}
Next, we state our finite-time identification results. We express these results in terms of the $P$-norm $\|z\|_P = \sqrt{z^\top P z}$, which differs slightly from our narrative in the introduction where for simplicity we used the Euclidean norm, i.e., $P = I$. %requiring us to slightly modify some of the notions from the introduction, which used the Euclidean norm. 
To start, recall that the set we identify depends on the active set at the solution we converge to \(\zs\). In particular, it depends on the index sets
% \begin{align*}
%     &\indexN=\{j\in[m]:g_j(\xs)<0\},\quad
%     % \indexB=\{j\in[m]:g_j(\xs)=0\}\,.
%     \indexBs=\{j\in[m]:g_j(\xs)=0,\,\ys_j>0\}, \quad \text{ and}\\
%     & \hspace{3cm}\indexBw=\{j\in[m]:g_j(\xs)=0,\,\ys_j=0\}\,.
% \end{align*}
\begin{equation}
\label{eq:indexSets}
    \begin{aligned}
        \indexN&=\{j\in[m]:g_j(\xs)<0\},\\%\quad
        % \indexB=\{j\in[m]:g_j(\xs)=0\}\,.
      \indexBs&=\{j\in[m]:g_j(\xs)=0,\,\ys_j>0\}, %\quad
                \text{ and}\\
         \indexBw&=\{j\in[m]:g_j(\xs)=0,\,\ys_j=0\}\,.
    \end{aligned}
\end{equation}
We call \(\indexN\) the set of {\em nonactive} indices, \(\indexBs\) the set of {\em active} indices and \(\indexBw\) the set of {\em degenerate} indices.\footnote{The indices in $\indexBs$ and $\indexBw$ also go under the names of strongly and weakly active in the literature \cite{oberlin2006active}.} Further, we use the placeholder $\indexB = \indexBs \cup \indexBw$. We say the solution \(\zs\) is {\em degenerate} if it does not satisfy strict complementarity, that is, if \(\indexBw\neq\emptyset\). %We remark that our theory is indifferent to whether the solution \(\zs\) is degenerate or not.

\begin{figure}[t]
    \centering
\includegraphics[width=0.9\linewidth]{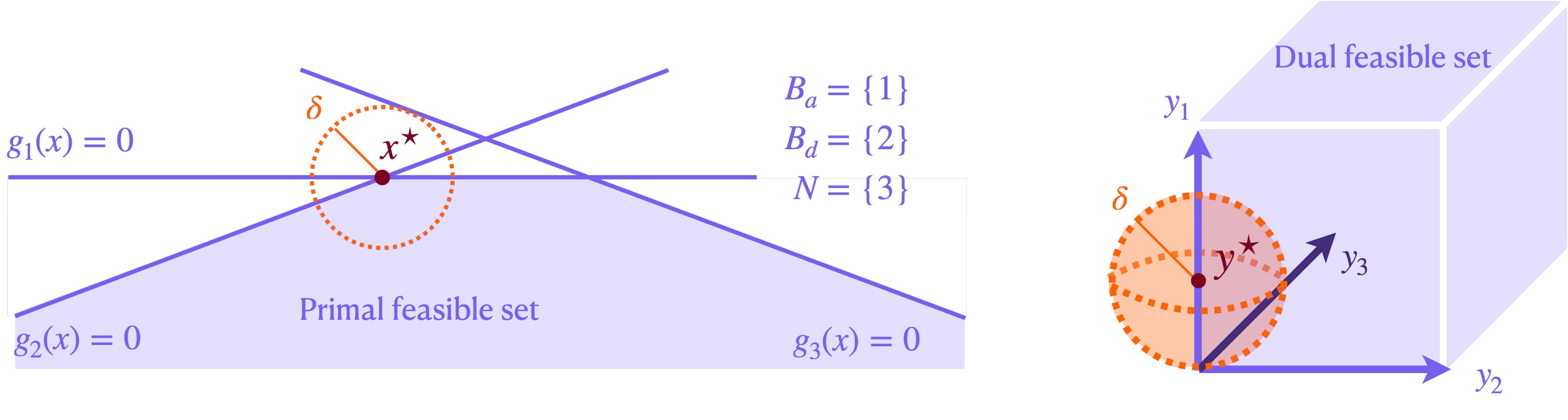}
\vspace{-8pt}
    \caption{Illustration of the radius of active-set stability \eqref{eq:degeneracy-metric}.}
    \label{fig:radius}
\end{figure}
\noindent With this index partition, we define the identifiable set as
\begin{equation}
\label{eq:identifiableSet}
    \cM=\Big\{(x,y)\in\R^n\times\R^m_+%\YY
    \mid G(x)_\indexN<0,\,y_{\indexN}= 0, \,\,\text{and}\,\, y_{\indexBs} > 0\Big\}.
    %&=\bigcup_{I\in 2^{\indexB}}\{(x,y)\in\R^{n+m}:\,G(x)_\indexN<0,\,y_{\indexN\cup I}=0,\, y_{\indexB\backslash I}>0\}
\end{equation}
% Recall $\YY = \RR^{n}\times \RR^{m}_{+}$. 
In light of Assumption~\ref{assumption:asymptoticIdentification}, this is the effective domain of our algorithms. The weakly active indices do not otherwise enter the definition; for $(x,y)\in\cM$ they are only required to satisfy $y_{\indexBw}\ge 0$. Consequently, if $\indexBw\neq\emptyset$, the set $\cM$ need not be a manifold---it has a border along ${y_i=0:i\in\indexBw}$.
% Note that, for any \((x,y)\in\cM\),
% \[
%     \indexN\subseteq \{j\in[m]: y_j=0\}\subseteq\indexN\cup\indexBw\quad\text{and}\quad  \indexBs\subseteq \{j\in[m]: y_j>0\}\subseteq\indexBs\cup\indexBw\,.
% \]
% Thus, identifying the set \(\cM\) implies identifying \(\indexN^c=\indexBs\cup \indexBw\) and \(\indexBs^c= \indexN\cup\indexBw\). This means identifying the sets \(\indexN\) and \(\indexBs\) up to the uncertainty given by the degenerate indices \(\indexBw\).
% To determine a sufficient condition for points belonging to the active set \(\cM\),
We define the {\em radius of active-set stability} as the size of the smallest perturbation of $\zs$, with respect to the \(P\)-seminorm, that violates one  constraints \(G_\indexN(x)<0\) and $y_{\indexBs} > 0$; formally
\begin{equation}
\label{eq:degeneracy-metric}
    \delta=\sup \Big\{t\in(0,\infty)\,\,\big|\, \,\text{For all } (x,y)\in\ballP{\zs}{t}%\ballP{\zs}{t}\cap\YY 
    \text{ we have }\,G_\indexN(x) < 0, \text{ and } y_{\indexBs} > 0\Big\}\,,
\end{equation}
% As shown in Appendix~\ref{appendix:manifoldIdentification}, under Assumption~\ref{assumption:asymptoticIdentification}, the magnitude \(\delta\) is non zero.
The  radius of  active-set stability quantifies how internal the point $\zs$ is with respect to the set $\cM$,  modulo the requirement $y_{N} = 0.$ Indeed, we have \(\ballP{\zs}{\delta}\cap \YN\subseteq \cM\) (Proposition~\ref{prop:manifoldInclusion} in Appendix~\ref{appendix:proof-asymptoticIdentification}), where
%\(
\begin{equation}
\label{eq:YN}
    \YN=\{(x,y)\in\R^n\times\R^m_+%\YY
    : y_\indexN=0\}\,.%\pil{\text{I changed this definition to an equation to be able to reference it}}
\end{equation}
%\)
% how close to degeneracy the non-degenerate constraints are: smaller values imply them being closer to degeneracy.  Furthermore, as shown in Appendix~\ref{appendix:manifoldIdentification}, for any \(\gamma<\delta\) we have \(\ballP{\zs}{\gamma}\cap \YN\subseteq \cM\), where
% \[
    % \YN=\{(x,y)\in\YY: y_\indexN=0\}\,.
% \]
% This inclusion is crucial for our argument; it gives a sufficient condition for belonging to \(\cM\).
% \begin{equation}
% \label{eq:degenerancyMetric}
%     \delta=\sup\{t\in(0,\infty)\,|\,\forall (x,y)\in\ballP{\zs}{t}\cap\YY:\,G_\indexN(x)\leq 0\}\,,
% \end{equation}
% \begin{wrapfigure}{r}{0.5\textwidth}
%   \vspace{-0.9cm}
%   \begin{center}
%     \includegraphics[width=0.5\textwidth]{sections/distancedegeneracy.png}
%   \end{center}
%   \vspace{-0.2cm}
%   \caption{Depection of \(\delta\) (dual not shown).}
%   \label{fig:distanceToDegeneracy}
% \end{wrapfigure}
%  As shown in Appendix~\ref{appendix:manifoldIdentification}, under Assumption~\ref{assumption:asymptoticIdentification}, the magnitude \(\delta\) is non zero. The {\em degeneracy radius} \(\delta\) quantifies how close to degeneracy the non-degenerate constraints are: smaller values imply them being closer to degeneracy.  Furthermore, as shown in Appendix~\ref{appendix:manifoldIdentification}, for any \(\gamma<\delta\) we have \(\ballP{\zs}{\gamma}\cap \YN\subseteq \cM\), where
% \[
%     \YN=\{(x,y)\in\YY: y_\indexN=0\}\,.
% \]
% This inclusion is crucial for our argument; it gives a sufficient condition for belonging to \(\cM\).
The following result shows that converging primal-dual algorithms with a projected gradient ascent dual update eventually identify this intersection. We defer its proof to Appendix~\ref{appendix:proof-asymptoticIdentification}.

\begin{theorem}
    \label{theorem:asymptoticIdentification}
    Suppose Assumptions~\ref{assumption:problem} and~\ref{assumption:asymptoticIdentification} hold. Let $z^k$ be the \(k\)th iterate generated by update \eqref{eq:algorithm}. Then, there exists \(K\in\NN\) such that \(z^k\in \ballP{\zs}{\delta}\cap\YN\subseteq\cM\) for all \(k\geq K\).
\end{theorem}
We note that we only require two basic assumptions for this asymptotic result. Variations of this result have appeared in the literature before \cite{burke1988identification, wright1993identifiable, hare2007identifying}, typically without an explicit quantification of the identification neighborhood. Making this radius explicit is the starting point for nonasymptotic rates. Combining our characterization of $\delta$ with metric subregularity yields explicit rates of convergence towards the ball $\ballP{\zs}{\delta/2}$. Once inside the neighborhood, the iterates reach $\YN$ after a small number of iterations, depending only on the stepsize and Lipschitz modulus.

\begin{theorem}
    \label{theorem:finiteTimeIdentificationLinear}
     %\jnote{(assigned to JY) double check this theorem and its proof} 
     Suppose Assumptions~\ref{assumption:problem}, ~\ref{assumption:asymptoticIdentification},~\ref{assumption:rapidLocalConvergence}, and~\ref{assumption:finiteTimeIdentification} hold and that the \(k\)-th iterate \(z^k\) and the \(k\)-th intermediate iterate \(\tilde{z}^k\) of the meta-algorithm defined in~\eqref{eq:algorithm} are equal. Then, \(z^k\in\ballP{\zs}{\delta/2}\cap\YN\subseteq\cM\) provided
    \[
        k > K:=\left\lceil e\left(\frac{\gamma\lambdamax(P)}{\alpha_G}\right)^{\irate}\right\rceil\left[1+{\irate}\ln\left(\frac{4\lambdamax(P)^{\frac{3}{2}}\dist_2(z^0,\sol)}{\alpha_G\delta}\right)\right]+\left\lceil\max_{j\in\indexN}\frac{\Lyjd}{\eta \overline\Lxjd}\right\rceil. %\quad\text{where}\quad \rho=.
    \]
  %  where, importantly, $\Lxj, \Lyj$ are the smallest constants for which \eqref{eq:lips} holds.
     % \(\rate\), \(\step\), \(\gamma\) and \(P\) are the rate of convergence, the extra steps required for dual Lipschitzness, the sublinear constant, and the positive semidefinite matrix defined in Assumptions~\ref{assumption:asymptoticIdentification} and ~\ref{assumption:rapidLocalConvergence}; \(\alpha_G\) is the metric-subregularity constant defined in~\eqref{eq:alphaG}; \(\sol\) is the set of primal-dual solutions defined in~\eqref{eq:primalDualSolutions}; \(\delta\) is the {\em distance to degenerancy} defined in~\eqref{eq:degenerancyMetric}; \(\eta\) is the stepsize defined in~\eqref{eq:algorithm}; and \(\bar{L}_j\) and \(\bar{\sigma}_j\) are optimal constants satisfying the \(P\)-Lipschitz regularity bounds of Assumption~\ref{assumption:asymptoticIdentification}.
\end{theorem}
The constant $\overline \Lxjd$ in the definition of $K$ is the smallest Lipschitz constant of $g_j$ on the ball $\ballP{\zs}{\delta}$.\footnote{Formally, $\overline \Lxjd := \sup_{z \in \ballP{\zs}{\delta} \setminus \zs} \tfrac{g_j(x) - g_j(x^\star)}{\|z - \zs\|_P}$. If \(g_j\) is constant in \(\ballP{\zs}{\delta}\), then \(\overline\Lxjd=0\), whence the bound in Theorem~\ref{theorem:finiteTimeIdentificationLinear} becomes unrealizable. In that case, as generalized in Theorem~\ref{theorem:finiteTimeIdentificationG}, \(\overline\Lxjd\) can be replaced by \(-2g_j(\xs)/\delta\).} The constant $\alpha_{G}$ corresponds with the `global' metric subregularity modulus in \eqref{eq:alphaG}. The proof parallels that of Theorem~\ref{theorem:linearConvergence} and is deferred to Appendix~\ref{appendix:proof-finiteTimeIdentification}. The first summand in the definition of $K$ corresponds to the time to reach the $\delta$ ball around $\zs,$ the second summand bounds the number of additional iterations required to reach $\YN.$ The latter depends on the choice of stepsize tuning. 

As noted in the paragraph after \eqref{eq:alphaG}, for badly conditioned problems the constant $\alpha_{G}$ can be small---we will see an explicit example in Section~\ref{section:constantComparison}---
so the resulting bound may be overly conservative.
% \begin{theorem}
To derive a more meaningful bound, we use a completely different approach that does not depend on the global constant \(\alpha_G\), but rather on a {\em local} metric subregularity modulus given by
\begin{equation}
    \label{eq:alphaL}
 \alpha_L=\inf_{z\in\cD}\frac{\dist_2(0,\sub(z))}{\dist_2(z,\sols)}\,,
\end{equation}
where \(\sols\) is defined as the set of primal-dual solutions to the following reduced system of equations
\begin{equation}
\label{eq:reducedSystemIdentification}
    \begin{aligned}
            f(x) -h(y) \leq 0,\quad
            %G(x)_{\indexBs \cup \indexBw} 
            G(x)_{\indexB} \leq 0,\quad\text{and}\quad
            y\geq 0\,. 
    \end{aligned}
\end{equation}
The set of solutions $\sol$ of the original problem solves a bigger system of equations \eqref{eq:primalDualSolutions}. Hence, \(\sol \subseteq \sols\) and comparing the two definitions, we derive \(\alpha_L \geq \alpha_{G}.\) With it we obtain the following.
\begin{THM}[label=theorem:finiteTimeIdentification]{Finite time identification}{}
     %\label{}
      Suppose          Assumptions~\ref{assumption:problem},~\ref{assumption:asymptoticIdentification},~\ref{assumption:rapidLocalConvergence}, and~\ref{assumption:finiteTimeIdentification} hold. Let $z^k$ be the \(k\)th iterate generated by update \eqref{eq:algorithm}. Then, \(z^k\in\ballP{\zs}{\delta/2}\cap\YN\subseteq\cM\) for all
     \[
        k > K:=\left(\max\left\{1,\frac{1}{\alpha_L}\right\}\frac{8\lambdamax(P)^{\frac{3}{2}}\dist_2(z^0,\sol)}{\delta}\right)^{\irate}+\left\lceil \max_{j\in\indexN}\frac{\Lyjd}{\eta \overline \Lxjd} \right\rceil\,.
    \]
     % where \(\rate\), \(\step\), \(\gamma\) and \(P\) are the rate of convergence, the extra steps required for dual Lipschitzness, the sublinear constant, and the positive semidefinite matrix defined in Assumptions~\ref{assumption:asymptoticIdentification},~\ref{assumption:rapidLocalConvergence}, and~\ref{assumption:finiteTimeIdentification}; \(\alpha_L\) is the constant defined in ~\eqref{eq:alphaL}; \(\sol\) is the set of primal-dual solutions defined in~\eqref{eq:primalDualSolutions}; \(\delta\) is the {\em distance to degenerancy} defined in~\eqref{eq:degenerancyMetric}; \(\eta\) is the stepsize defined in~\eqref{eq:algorithm}; and \(\bar{L}_j\) and \(\bar{\sigma}_j\) are optimal constants satisfying the \(P\)-Lipschitz regularity bounds of Assumption~\ref{assumption:asymptoticIdentification}.
\end{THM}
A couple of remarks are in order. At first sight, it might appear that the bound in Theorem~\ref{theorem:finiteTimeIdentificationLinear} gives a fast linear convergence in terms of $\delta$, while Theorem~\ref{theorem:finiteTimeIdentification} only yields sublinear convergence. However, the linear convergence rate relies on the conservative global constant $\alpha_G$. In contrast, the sublinear convergence is dependent on the local metric subregularity modulus $\alpha_L$, and can be more informative than the global constant $\alpha_G$. 
% We have included an example in Appendix~\ref{appendix:example} illustrating when the bound of Theorem~\ref{theorem:finiteTimeIdentification} becomes sharper than that of Theorem~\ref{theorem:finiteTimeIdentificationLinear} in terms of the geometry of the problem. 
Secondly, we would like to comment that the additive term $\left\lceil \max_{j\in\indexN}\frac{\Lyjd}{\eta \overline \Lxjd} \right\rceil$ can be viewed as a rather ``small'' constant term that does not affect much of the order of rate. The goal for this term is to guarantee that the iterates identify the set $\cM$ after reaching the ball \(\ballP{\zs}{\delta/2}\).

\subsection{Local rapid convergence}
    \label{section:localRapidConvergence}
    
    %Thanks to Theorems~\ref{theorem:asymptoticIdentification}, ~\ref{theorem:finiteTimeIdentification}, and ~\ref{theorem:finiteTimeIdentificationLinear}, the iterates of 
    So far we have established that the meta algorithm \eqref{eq:algorithm} identifies the union of manifolds \(\cM\) after enough iterations. After which, the algorithm effectively solves the primal dual problem restricted to \(\cM\). This restricted problem is better conditioned than the global problem, as it eliminates the nonactive constraints at the limit solution. In turn, this speeds up the convergence from sublinear to linear. We quantify this phenomenon via a local metric subregularity modulus
    \begin{equation}
    \label{eq:alphaM}
        \alpha_{\indexM} = \inf_{z\in\cDl\cap\cM}\frac{\dist_2(0,\sub(z))}{\dist_2(z,\sol)}\quad\text{where}\quad\cDl=\cD\cap\ballP{\zs}{{\delta}/{2}}\,.
    \end{equation}
    %If \(\theta=1\), we just write \(\alpha_{\indexM}\). Theorem~\ref{theorem:linearConvergenceFast} shows we obtain \(\alpha_{\cM_\theta}\)-rapid linear convergence after identifying \(\cD_\theta\cap\cM\). 
    
    The following proposition demonstrates the faster local convergence after identification. More formally, it states that, suppose after $K$ iterations, all iterates $z^k$ for $k>K$ stay in the union of manifold $M$ and they are not too far away from $z^*$, then the iterates enjoy a faster local linear convergence rate to the optimal solution set that only relies on the local sharpness constant $\alpha_{\indexM}$ instead of the global sharpness constant $\alpha_G$.

\begin{proposition}\label{theorem:linearConvergenceFast}
    % \begin{theorem}
     %\begin{THM}[label=theorem:linearConvergenceFast]{Fast convergence after identfication}{}
          Suppose Assumptions~\ref{assumption:problem}, \ref{assumption:metric-subregularity} and~\ref{assumption:rapidLocalConvergence}. Let \(z^k\) be the \(k\)th iterate generated by the update~\eqref{eq:algorithm}. Further suppose that there exists $K > 0$ such that for all $k > K$ we have \(z^k\in\cD_\delta\cap\cM\). Then,
        \[
          \dist_2(z^{k + K},\sol)\leq \sqrt{e}\nu_{\indexM}\exp\left({-\rate\frac{k}{\left\lceil e\nu_{\indexM}^{\irate}\right\rceil}}\right)\dist_{2}(z^K,\sol)\quad\text{with}\quad \nu_{\indexM}=\frac{\gamma\lambdamax(P)}{\alpha_{\indexM}}\,.
        % \dist_2(z^{k + K},\sol)\leq \exp\left( \rate k \right)\nu_{\indexM}\dist_{2}(z^K,\sol)\quad\text{where}\quad \nu_{\indexM}=\frac{\gamma\lambdamax(P)}{\alpha_{\indexM}}\,.
        \]
   %     \(\rate\), \(\gamma\), and \(P\) are the rate of convergence, the sublinear constant, and the positive semidefinite matrix defined in Assumption~\ref{assumption:rapidLocalConvergence}; \(\alpha_{\cM_\theta}\) is the metric-subregularity constant defined in~\eqref{eq:alphaM}; \(\sol\) is the set of primal-dual solutions defined in~\eqref{eq:primalDualSolutions}.
    %\end{THM}
\end{proposition}
Proposition \ref{theorem:linearConvergenceFast} follows with the same proof as that of Proposition~\ref{theorem:linearConvergence} by replacing $\nu$ with $\nu_{\indexM}$, and is therefore omitted.

    Putting together the finite time identification (Theorem~\ref{theorem:finiteTimeIdentification}) and the faster linear convergence after identification (Proposition~\ref{theorem:linearConvergenceFast}), we derive the following result, which presents the full characterization of the two-stage convergence behavior of primal-dual algorithms. Its proof is deferred to Appendix~\ref{appendix:proof-twoStageConvergence}.

    % \begin{corollary}
     \begin{THM}[label=theorem:twoStageConvergence]{Two-stage convergence rates}{}
    Suppose Assumptions~\ref{assumption:problem},~\ref{assumption:metric-subregularity},~\ref{assumption:asymptoticIdentification},~\ref{assumption:rapidLocalConvergence},~\ref{assumption:finiteTimeIdentification} hold and fix \(\varepsilon>0\). Let $z^k$ be the \(k\)th iterate generated by update \eqref{eq:algorithm}. Then, we have \(\dist_2(z^k,\sol)\leq \varepsilon\) provided that
    \[
        k>\underbrace{\left[\max\left\{1,\frac{1}{\alpha_L}\right\}\frac{8\lambdamax(P)^{\frac{3}{2}}\dist_2(z^0,\sol)}{\delta}\right]^{\irate}+\left\lceil\max_{j\in\indexN}\frac{\Lyjd}{\eta \overline{\Lxjd}} \right\rceil}_{\text{Finite time identification}}+ \underbrace{\rho_{\indexM}\left[1+\irate\ln\left[\frac{\gamma\lambdamax(P)^{\frac{1}{2}}\delta}{2\alpha_{\indexM}\varepsilon}\right]\right]}_{\text{Fast linear convergence}}\,,
    \]
    %\(\rate\), \(\step\), \(\gamma\) and \(P\) are the rate of convergence, the extra steps required for dual Lipschitzness, the sublinear constant and the positive definite matrix defined in Assumptions~\ref{assumption:asymptoticIdentification},~\ref{assumption:rapidLocalConvergence}, and~\ref{assumption:finiteTimeIdentification}; \(\alpha_L\) is the constant defined in~\eqref{eq:alphaL}, \(\alpha_{\cM_\theta}\) is the constant defined in~\eqref{eq:alphaM}; \(\sol\) is the set of primal-dual solutions defined in~\eqref{eq:primalDualSolutions}; \(\delta\) is the {\em distance to degenerancy} defined in~\eqref{eq:degenerancyMetric}; \(\eta\) is the stepsize defined in~\eqref{eq:algorithm}; \(\bar{L}_j\) and \(\bar{\sigma}_j\) are optimal constants satisfying the \(P\)-Lipschitz regularity bounds of Assumption~\ref{assumption:asymptoticIdentification}; and 
    with \(\rho_{\indexM}=\left\lceil e\gamma^\irate\lambdamax(P)^\irate/\alpha_{\indexM} ^{\irate}\right\rceil\).
    \end{THM}

    %Theorem \ref{theorem:twoStageConvergence} presents the total iteration complexity required for the algorithm to compute an $\eps$-optimal solution. 
    Thus, he overall complexity is $O\left(\frac{1}{\alpha_{L}^2 \delta^2} + \frac{1}{\alpha_{\indexM}^2}\ln\left(\frac{\delta}{\alpha_{\indexM}} \right) \right)$. A few observations in order. First, both $\alpha_L$ and $\alpha_{\indexM}$ are local metric subregularity constants, thereby avoiding dependence on potentially conservative global constants. Second, this complexity order is consistent with that of PDHG for linear programming, as described in \cite{lu2024geometry}. Our results extend the analysis in \cite{lu2024geometry} to general convex optimization problems and a broader class of algorithms.
    % \end{corollary}

    % \noindent\textit{Remark.} In the case that \(\delta\) is large, i.e. if \(\bar{\theta}:= R/\delta\in (0,1/2]\) where \(R>0\) is such that \(z^k,\tilde{z}^k\in\ballP{\zs}{R}\) for all \(k\in\NN\), then the finite time identification stage collapses to just the identification of \(\YN\). Then, the number of steps required for \(\varepsilon\)-accuracy in the context of Theorem~\ref{theorem:twoStageConvergence} becomes
    % \[
    %     \left\lceil \max_{j\in\indexN}\frac{\bar{\sigma}_j}{\eta \bar{L}_j}\right\rceil+ \frac{\rho_{\cM_{\bar{\theta}}}}{\rate}\ln\left[\frac{\gamma\lambdamax(P)\dist_2(z^0,\sol)}{\alpha_{\cM_{\bar{\theta}}}\varepsilon}\right]\,.
    % \]

    \subsection{Comparison between metric subregularity moduli}
    \label{section:constantComparison}
    In this section, we provide a comparison between the three metric subregularity moduli $\alpha_{G}$, $\alpha_L$, and $\alpha_{\indexM}$, defined in \eqref{eq:alphaG}, \eqref{eq:alphaL} and \eqref{eq:alphaM}, respectively, which underpin our theoretical results. We start by showing that under mild conditions, we have $\alpha_{\indexM} \geq \alpha_L \geq \alpha_G.$ %Later, we exhibit a simple linear program in which the gaps between these constants span several orders of magnitude.
    %Proposition~\ref{prop:constantComparison} shows that both the local constants \(\alpha_{\cM_\theta}\) and \(\alpha_L\) are better than \(\alpha_G\). Also, \(\alpha_{\cM_\theta}\) is provably better than \(\alpha_L\) when  \(P\) is positive definite and \(\gamma\) is scaled by the square root of the condition number of \(P\). The proof of it is deferred to Appendix~\ref{appendix:proof-constantComparison}.
    % \begin{assumption}
    %     \label{assumption:positiveDefinite}
    %     The geometry matrix \(P\) is positive definite and the scaling factor for the constant \(\alpha_{\cM_\theta}\) satisfies \(\theta\leq\kappa^{-1/2}(P)\).
    % \end{assumption}
    \begin{proposition}
    \label{prop:constantComparison}
        The metric subregularity constants satisfy \(\min\{\alpha_{\indexM},\alpha_L\}\geq\alpha_G\). Furthermore, if the condition number $\kappa(P) \leq 4$, then, \(\alpha_{\indexM}\geq \alpha_L\).
    \end{proposition}
   The upper bound on $\kappa(P)$ is immaterial; except for ADMM, all algorithms we study satisfy it provided that we take the stepsize $\eta$ small enough. Further, we could relax it, as it only reflects our choice of the identification radius of $\delta/2$. Specifically, we could modify Proposition~\ref{theorem:linearConvergenceFast} to reach the ball $\theta \delta$ with $\theta \in (0, 1/2]$ and all our rates will change by constants and the constraint here would reduce to $\kappa(P) \leq \theta^{-2}.$ We decided to state this version of the results in favor of simplicity. The next example shows that the gap between these constants can be significant even in low dimensions.
    % As shown in the following example, \(\alpha_{\indexM}\), \(\alpha_L\) and \(\alpha_G\) can be of orders of magnitude different.
    % The derivation of the results of the following example is deferred to Appendix~\ref{appendix:example}.
    % In remarkable consistency with Proposition~\ref{prop:constantReducedSystemRelation}, \(\alpha_L\) doesn't shrink when \(c_1/c_2\) is small, since \(\dist_2(z,\sol_\indexB\cap\YN)\) is small when \(\dist_2(0,\sub(z))\) is small---not as \(\dist_2(z,\sol)\), which remains large 
    % %around the point \((0,1+c_2/c_1)\)
    % when \(\dist_2(0,\sub(z))\) is small, shrinking \(\alpha_G\). Furthermore,  \(\alpha_{\indexM}=1\) independently of \(c_1/c_2\), as it depends only on the local geometry of problem~\eqref{eq:reducedPrimalProblem}, which is essentially invariant to \(c_1/c_2\).
    % \begin{wrapfigure}{R}{0.32\textwidth}
    %     \vspace{-0.8cm}
    %     \begin{center}
    %         \includegraphics[width=0.3\textwidth]{sections/example.png}
    %     \end{center}
    %     \caption{Example~\ref{ex:rotatedHouse} primal geometry. \((a)\) shows \(\xs=(\delta,\delta)\), \((b)\) considers \(\delta=1\) and \(\xs=(\frac{c_2}{c_1},\frac{c_1}{c_2})\).}
    %     \label{fig:example}
    % \end{wrapfigure}
    \begin{example}  
    \label{ex:rotatedHouse}
    Let \(c\in\mathbb{R}^2\) such that \(c_1,c_2>0\) and \(\|c\|_2=1\). Define
    % \(\displaystyle p^\star = \begin{cases}\displaystyle\min_{x\in\R^2} & \langle c,x\rangle\\
    %     \text{s.t.} &  \langle c,x\rangle \geq \|c\|_1 \\
    %      & x_1,x_2\geq 0 .
    %     \end{cases}\)  
    % \begin{minipage}{0.3\textwidth}
    %     \centering
    %     \includegraphics[width=\textwidth]{sections/example.png}
    %     \captionof{figure}{Example~\ref{ex:rotatedHouse} primal geometry.}
    %      \label{fig:example}
    % \end{minipage}
    \begin{equation}
        \label{eq:example}
        \min_{x\in\R^2} \langle c,x\rangle \qquad 
        \text{s.t.} \qquad \langle c,x\rangle \geq \|c\|_1 \quad
         \text{and} \quad x_1,x_2\geq 0 .
    \end{equation}
            \begin{figure}[!t]
    \centering
    \begin{subfigure}[t]{0.31\textwidth}
        \includegraphics[width=\textwidth]{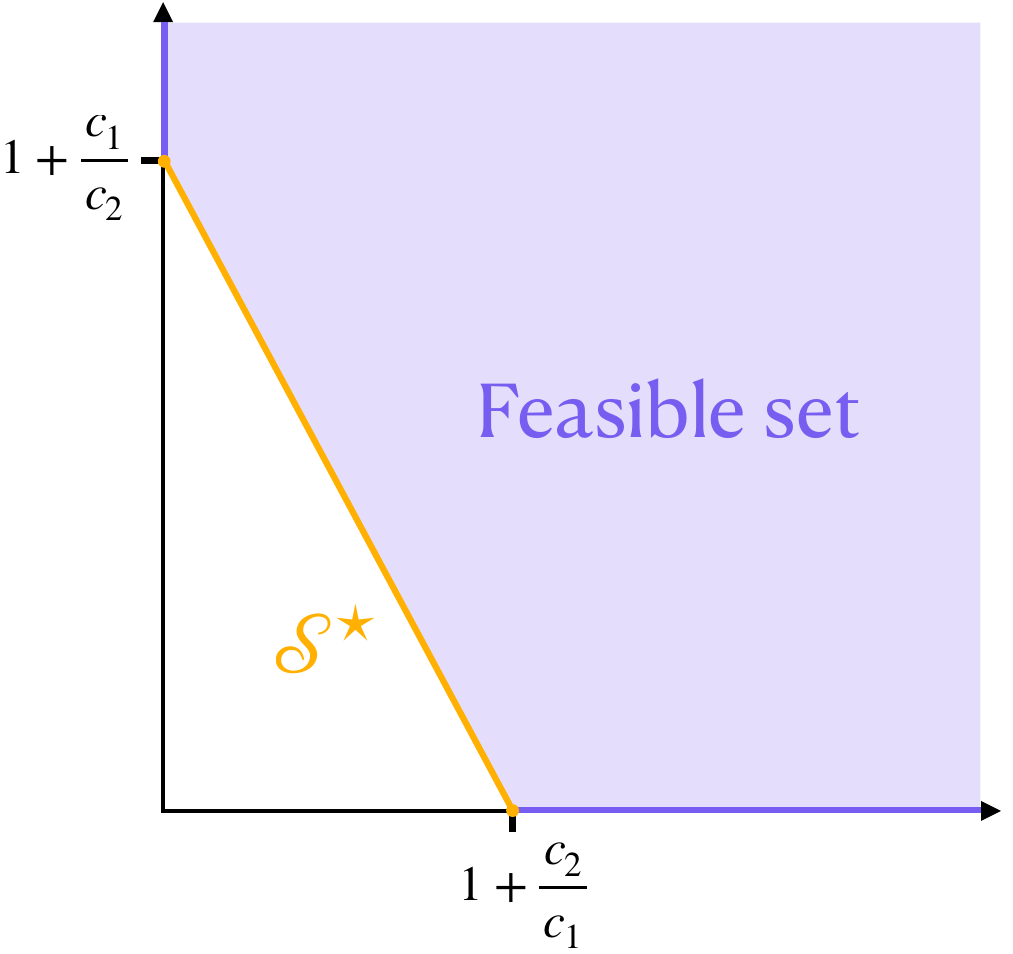}
        \caption{Primal geometry.}
        \label{fig:exampleSetting}
    \end{subfigure}\hspace{2cm}
    \begin{subfigure}[t]{0.325\textwidth}
        \includegraphics[width=\textwidth]{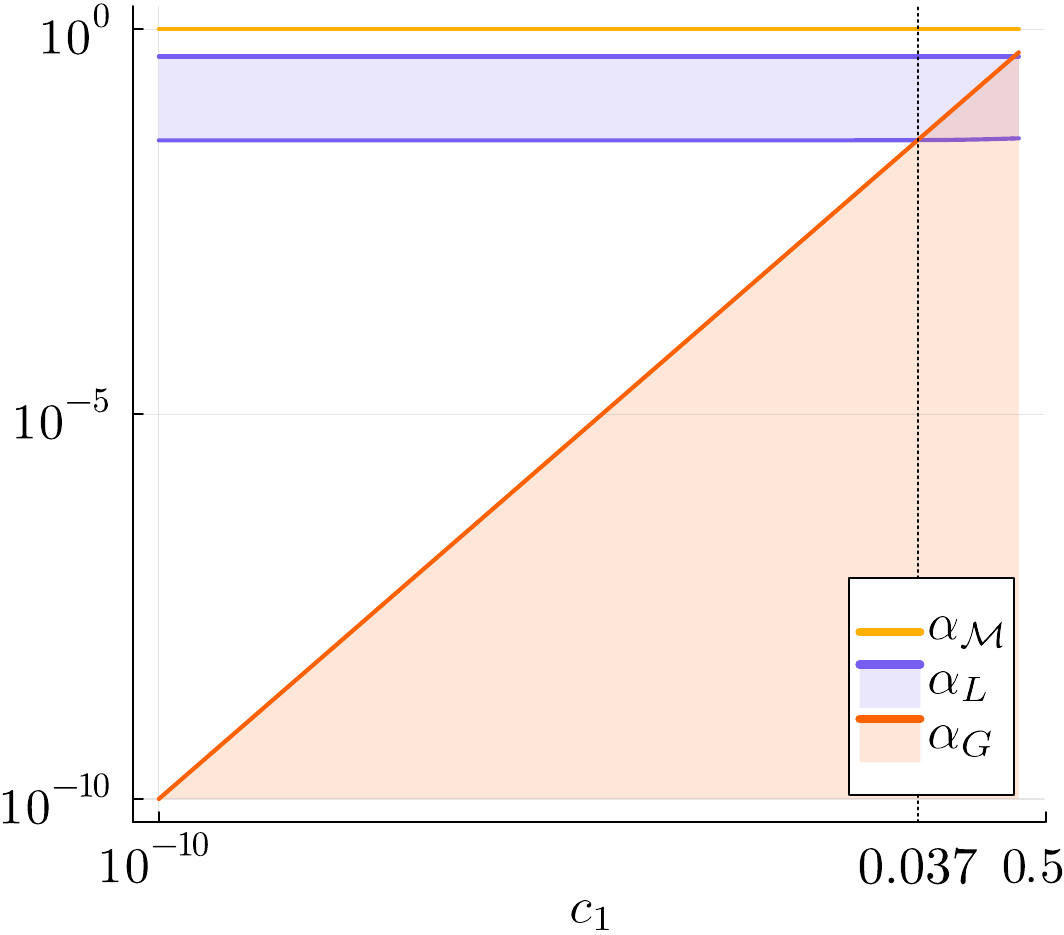}
        \caption{Bounds on \(\alpha_{\indexM}\), \(\alpha_L\), \(\alpha_G\).}
        \label{fig:constantComparison}
    \end{subfigure}
    \caption{Example~\ref{ex:rotatedHouse}. The left plot shows the feasible set and solutions $\sol.$ The right plot displays the regions where the metric subregularity moduli could land versus $c_1$ (we take $c_1 \leq c_2).$
    }
    \label{fig:example}
\end{figure}
    % \begin{figure}[!ht]
    %     \centering
    %     \includegraphics[width=\textwidth]{figures/example_constants_4.pdf}
    %     \caption{Example~\ref{ex:rotatedHouse}. \((a)\) Primal geometry. \((b)\) Regions of possibility of \(\alpha_{\indexM}\), \(\alpha_L\), and \(\alpha_G\), for \(\tau=2\), \(\wh{\delta}=1\), and varying \(c\).}
    %     \label{fig:example}
    % \end{figure}
    % For \(\zs=(\delta,\delta,1,0,0)\) (Fig.~\ref{fig:example} \((a)\)) we have \(\alpha_{\indexM}=1\) while
    %     \begin{itemize}
    %         \item[\((i)\)]\textbf{(Near degeneracy)}: \(\alpha_L\leq \frac{2\delta}{\sqrt{1+\delta^2}}=O(\delta)\).
    %         \item[\((ii)\)]\textbf{(Iterates far from solution)}: \(\alpha_L\leq \frac{2}{\sqrt{1+R^2}}=O(\frac{1}{R})\).
    %         \item[\((iii)\)]\textbf{(Almost L.I. gradients)}: \(\alpha_G\leq \min\{c_1,c_2\}\).
    %     \end{itemize}
        % \[
        %     \alpha_G\leq \min\left\{\frac{2\delta}{\sqrt{1+\delta^2}},\frac{2}{\sqrt{1+R^2}},\min\{c_1,c_2\}\right\},\quad\alpha_L\leq \left\{\frac{2\delta}{\sqrt{1+\delta^2}},\frac{2}{\sqrt{1+R^2}}\right\}\quad\text{and}\quad \alpha_{\indexM}=1\,.
        % \]
        \noindent We show in Appendix~\ref{appendix:example}, using a blend of theoretical reductions and numerics, that when \(\tau=2\), \(\zs\in\relint(\sol)\) and \(P=I\), we have
        \[
             \alpha_G\leq\min\{c_1,c_2\},\quad0.037\leq\alpha_L\leq 0.44,\quad\text{and}\quad\alpha_{\indexM}=1.
        \]
        Figure~\ref{fig:constantComparison} shows that for small values of \(\min\{c_1,c_2\}\), the  global modulus \(\alpha_G\) is orders of magnitude smaller than the local moduli \(\alpha_L\) and \(\alpha_{\indexM}\). %We see that even in this simple, low-dimensional example, these the `local' notions of conditioning $\alpha_L$ and $\alpha_{\indexM}$ can be orders of magnitude larger than the global one $\alpha_G$. 
    \end{example}

\section{Experiments}
\label{sec:experiments}
In this section, we present numerical results to verify our major theoretical findings, i.e., finite-time identification with subsequent linear convergence, even in presence of degeneracy. In particular, we run EGM, PDHG and ADMM over linear programming (LP), convex quadratic programming (QP), and convex quadratically constrained quadratic programming (QCQP) instances. PPM is excluded as its update rule does not have closed-form solution on these instances. The code for reproducing these experiments is available at
%\[
\begin{quote} \centering
    %\text{
    \url{https://github.com/pizqleh/degenerate-active-set}.%}
\end{quote}
%\]
%\jnote{Does convex QCQP satisfy metric subregularity?} \md{YOLO, Jinwen}
% To verify our theoretical results, we perform numerical experiments for the algorithms PDHG, ADMM, and EGM over LP, QP, and QCQP instances. The proximal point method, although analyzed in section~\ref{section:algorithms}, was not implemented because its updates don't have a closed-form solution. Our results demonstrate the two-stage behavior at every instance, provided that we initialize the algorithms sufficiently far from the solution set, i.e., outside of the rapid convergence region. Furthermore, whether the instance is degenerate or not, the second stage is always triggered by the identification of the set~\eqref{eq:identifiableSet}. 
% The experiments were run on a MacBook Pro with an Apple M1 chip and 16 GB of RAM. Then,  the instances for the experiments were selected from the small to medium-sized problems (in the order of thousands of constraints and variables) of the MIPLIB 2017, Maros-Meszaros, and QPLIB datasets.
\paragraph{Experiment setup.} We use a MacBook Pro with an Apple M1 chip and 16 GB of RAM for all experiments. We test the following three classes of problems. 
\begin{itemize}[leftmargin=10pt]
    \item[] \textbf{LP.} The instances are obtained by constructing root-node linear relaxations of mixed-integer programs from \textsc{MIPLIB}~2017, which we write in standard form
    \begin{equation*}
    \begin{aligned}&\displaystyle\min_{x\in\R^n} \langle c,x\rangle\quad\text{s.t.}\quad  Ax\leq b\,.
    \end{aligned}
    \end{equation*} 
    We initialize algorithms either at zero or at a random point of a sphere of radius \(10^3\), if zero is already in the rapid convergence region.
    
    \item[] \textbf{Convex QP.} We select instances from Maros-Meszaros datasets written in standard form
    \begin{equation*}
    \begin{aligned}&\displaystyle\min_{x\in\R^n} \langle c,x\rangle+\frac{1}{2}\langle x,Qx\rangle\quad\text{s.t.}\quad  Ax\leq b\,.
    \end{aligned}
    \end{equation*} 
    We initialize either at zero or at a random point of the sphere of radius \(10^{4}\), if zero is already in the rapid convergence region. 
    
    \item[] \textbf{Convex QCQP.} We consider instances from QPLIB dataset written in standard form
    \begin{equation*}
    \displaystyle\min_{x\in\R^n} \langle c^0,x\rangle+\frac{1}{2}\langle x,Q^0x\rangle\quad\text{s.t.}\quad  \langle c^k,x\rangle+\frac{1}{2}\langle x,Q^kx\rangle \leq b^k\quad\text{for}\quad k\in\{1,\dots m\}\,.
    \end{equation*} 
    We only consider EGM for this class, since all the other algorithms cannot handle quadratic constraints. We initialize EGM either at the suggested initial point \(z^0_{\text{QPLIB}}\) by QPLIB or, if that is in the rapid convergence region, at a random point of a sphere of radius \(10^2\) centered at \(z^0_{\text{QPLIB}}\).
\end{itemize}
We set the iteration limit to \(10^6\) for all experiments. We use the stepsize \(0.99\cdot \|A\|_{\rm{op}}^{-1}\) for all algorithms solving LPs and QPs.\footnote{This step size ensures convergence for PDHG and ADMM, but not necessarily for EGM. Nevertheless, EGM still converges in the examples we tested.}  For QCQP, we use the stepsize \((\|A\|_{\rm{op}}+\sum_{k=0}^m\|Q^k\|_{\rm{op}})^{-1}\), where \(A\) is the matrix whose \(k\)th row is \(c^k\).\footnote{This step size does not ensure convergence for EGM either, but it works in our tested examples.}
\begin{figure}[t]
    \centering
    \begin{subfigure}[t]{0.325\textwidth}
        \includegraphics[width=\textwidth]{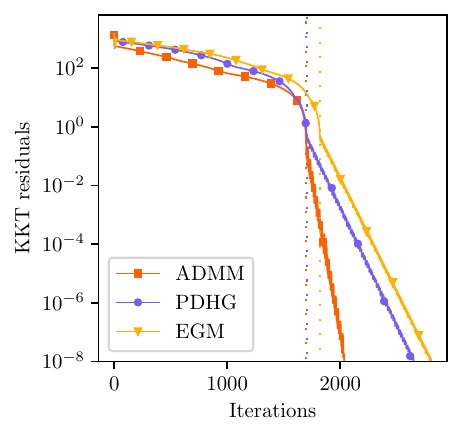}
        \caption{\texttt{stein9inf}}
        \label{fig:LP1}
    \end{subfigure}
    \begin{subfigure}[t]{0.325\textwidth}
        \includegraphics[width=\textwidth]{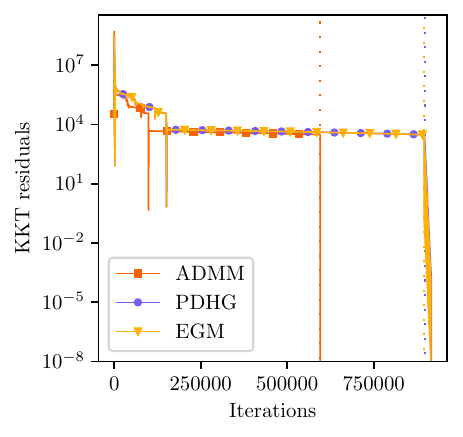}
        \caption{\texttt{gt2}}
        \label{fig:LP2}
    \end{subfigure}
     \begin{subfigure}[t]{0.325\textwidth}
        \includegraphics[width=\textwidth]{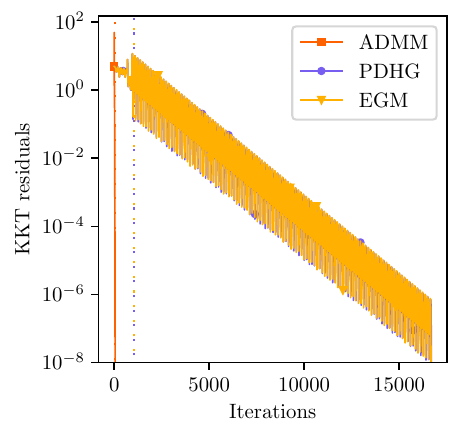}
        \caption{\texttt{noswot}}
        \label{fig:LPdegen}
    \end{subfigure}\\
    \caption{Root-node relaxation of MIPLIB 2017 linear programs. Vertical dotted lines indicate active set identification. Figures~\ref{fig:LP2} and~\ref{fig:LPdegen} all methods converge to degenerate solutions. In~\ref{fig:LPdegen}, ADMM identifies \(\cM\) and converges in just a couple of iterations.
    }
    \label{fig:LPExperiments}
\end{figure}
\paragraph{Convergence criterion.} We report the KKT residual as progress measure, namely,
\[
    \operatorname{KKT}(x,y)=\left\|\begin{bmatrix}(f(x)-h(y))_+\\
    (G(x))_+\\
    (-y)_+
    \end{bmatrix}
    \right\|_2\,.
\]
The KKT residual penalizes the deviations from the KKT system~\eqref{eq:primalDualSolutions}, whence it is zero if and only if \(z=(x,y)\in\sol\). We terminate the algorithms when the KKT residual of their iterates is no larger than tolerance \(10^{-8}\). We denote \(\bar{k}\) as the index of the last iteration of the generated sequence, and we set the converging optimal solution \(\zs=z^{\bar{k}}\). % Since \(\sol\) is unknown, for our results we show plots of \(\operatorname{KKT}(x^k,y^k)\) instead of \(\dist_2(z^k,\sol)\).

\paragraph{Active set identification and degeneracy.}
In order to account for numerical inaccuracies, we define the approximate identifiable set as
\[
    \cM^\eps=\big\{(x,y)\in\R^n\times\R^m_+\,|\, G(x)_{\indexN^\eps}<-\eps,\, |y_{\indexN^\eps}|<\eps,\ \text{and }y_{\indexBs^\eps}>\eps\big\}\,,
\]
where \(\indexN^\eps=\{j\in[m]:g_j(x^{\bar{k}})<-\eps,\,|y_j^{\bar{k}}|<\eps\}\) and \( \indexBs^\eps=\{j\in[m]:y_j^{\bar{k}}>\eps\}\).
% \begin{align*}
%     \indexN^\eps&=\{j\in[m]:g_j(x^{\bar{k}})<-\eps,\,|y_j|<\eps\}\\
%     \indexBs^\eps&=\{j\in[m]:|g_j(x^{\bar{k}})|<\eps,\,y^{\bar{k}}>\eps\},\ \text{and}\\
%     \indexBw^\eps&=\{j\in[m]:|g_j(x^{\bar{k}})|<\eps,\, |y_j^{\bar{k}}|<\eps\}\,.
% \end{align*}
% for each \(k\in [\bar{k}]\), the set of candidate nonactive and active indices at iteration $k$ as
% \begin{equation}
% \label{eq:indexSetsK}
%     \begin{aligned}
%         \indexN^k&=\{j\in[m]:g_j(x^k)<-\varepsilon,\, |y_j|<\varepsilon\}\quad\text{and}\quad
%       \indexBs^k&=\{j\in[m]:\,y^k_j>\varepsilon\}\,,
%     \end{aligned}
% \end{equation}
% respectively. 
Here, we set the numerical tolerance to \(\varepsilon=10^{-10}\).
% We then consider the approximate identifiable set as
% \[
%     \cM^\eps=\big\{(x,y)\in\R^n\times\R^m_+\,|\, G(x)_{\indexN^\eps}<-\eps,\, |y_{\indexN^\eps}|<\eps,\ \text{and }y_{\indexBs^\eps}>\eps\big\}
% \]
We define the iteration at which the algorithm identifies the active set as
\[
    k^\star=\{\min k\in[\bar{k}]:z^\ell \in\cM^\eps\quad \text{for all}\quad\ell\geq k\}\,.
\]
% \[
%     k^\star=\min\{k\in[\bar{k}]:\indexN^{\bar{k}}\subseteq \indexN^\ell\quad\text{and}\quad \indexBs^{\bar{k}}\subseteq \indexBs^\ell\quad\text{for all}\quad \ell\geq k\}\,.
% \]
The limiting solution \(z^{\bar{k}}\) is declared degenerate if the index set \(\indexBw^\eps=\{j\in[m]:|g_j(x^{\bar{k}})|<\eps,\, |y_j^{\bar{k}}|<\eps\}\) is nonempty.

\begin{figure}[!t]
    \centering
    \begin{subfigure}[t]{0.325\textwidth}
        \includegraphics[width=\textwidth]{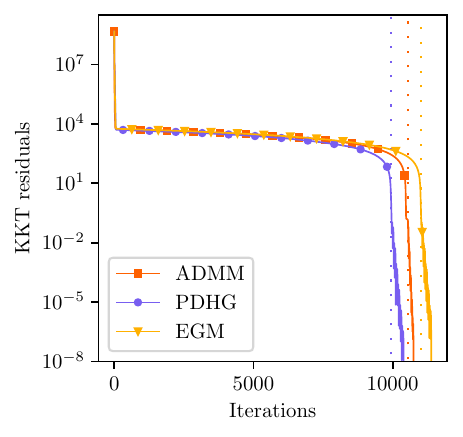}
        \caption{\texttt{HS76}}
        \label{fig:QP1}
    \end{subfigure}
    \begin{subfigure}[t]{0.325\textwidth}
        \includegraphics[width=\textwidth]{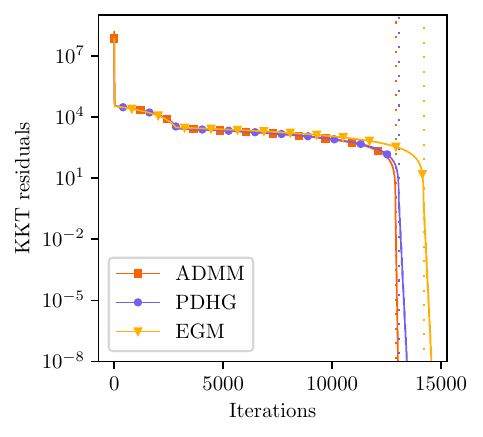}
        \caption{\texttt{ZECEVIC2}}
        \label{fig:QP2}
    \end{subfigure}
     \begin{subfigure}[t]{0.325\textwidth}
        \includegraphics[width=\textwidth]{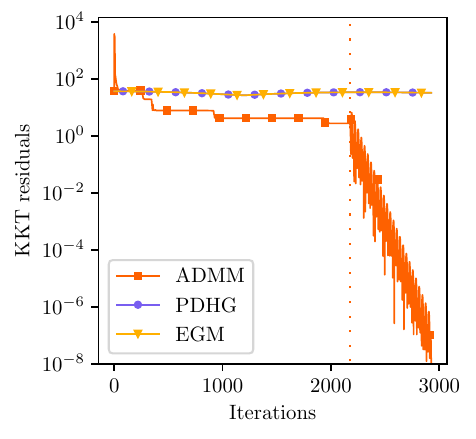}
        \caption{\texttt{QRECIPE}}
        \label{fig:QPdegen}
    \end{subfigure}\\
    \caption{Maros-Meszaros convex quadratic programs. Vertical dotted lines indicate active set identification. In~\ref{fig:QPdegen}, the only algorithm that converged after \(10^5\) iterations was ADMM; PDHG and EGM showed very slow progress.
    }
    \label{fig:QPExperiments}
\end{figure}

\begin{figure}[!t]
    \centering
    \begin{subfigure}[t]{0.325\textwidth}
        \includegraphics[width=\textwidth]{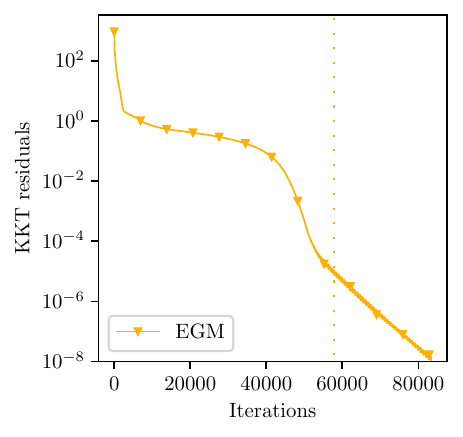}
        \caption{\texttt{7579}}
        \label{fig:QCQP1}
    \end{subfigure}
    \begin{subfigure}[t]{0.325\textwidth}
        \includegraphics[width=\textwidth]{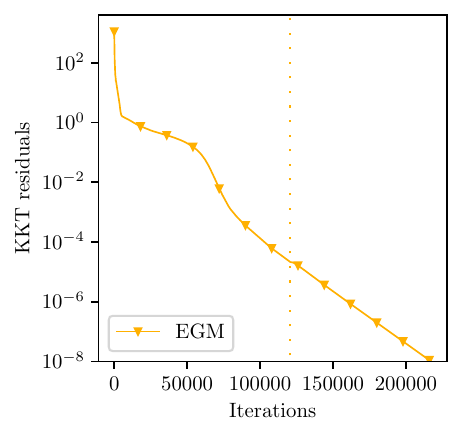}
        \caption{\texttt{3678}}
        \label{fig:QCQP2}
    \end{subfigure}\\
    %  \begin{subfigure}[t]{0.325\textwidth}
    %     \includegraphics[width=\textwidth]{figures/experiments/QP/degenerate/QP_MM_36_KKT.pdf}
    %     \caption{\texttt{QRECIPE}}
    %     \label{fig:QPdegen}
    % \end{subfigure}\\
    \caption{QPLIB convex quadratic programs with quadratic constraints. Vertical dotted lines indicate active set identification. In both~\ref{fig:QCQP1} and~\ref{fig:QCQP2} the transition toward fast linear convergence begins prior to identification, and locks in to a stable linear rate once identification occurs.
    }
    \label{fig:QCQPExperiments}
\end{figure}
\paragraph{Results.} Figure~\ref{fig:LPExperiments} and \ref{fig:QPExperiments}  present the behavior of EGM, PDHG and ADMM on LPs and QPs respectively, while Figure~\ref{fig:QCQPExperiments} displays the behavior of EGM when applied to convex QCQP instances (note that PDHG and ADMM are not applicable to QCQPs as their constraints are not affine). 
As shown, all algorithms under consideration exhibit the expected two-stage behavior across all instances: the iterates initially converge sublinearly toward a solution, and after identification (marked by the vertical dotted lines in the figures), they transition to a much faster linear convergence regime. This behavior persists even when the algorithms converge to a degenerate solution. In particular, in Figures~\ref{fig:LP2}, \ref{fig:LPdegen}, \ref{fig:QPdegen}, and~\ref{fig:QCQP2}, the convergent methods approach a degenerate solution. 
% As shown, all algorithms of consideration exhibit the expected two-stage behavior across all instances: the iterates initially have a sublinear convergence towards a solution, and after identification (i.e., vertical dotted lines in figures), they enjoy convergence with a much faster linear rate. This behaviour is exhibit even when algorithms converge to a degenerate solution. In particular, in Figures~\ref{fig:LP2}, ~\ref{fig:LPdegen}, \ref{fig:QPdegen}, and both subfigures in Figure~\ref{fig:QCQPExperiments}, the (converging) algorithms converge to a degenerate solution.
These empirical observations support the theoretical results from Section \ref{sec:guarantees}.

\section*{Acknowledgments}
We thank Robert M. Freund and Stephen Wright for insightful conversations and pointers to relevant related work.
\bibliographystyle{abbrvnat}
\bibliography{references}
%=================================================================
\appendix
\section{Missing details from Section~\ref{sec:intro}}\label{app:intro}
   \begin{wrapfigure}{r}{0.33\textwidth}
  \begin{center}
      \vspace{-0.8cm}
    \includegraphics[width=0.33\textwidth]{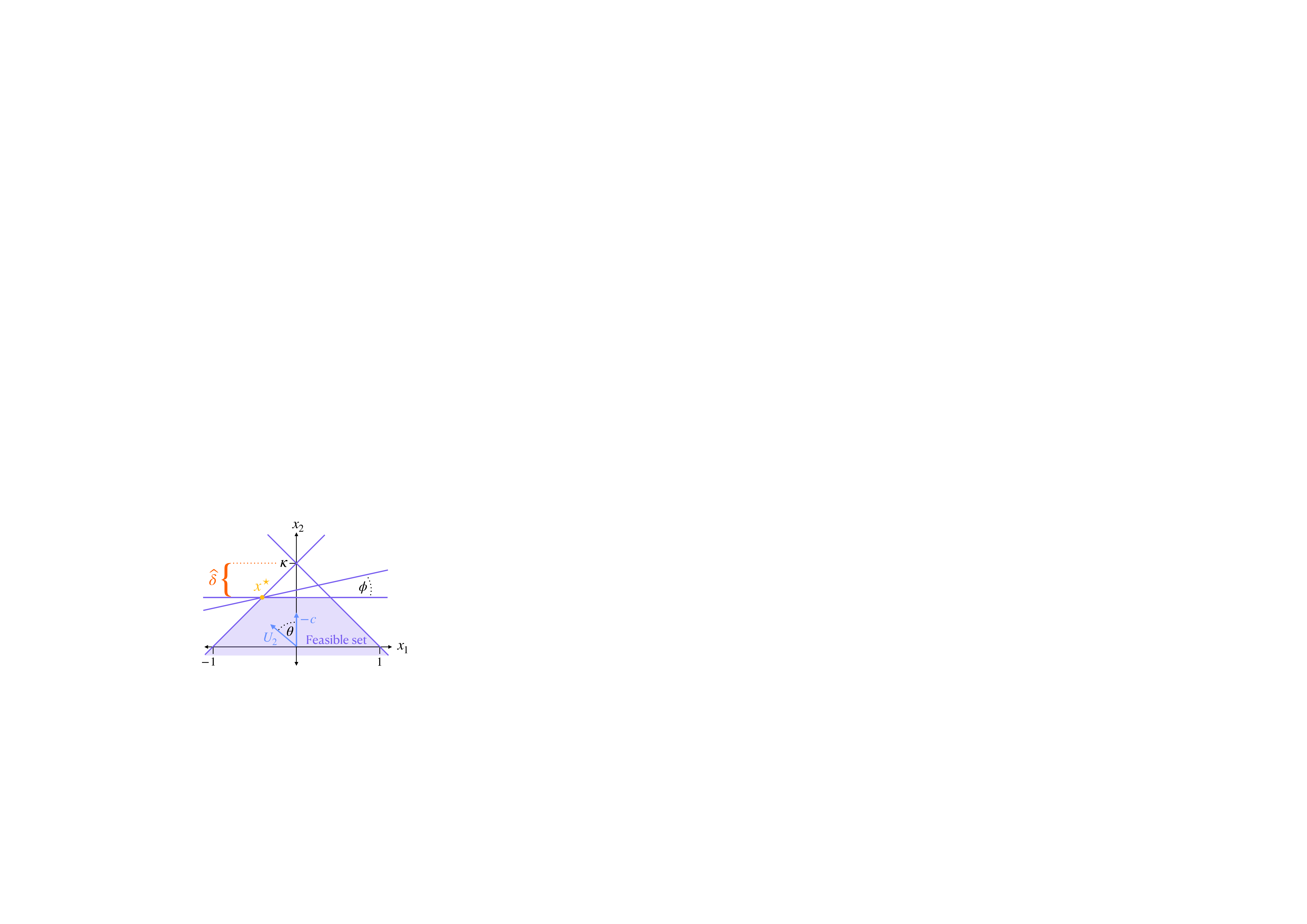}
    \vspace{-0.7 cm}
  \end{center}
  \caption{Primal geometry. The constant \(\wh{\delta}\) is a proxy for the radius of active set stability \(\delta\). The constant \(\zeta\) in~\eqref{eq:intro-example-constraints} satisfies \(\zeta = \arctan(\phi)\). The vector \(U_2\) is the second column of \(U\) and satisfies \(\ker(Q)=\text{span}\{U_2\}\).}
  \label{fig:intro-example-appendix}
\end{wrapfigure}
In this section, we show the details of the QP from Figure~\ref{fig:intro-example}. The QP is written in standard form \begin{equation*}
    \begin{aligned}&\displaystyle\min_{x\in\R^n} \langle c,x\rangle+\frac{1}{2}\langle x,Qx\rangle\quad\text{s.t.}\quad  Ax\leq b\,
    \end{aligned}
    \end{equation*} 
    where the objective is given by
    \begin{align*}
    & c=\left[\begin{array}{rr}
            0\\
            -1
        \end{array}\right]\quad\text{and}\quad Q=UDU^T,\quad\text{where}\quad D=\begin{bmatrix}
        1 & 0\\
        0 & 0
        \end{bmatrix}\\&\qquad \qquad\qquad\text{and} \quad U=\begin{bmatrix}
        \cos(\theta) & -\sin(\theta)\\
        \sin(\theta) & \cos(\theta)
        \end{bmatrix}\,,
    \end{align*}
    with \(\theta=\pi/64\), and the constraints are defined by
    \begin{equation}
    \label{eq:intro-example-constraints}
      A=\left[\begin{array}{rr}
            1 & 1/\kappa\\
            -1 & 1/\kappa\\
            0 & 1\\
            -\zeta & 1
        \end{array}\right]\quad\text{and}\quad b= \left[\begin{array}{c}
            1\\
            1\\
            \kappa-\wh{\delta}\\
            k-\wh{\delta}\left(1-\frac{\zeta}{\kappa}\right)
        \end{array}\right]\,,
    \end{equation}
    \clearpage
    \noindent with \(\zeta=1/6\), \(\kappa=1/2\), and \(\wh{\delta}=1/2^{10}\). To solve this QP, we initialize all algorithms at zero, and we use the following stepsizes: for PDHG, we use \(0.99\|A\|_{\rm{op}}^{-1}\); for ADMM, we use \(2\cdot 0.99\|A\|_{\rm{op}}^{-1}\); and, for EGM, we use \(0.99\sqrt{(\|Q\|_{\rm{op}}+\|A\|_{\rm{op}})^2+\|A\|_{\rm{op}}^2}^{-1}\). These choices ensure the convergence of their respective algorithms. The criteria for declaring `convergence to the active set' and `degeneracy' are the same as in Section~\ref{sec:experiments}. The convergence tolerance is set to \(10^{-10}\), while the active-set and degeneracy tolerances are set to \(10^{-8}\). The solution \((\xs,\ys)\) to which the algorithms converge satisfies
    \[
        A\xs-b=(-3.906\cdot 10^{-3}, 0, 0, 0)\quad\text{and}\quad\ys= (0,         0,         0.863, 0.135)\,.
    \]
    One can verify that \(\indexBw=\{2\}\neq\emptyset\), hence this solution is degenerate.
    
\section{Missing proofs from Section~\ref{section:settings}}

In this section, we present the missing proofs from Section~\ref{section:settings}.

\subsection{Proof of Proposition~\ref{theorem:linearConvergence}}
\label{appendix:proof-linearConvergence}
  In this section, we derive the following slightly more general version of Proposition~\ref{theorem:linearConvergence}. 
  \begin{proposition}[Generalization of Proposition~\ref{theorem:linearConvergence}]
    \label{prop:linearConvergence}
          Suppose Assumptions~\ref{assumption:problem}, \ref{assumption:metric-subregularity}, and~\ref{assumption:rapidLocalConvergence} hold. Let $z^{k}$ be the \(k\)th iterate generated by update \eqref{eq:algorithm}. Then,
          \begin{equation}
    0<\alpha_G:=\inf_{z\in\cD}\frac{\dist_2(0,\sub(z))}{\dist_2(z,\sol)} \qquad \text{with} \qquad \cD= \sol+\tau\ball\,.%\left( \sol+\tau\ball \right)\cap\YY\,.
    \end{equation}
          Further, for any $k \in \NN$ we have
       \begin{equation}
          \label{eq:ay-ay-ay} \max\left\{\dist_2(z^k,\sol),\frac{\dist_P(z^k,\sol)}{\lambdamax(P)^\frac{1}{2}}\right\}\leq \sqrt{e}\nu \exp\left(-\rate \frac{k}{\left\lceil e\nu^{\irate}\right\rceil}\right)\min\left\{\dist_{2}(z^0,\sol),\frac{\dist_P(z^0,\sol)}{\lambdamax(P)^{\frac{1}{2}}}\right\}\,,
        \end{equation}
        where \(\nu=\gamma\lambdamax(P)/\alpha_G\).
        \end{proposition}
  Notice that in this convergence statement, we simultaneously bound the $P$-seminorm and the $\ell_2$-norm, which recovers the original statement. We prove this slightly stronger statement since it will be used later on in other proofs. Further, as Lemma~\ref{prop:normEquivalence} below shows the minimum on the right-hand-side of \eqref{eq:ay-ay-ay} is always attained by the $P$-seminorm term. We included this redundancy as it makes it clear that this statement generalizes Proposition~\ref{theorem:linearConvergence}.

  Before we delve into the proof of this result, we derive two auxiliary lemmas that we will use. The first lemma establishes the equivalence between the $P$-seminorm and the $\ell_2$-norm on the range of $P$. The second lemma shows that `Euclidean' metric subregularity implies $P$-norm metric subregularity. 
  %We start by showing the equivalence between the %The following proposition shows that the 
  %\(P\) semi norm and \(\ell_2\) norms are equivalent over the range of $P$. We will use this fact repetitively in the proof.  %(\range(P)\).

  %\jnote{discuss if this is a proposition or a fact. should we merge some propositions into proofs?}
  
  \begin{lemma}
    \label{prop:normEquivalence}
        For any positive semidefinite matrix \(P\in\cS^{n+n}_+\) and point \(z\in\range(P)\), the following inequalities hold:
        \begin{align*}
            \sqrt{\lambdaminp(P)}\|z\|_2&\leq \|z\|_{P}\leq \sqrt{\lambdamax(P)}\|z\|_2\\
            \frac{1}{\sqrt{\lambdamax(P)}}\|z\|_2&\leq \|z\|_{P^\dagger}\leq \frac{1}{\sqrt{\lambdaminp(P)}}\|z\|_2
        \end{align*}
    \end{lemma}

    \begin{proof}
         Define \(p=\dim(\range(P))\), \(Q=\proj_{\range(P)}\in\R^{(n+m)\times (n+m)}\) and \(U\in\R^{(n+m)\times p}\) such that \(Q=UU^\top \) and \(U^\top U=I\). Notice that since \(P\) is symmetric, \(P=U\Sigma U^\top \) where \(\Sigma\in\RR^{p\times p}\) is a diagonal matrix with non-negative entries. For any $z \in \range(P)$ we have
        \[
            \|z\|_2^2=\|Qz\|_2^2= z^\top Q^\top Qz=z^\top UU^\top UU^\top z=z^\top UU^\top z=\|U^\top z\|^2=\|\Sigma^{-\frac{1}{2}}\Sigma^{\frac{1}{2}}U^\top z\|_2^2,
        \]
        where we used the invariance under projection of $z$.
        On the other hand,
        \begin{align*}
            \|\Sigma^{\frac{1}{2}}U^\top z\|_2^2=\|P^{\frac{1}{2}}z\|_2^2=\langle P^{\frac{1}{2}}z,P^{\frac{1}{2}}z\rangle=\langle z,Pz\rangle=\|z\|_P^2.
        \end{align*}
        Also, by definition
        \begin{align*}
            \lambdamax(\Sigma^{-\frac{1}{2}})^2&=\lambdamax(\Sigma^{-1})=\frac{1}{\lambdaminp(\Sigma)}=\frac{1}{\lambdaminp(P)}, \quad \text{and}\\ \lambdaminp(\Sigma^{-\frac{1}{2}})^2&=\lambdaminp(\Sigma^{-1})=\frac{1}{\lambdamax(\Sigma)}=\frac{1}{\lambdamax(P)}.
        \end{align*}
        Applying these identities in tandem with Cauchy-Schwarz gives
        \begin{align*}
            \|z\|_2^2&=\|\Sigma^{-\frac{1}{2}}\Sigma^{\frac{1}{2}}U^\top z\|_2^2\leq\lambdamax(\Sigma^{-\frac{1}{2}})^2\|\Sigma^{\frac{1}{2}}U^\top z\|_2^2=\frac{1}{\lambdaminp(P)}\|z\|_P^2, \quad \text{and}\\
            \|z\|_2^2&=\|\Sigma^{-\frac{1}{2}}\Sigma^{\frac{1}{2}}U^\top z\|_2^2\geq\lambdaminp(\Sigma^{-\frac{1}{2}})^2\|\Sigma^{\frac{1}{2}}U^\top z\|_2^2=\frac{1}{\lambdamax(P)}\|z\|_P^2.
        \end{align*}
        The result for \(\|z\|_{P^\dagger}\) follows analogously.
    \end{proof}

    % The following Lemma shows that if the saddle sub-differential intersects \(\range(P)\), then we can translate the metric sub-regularity condition in the \(\ell_2\) norm to a metric sub-regularity condition from the \(\ell_2\) norm to the \(P^\dagger\) seminorm.
    % \begin{lemma}
    %     \label{lemma:PDaggerInequality}
    %     Under Assumption~\ref{assumption:metricSubRegularity}, for any positive semidefinite matrix \(P\) and point \(z\in(\sol\cap\cballP{\zs}{R})+\cballt{0}{\tau}\), the following inequality holds
    %     \[
    %         \alpha\dist_2(z,\sol)\leq \sqrt{\lambdamax(P)}\dist_{P^\dagger}(0,\sub(z)\cap\range(P))\,.
    %     \]
    % \end{lemma}

    % \begin{proof}
    %     Considering that \(0\in\range(P)=\range(P^\dagger)\) we have
    %     \begin{align*}
    %         \alpha \dist_2(z,\sol)&\leq \dist_2(0,\sub(z)) & \text{(Assumption~\ref{assumption:metricSubRegularity})}\\
    %         &\leq \dist_2(0,\sub(z)\cap\range(P)) & \text{(\(\sub(z)\cap\range(P)\subseteq\sub(z)\))}\\
    %         &\leq \sqrt{\lambdamax(P)}\dist_{P^\dagger}(0,\sub(z)\cap\range(P)) & \text{(Proposition~\ref{prop:normEquivalence})}
    %         % &=\dist_{P^\dagger}(0,\sub(z)) & \text{(\(\sub(z)\cap\range(P)\neq\emptyset\))}
    %     \end{align*} 
    % \end{proof}

%    The following proposition shows that we can translate the metric sub-regularity constants in the \(\ell_2\) norm to metric sub-regularity constants from the \(P\) seminorm to the \(P^\dagger\) seminorm.
    \begin{lemma}
        \label{prop:metricSubregularityPnorm}
        Let \(P\in\cS_+^{n+m}\) be a positive semidefinite matrix, \(S\subseteq\RR^{n+m}\) a set, and \(z\in\mathbb{R}^{n+m}\) such that \(\alpha \dist_2(z,S)\leq \dist_2(0,\sub(z))\) for some \(\alpha>0\). Then,
        \begin{equation*}
            \frac{\alpha}{\lambdamax(P)^{\frac{1}{2}}} \max\left\{\dist_w(z, S), \frac{\dist_P(z,S)}{\lambdamax(P)^{\frac{1}{2}}}\right\}\leq   \dist_{P^\dagger}(0,\sub(z)\cap\range(P))\,.
        \end{equation*}
        % Then, there exists \(\beta>0\) such that
        % \[
        %     \beta\dist_P(z,S)\leq \dist_{P^\dagger}(0,\sub(z)\cap\range(P))\,.
        % \]
        % In particular, we can take \(\beta=\lambdamax(P)^{-1}\alpha\).
    \end{lemma}

    \begin{proof}
        Let \(\bar{z}\in S\) be arbitrary. Then,
        \begin{align*}
            \|z-\bar{z}\|_2&\geq\|\proj_{\range(P)}(z-\bar{z})\|_2 & \text{(Non-expansiveness)}\\
            &\geq \lambdamax(P)^{-\frac{1}{2}}\|\proj_{\range(P)}(z-\bar{z})\|_P & \text{(Lemma~\ref{prop:normEquivalence})}\\
            &=\lambdamax(P)^{-\frac{1}{2}}\|z-\bar{z}\|_P\,,
        \end{align*}
        where for the last line we used that $\proj_{\range(P)} P\, \proj_{\range(P)} = P.$
        Since \(\bar{z}\in S\) is arbitrary, it follows that \(\dist_P(z, S)\leq\lambdamax(P)^{\frac{1}{2}}\dist_2(z, S)\). On the other hand,
        \begin{equation}
            \begin{aligned}
                \alpha \dist_2(z, S)&\leq \dist_2(0,\sub(z)) & \text{(By assumption)}\\
                &\leq \dist_2(0,\sub(z)\cap\range(P)) & \text{(\(\sub(z)\cap\range(P)\subseteq\sub(z)\))}\\
                &\leq \lambdamax(P)^{\frac{1}{2}}\dist_{P^\dagger}(0,\sub(z)\cap\range(P))\, . & \text{(Lemma~\ref{prop:normEquivalence})}
            \end{aligned}
        \end{equation}
        The result follows immediately by combining these inequalities.
        % Gathering everything up, we conclude that \(\lambdamax(P)^{-1}\alpha\dist_P(z, S)\leq \dist_{P^\dagger}(0,\sub(z)\cap\range(P))\).
    \end{proof}

    We are ready to prove a generalization of Proposition~\ref{theorem:linearConvergence} under our regularity assumptions.
    % \begin{theorem}[\ref{theorem:linearConvergence}]
    %     Under Assumptions~\ref{assumption:problem} and~\ref{assumption:rapidLocalConvergence}, for any  \(k\in\NN\) the \(k\)-th iterate \(z^k\) of the meta algorithm defined in~\eqref{eq:algorithm} satisfies
    %     \[
    %        \dist_2(z^k,\sol)\leq e^{-\frac{k\rate}{\rho}}\nu\dist_{2}(z^0,\sol)\quad\text{where}\quad \nu=\frac{\gamma\lambdamax(P)}{\alpha_G},\quad\rho=\left\lceil e\nu^{\irate}\right\rceil\,,
    %     \]
    %     \(\rate\), \(\gamma\), and \(P\) are the rate of convergence, the sublinear constant, and the positive semidefinite matrix defined in Assumption~\ref{assumption:rapidLocalConvergence}; \(\alpha_G\) is the metric-subregularity constant defined in~\eqref{eq:alphaG}; \(\sol\) is the set of primal-dual solutions defined in~\eqref{eq:primalDualSolutions}.
    % \end{theorem}
    
    \begin{proof}[Proof of Proposition~\ref{prop:linearConvergence}]
    % \jnote{TODO: discuss notation, to be consistent with assumptions? $c$ vs $\gamma$, $\alpha_G$, $\alpha_P$ etc.}
    % Note that
    % \begin{align*}
    %     \lambdamax(P)^{-\frac{1}{2}}\alpha_G\dist_2(z^k,\sol)&\leq \lambdamax(P)^{-\frac{1}{2}}\dist_{2}(0,\sub(z^k))\\
    %     &\leq \lambdamax(P)^{-\frac{1}{2}}\dist_{2}(0,\sub(z^k)\cap\range(P))\\
    %     &\leq \dist_{P^\dagger}(0,\sub(z^k)\cap\range(P))\\
    %     &=\dist_{P^\dagger}(0,\sub(z^k))
    % \end{align*}
    Let us start by establishing a few consequences of metric subregularity. From Assumption~\ref{assumption:metric-subregularity} we obtain \(\alpha_G>0\). Furthermore, from Assumption~\ref{assumption:rapidLocalConvergence} \((ii)\) we have \(z^k\in\cD\) for all \(k\in\NN\). Then, for any \(k\in\NN\) we have
    $
        \alpha_G\dist_2(z^{k},\sol)\leq \dist_2(0,\sub(z^k)).
    $
    % \begin{equation}
    % \label{eq:metricSubregularityLinearConvergence2}
    % \begin{aligned}
    %     \alpha_G\dist_2(z^{k},\sol)&\leq \dist_2(0,\sub(z^k)) & \text{(By definition of \(\alpha_G\)~\eqref{eq:alphaG})}\\
    %     &\leq\dist_2(0,\sub(z^k)\cap\range(P)) & \text{(\(\sub(z^k)\cap\range(P)\subseteq\sub(z^k)\))}\\
    %     &\leq \lambdamax(P)^{\frac{1}{2}}\dist_{P^\dagger}(0,\sub(z^k)\cap\range(P)) & \text{(Proposition~\ref{prop:normEquivalence})}.
    % \end{aligned}
    % \end{equation}
    Hence, invoking Lemma~\ref{prop:metricSubregularityPnorm}, yields that for any \(k\in\NN\) we have
    % \begin{equation}
    %     \label{eq:metricSubregularityLinearConvergence}
    %     \alpha_G\dist_P(z^{k},\sol)\leq \lambdamax(P)\dist_{P^\dagger}(0,\sub(z^k)\cap\range(P))\,.
    % \end{equation}
    % Merging~\eqref{eq:metricSubregularityLinearConvergence2} and~\eqref{eq:metricSubregularityLinearConvergence} we obtain
    \begin{equation}
    \label{eq:metricSubregularityLinearConvergence3}
        \alpha_G\max\left\{\dist_2(z^k,\sol),\frac{\dist_P(z^k,\sol)}{\lambdamax(P)^{\frac{1}{2}}}\right\}\leq \lambdamax(P)^{\frac{1}{2}}\dist_{P^\dagger}(0,\sub(z^k)\cap\range(P))\,.
    \end{equation}

    Fix the integer \(\rho=\lceil e\nu^2\rceil\) where $\nu = \gamma \lambda_{\max}(P)/\alpha_G$. The strategy to prove this result is simple: we show that after $\rho$ consecutive iterations, the distance from $0$ to $\sub(z^k) \cap \range(P)$ contracts by a constant factor, and, then, we relate this back to the distance from the iterates to the solution set.
    Consider two cases.

    \paragraph{Case 1.} First suppose that \(k \geq \rho\) and let \(n\in\NN\) such that
    % \begin{equation}
    % \label{eq:nBounds}
       $ n\rho\leq k \leq (n+1)\rho.$
      % \end{equation}
      % \md{Everything holds if $n \geq 1$, but not with $n = 0$.}
    Hence,
    \begin{align*}
        \dist_{P^\dagger}^{\irate}(0,\sub(z^k)\cap \range(P)) &\leq \frac{\gamma^{\irate}}{\rho}\dist_{P}^{\irate}(z^{k-\rho},\sol) & \text{(Assumption~\ref{assumption:rapidLocalConvergence} \((i)\))}\\
        % & = e^{-1}(\lambdamax(P)^{-1}\alpha_G)^{\irate}\lceil e\nu^{\irate}\rceil\lceil e\nu^{\irate}\rceil^{-1}\dist_{P}^{\irate}(z^{k-\rho},\sol) & \text{(\(\nu,\rho\) definition)}\\
        &\leq e^{-1}(\lambdamax(P)^{-1}\alpha_G)^{\irate}\dist_{P}^{\irate}(z^{k-\rho},\sol)& \text{(\(\nu,\rho\) definition)}\\
        &\leq e^{-1}\dist_{P^\dagger}^{\irate}(0,\sub(z^{k-\rho})\cap\range(P))\,.&\text{(From~\eqref{eq:metricSubregularityLinearConvergence3})}
    \end{align*}
    By recursively applying the same argument $n$ times, we obtain
    \begin{align*}
       \dist_{P^\dagger}^{\irate}(0,\sub(z^{k})\cap\range(P))&\leq e^{-n}\dist_{P^\dagger}^{\irate}(0,\sub(z^{k-n\rho})\cap\range(P))\\
       &\leq e^{-n}\gamma^{\irate}\dist_P^{\irate}(z^0,\sol) & (\text{Assumption~\ref{assumption:rapidLocalConvergence} \((i)\))}\\
        &\leq e^{1-\frac{k}{\rho}}\gamma^{\irate}\dist_P^{\irate}(z^0,\sol) & \text{(\(n\geq k/\rho -1\))}\\
        &\leq e^{1-\frac{k}{\rho}}\gamma^{\irate}\lambdamax(P)\dist_2^{\irate}(z^0,\sol)\,. & \text{(Lemma~\ref{prop:normEquivalence})}
    \end{align*}
    Taking a square root at both sides of the above inequality, we obtain
    \begin{equation}
    \label{eq:saddleBoundLinearConvergence}
        \dist_{P^\dagger}(0,\sub(z^{k})\cap\range(P))\leq \sqrt{e}\exp\left(-\frac{k}{2\rho}\right)\gamma\lambdamax(P)^{\frac{1}{2}}\min\left\{\frac{\dist_P(z^0,\sol)}{\lambdamax(P)^{\frac{1}{2}}},\dist_2(z^0,\sol)\right\}\,.
    \end{equation}
    The result follows by combining~\eqref{eq:metricSubregularityLinearConvergence3} and~\eqref{eq:saddleBoundLinearConvergence}.

    \paragraph{Case 2.} Suppose $k< \rho$, applying the same rationale as before we derive
    \begin{align*}
        \dist_{P^\dagger}^{\irate}(0,\sub(z^k)\cap \range(P)) \leq \frac{\gamma^2}{k}\dist_{P}^{\irate}(z^{0},\sol) \leq\frac{\gamma^2}{k} \lambdamax(P)\dist_{2}^{\irate}(z^{0},\sol).
      \end{align*}
      Combining these inequalities with the fact that $\frac{1}{k}\leq 1 \leq e \exp(-1) \leq e \exp(-k/\rho)$ yields that \eqref{eq:saddleBoundLinearConvergence} holds. Once more invoking \eqref{eq:metricSubregularityLinearConvergence3} yields the stated bound, which completes the proof.
    % On the other hand, considering that \(0\in\range(P)=\range(P^\dagger)\), we have
    % \begin{equation}
    %     \begin{aligned}
    %         \alpha_G \dist_2(z^k,\sol)&\leq \dist_2(0,\sub(z^k)) & \text{(\(\alpha_G\) definition~\eqref{eq:alphaG})}\\
    %         &\leq \dist_2(0,\sub(z^k)\cap\range(P)) & \text{(\(\sub(z^k)\cap\range(P)\subseteq\sub(z^k)\))}\\
    %         &\leq \lambdamax(P)^{\frac{1}{2}}\dist_{P^\dagger}(0,\sub(z^k)\cap\range(P))\,. & \text{(Proposition~\ref{prop:normEquivalence})}
    %     \end{aligned}
    % \end{equation}
    %     Gathering up, we conclude that \(\dist_2(z^k,\sol)\leq e^{-\frac{k\rate}{\rho}}\lambdamax(P)\alpha_G^{-1}\gamma\dist_{2}(z^0,\sol)\).
    \end{proof}

    \subsection{Missing proofs from Section~\ref{section:algorithms}}
\label{appendix:algorithms}
In this section, we prove that the PPM, the ADMM, the PDHG method, and the EGM satisfy the assumptions of the meta-algorithm introduced in Section~\ref{section:metaAlgorithm}.
We begin by weakening  Assumption~\ref{assumption:asymptoticIdentification}. This is required for the analysis of ADMM, which involves a positive semidefinite (PSD) matrix \(P\) instead of a positive definite (PD) matrix. Recall that \(z^{k}=(x^k,y^k)\) and  \(\tilde{z}^{k}=(\tilde{x}^k,\tilde{y}^k)\) denote the main and auxiliary iterates of the meta-algorithm~\eqref{eq:algorithm}.

 \begin{customassumption}{3*}[Weak Asymptotic Identification Conditions]
 \label{assumption:asymptoticIdentificationWeak}
        There exists a positive semidefinite matrix \(P\in\cS_+^{n+m}\) and a convex set \(\YY\subseteq \R^{n+m}\) with \(\R^n\times\R^m_+\subseteq \YY\) such that the following hold.
        \begin{itemize}
            \item[$(i)$] (\textbf{Convergence}) For any initial iterate \(z^0\in\RR^{n+m}\) there exists \(\zs\in\range(P)\) such that
            \[
                \max\{\|z^k-\zs\|_P,\|\tilde{z}^k-\zs\|_P\}\to 0\,.
            \]

            \item[$(ii)$] (\textbf{Dual update}) There exists \(\eta>0\) ({stepsize}) such that the dual update has the form
            \[
                y^{k+1}=\proj_{\RR^{m}_+}(y^k+\eta G(\tilde{x}^{k+1}))\,.
            \]
            Further, \(\tilde{z}^k\in\YY\).
            \item[\((iii)\)]
             (\textbf{Primal Lipschitzness}) For any radius \(t>0\) and index \(j\in\indexN\) there is a constant \(\Lxj\geq 0\) such that
            any point \(z \in\cballP{\zs}{t}\cap\YY\) satisfies
             \[
                     g_j(x)-g_j(\xs) \leq \Lxj\|{z}-\zs\|_P\quad\text{and}\quad |g_j(x)-g_j(\wh{x})|\leq \Lxj\|z-\wh{z}\|_P
               \]
               for all \(\wh{z} \in\cballP{\zs}{t}\cap\YY\) such that  \(\wh{z}_{P^\perp}=z_{P^\perp}\).
                \item[\((iv)\)]
             (\textbf{Dual Lipschitzness}) There exists \(\step\in\NN \cup \{0\}\) satisfying that for any radius \(t>0\) and index \(j\in\indexN\cup\indexBs\) there is a constant \(\Lyj\geq 0\) such that for any initial iterates \(z^0, \wh{z}^0 \in\ballP{\zs}{t}\cap\YY\) we have%\hnote{what is $p$ here, power, or supindex for what?}\pil{Done.}
            \[
            y^{\step}_j-\ys_j \leq \Lyj\|z^0-\zs\|_P \quad \text{if }j\in\indexN, \qquad\text{and}\qquad             |y_j^\step-\wh{y}_j^\step|\leq \Lyj\|z^0-\wh{z}^0\|_P \quad \text{if }j\in\indexBs.
            %\end{cases}
                     %\text{for }j\in\indexN,\quad ,\,\quad\text{and for }j\in\indexBs,\quad |y_j^\step-\wh{y}_j^\step|\leq \Lyj\|z^0-\wh{z}^0\|_P
            \]
               Recall that $z^\step$ and $\wh z^\step$ are the $\step$-th iterate of update \eqref{eq:algorithm} when initialized at $z^0$ and $\wh z^0$, respectively.
            %
            % (\textbf{\(P\)-Lipschitzness}) There exists a fixed $\step \in\NN$ such that for any radius \(t>0\) there are  constants \(\sigma_1,\dots ,\sigma_m,L_1,\dots ,L_m> 0\) satisfying that
            % % one of the following conditions hold
            % for all \({z}^{0}=({x}^0,{y}^{0})\) and \(\wh{z}^{0}=(\wh{x}^{0},\wh{y}^{0})\) in\ \(\cballP{\zs}{t}\), and any \(j\) in \([m]\), the following inequalities hold
            %  \[
            %          |g_j({z}^{\step})-g_j(\wh{z}^{\step})| \leq \Lxj\|{z}^0-\wh{z}^0\|_P\quad\text{and}\quad|({y}^\step)_j-(\wh{y}^\step)_j|\leq \Lyj\|{z}^0-\wh{z}^0\|_P\,
            %    \]
            %    where $z^\step$ and $\wh z^\step$ are the $\step$-th iterate of update \eqref{eq:algorithm} when initialized at $z^0$ and $\wh z^0$, respectively.
            \end{itemize}
     \end{customassumption}
     % The reader might ask why Assumption~\ref{assumption:asymptoticIdentificationWeak} \((iv)\) involves \(\step\).
     When $P$ is PD, Assumption~\ref{assumption:asymptoticIdentification} implies this weaker version. When \(P\) is PSD, the sequence generated by update~\eqref{eq:algorithm} might not converge, which makes the analysis more nuanced. Indeed, our assumptions ensure that it converges in the $P$-seminorm, which is equivalent to having the sequence of iterates projected onto \(\range(P)\) converges. Further, the point \(\zs\) might not belong to $\sol$.
     % Nevertheless, the following is true. %it does belong to the closure of $\proj_{\range(P)}\sol$.
     Moreover, %the main technical issue that arises from \(P\) being PSD is that
     the `simpler' \(P\)-Lipschitz condition in  Assumption~\ref{assumption:asymptoticIdentification} does not hold for ADMM. Nevertheless, as we will see, ADMM satisfies the Lipschitz conditions of  Assumption~\ref{assumption:asymptoticIdentificationWeak} with \(\step=1\). %\hnote{Also state which algorithms $p=0$ correspond to here.}\pil{Done.} 
     Also, we will see PPM, PDHG, and EGM satisfy Assumption~\ref{assumption:asymptoticIdentification}, which implies Assumption~\ref{assumption:asymptoticIdentificationWeak} with \(\step=0\).

The main ingredient in our argument is the next proposition, which identifies the generalized resolvent as a firmly nonexpansive mapping. This property immediately yields convergence of the corresponding fixed-point iteration, and in turn establishes the assumption in most cases.
In its proof and throughout the rest of the appendix, we use the following notation for the decomposition of vectors onto the range of a matrix $P\in\mathcal{S}^d$ and its orthogonal complement.
For any \(u\in\RR^{d}\), let 
    \begin{equation}\label{eq:ohthatP}
        u=u_P+u_{P^\perp}\quad\text{where}\quad u_P\in\range(P),\ u_{P^\perp}\in\range(P)^\perp=\ker(P)\,.
      \end{equation}
      The following is a well-known result; we include its proof for the reader's convenience.
\begin{proposition}
    \label{prop:pFirmlyNonExpansive}
    Let \(T\colon \RR^d\to\RR^d\) be a mapping, \(M\colon\RR^d\rightrightarrows\RR^d\) be a monotone set-valued operator,\footnote{An operator $M$ is monotone if $\dotp{x-y, u - v} \geq 0$ for all $x, y \in \RR^d$, $u \in M(x)$ and $v \in M(y).$} and \(P\in\cS^d_+\) be a positive semidefinite matrix such that
    \[
        P(I-T)(u)\in M(T(u))\quad\text{for all }u\in\RR^d\,.
    \]
    Then, \(T\) is firmly non-expansive in the \(P\) seminorm, in the sense that
    \[
        \|T(u)-T(v)\|_P^2+\|(I-T)(u)-(I-T)(v)\|_P^2\leq \|u-v\|_P^2\quad\text{for all }u,v\in\RR^d\,.
    \]
    In particular, if \(\sol_P:=\{u\in\RR^d:T(u)_P=u_P\}\neq\emptyset\), for any \(u\in\RR^d\) there exists \(u^\star\in\sol_P\) such that \(\|T^k(u)-u^\star\|_P\to0\) as \(k\to\infty\).
\end{proposition}

\begin{proof}
    Let \(u,v\in\RR^d\). Since \(P(I-T)(u)\in M(T(u))\) and \(P(I-T)(v)\in M(T(v))\), by the monotonicity of \(M\) we obtain
    \begin{equation}
    \label{eq:pFirmlyNonExpansives}
        \langle T(u)-T(v),P(I-T)(u)-P(I-T)(v)\rangle\geq0\,.
    \end{equation}
    Therefore, expanding the square and lower-bounding the crossterm yields
    \begin{align*}
        \|u-v\|_P^2&=\|u-T(u)+T(u)-T(v)+T(v)-v\|_P^2\\
                 &=\|(I-T)(u)-(I-T)(v)\|_P^2+\|T(u)-T(v)\|_P^2
                 \\ &\qquad
+2\langle T(u)-T(v),P(I-T)(u)-P(I-T)(v)\rangle\\
        &\geq \|(I-T)(u)-(I-T)(v)\|_P^2+\|T(u)-T(v)\|_P^2 & \text{(Using~\eqref{eq:pFirmlyNonExpansives})}
    \end{align*}
    Then, \((T^k(u)_P)_k\) is a Krasnosel'skiǐ-Mann iteration over \(\range(P)\) whence there exists a fixed point \(u^\star\in\sol_P\) such that \(\|T^k(u)_P-u^\star\|_P\to 0\)~\cite[Theorem 1]{ryu2022large}.
   %\md{TODO: We have not introduced yet what the subindex $P$ (for a vector) means here. Reshuffle.}
\end{proof}

% If \(P\) is invertible then the condition \(P(I-T)(x)\in M(T(x)))\) is equivalent to \(T(x)\in (I+P^{-1}M)^{-1}x\). Then, Proposition~\ref{prop:pFirmlyNonExpansive} refers to a generalized resolvent map. When \(M=\sub\), the above can be seen as a version of PPM over the geometry induced by \(P\).
%The ADMM and the PDHG method are instances of the fixed point iteration of an operator of the form given by Proposition~\ref{prop:pFirmlyNonExpansive} with \(M=\sub\). 
Armed with Proposition~\ref{prop:pFirmlyNonExpansive}, we are ready to prove Propositions~\ref{prop:ppm},~\ref{prop:pdhg},~\ref{prop:admm}, and~\ref{prop:EGM}.
\subsubsection{Proof of  Proposition~\ref{prop:ppm}}\label{appendix:PPM}
             By construction, the PPM update satisfies \(P(z^{k}-z^{k+1})\in\sub(z^{k+1})\) with \(P=\eta^{-1}I\). Then, Assumptions~\ref{assumption:asymptoticIdentification} \((i)\), ~\ref{assumption:rapidLocalConvergence} \((ii)\), and~\ref{assumption:finiteTimeIdentification} \((ii)\) follow from Proposition~\ref{prop:pFirmlyNonExpansive}. Moreover, \cite[Theorem 1]{lu2022infimal} shows
             \[
                \|z^{k+1}-z^{k}\|_P\leq \dist_{P^\dagger}(0,\sub(z^k))\leq \frac{1}{\sqrt{k}}\dist(z^0,\sol)\,.
             \]
             Therefore, Assumptions~\ref{assumption:rapidLocalConvergence} \((i)\) and~\ref{assumption:finiteTimeIdentification} \((i)\) are satisfied with \(\gamma=1\).  Moreover, note that the updates have the form
            \begin{align*}
                x^{k+1},y^{k+1}&=\arg\min_{x\in\R^n}\max_{y\in\R^m_+}f(x)+\langle y,G(x)\rangle +\frac{1}{2\eta}\left(\|x-x^k\|^2-\|y-y^k\|^2\right)\\
                &=\arg\min_{x\in\R^n}\left\{f(x)+\frac{1}{2\eta}\|x-x^k\|^2+\min_{y\in\R^m_+}\left\{-\langle y,G(x)\rangle +\frac{1}{2\eta}\|y-y^k\|^2\right\}\right\}
            \end{align*}
            Using first order optimality conditions we derive that the solution of the inner problem is $y(x) = \proj_{\RR^m_+}(y^k + \eta G(x)).$ Thus,            \begin{align*}
                x^{k+1}&=\argmin_{x\in\R^n}\left\{f(x)+\frac{1}{2\eta}\|x-x^k\|^2-\langle \proj_{\R_+^m}(y^k+\eta G(x)),G(x)\rangle \right.\\
                &\left.\hspace{2 cm}+\frac{1}{2\eta}\|\proj_{\R_+^m}(y^k+\eta G(x))-y^k\|^2\right\}\\
                y^{k+1}&=\proj_{\R_+^m}\left(y^k+\eta G(x^{k+1})\right)
            \end{align*}
            Then, PPM satisfies Assumption~\ref{assumption:asymptoticIdentification} \((ii)\). To check Assumption~\ref{assumption:asymptoticIdentification} \((iii)\) note that, by Assumption~\ref{assumption:problem} \((i)\), the functions \(g_j\) are locally Lipschitz~\cite[Theorem 10.4]{rockafellar1997convex}. Then, since \(P=\eta^{-1}I\), Assumption~\ref{assumption:asymptoticIdentification} \((iii)\) holds with  \(\Lyj=\eta\) and \(\Lxj\) equal to \(\eta\) times a local Lipschitz constant of \(g_j\) over \(\ballt{\zs}{t}\).
            % Furthermore, Assumption~\ref{assumption:asymptoticIdentification} \((iii)\) is trivially satisfied, since \(P\) is positive definite and the functions \(g_j\) are convex, whence locally Lipschitz.
            %Notice that the subproblem for \(x^{k+1}\) might be non-convex and then more difficult to solve than the original primal problem~\eqref{eq:primalProblem}. It is convex when \(G\) is linear, i.e. \(G(x)=Ax\) for some \(A\in\RR^{m\times n}\).
        
   \subsubsection{Proof of Proposition~\ref{prop:pdhg}}\label{appendix:pdhg}
            From~\cite{lu2023unified}, the PDHG update satisfies \(P(z^{k}-z^{k+1})\in\sub(z^{k+1})\) for \(P\).
            A Schur complement argument reveals that \(P\succ 0\) if, and only if,
            % \(\eta^{-2} I\succ  A^\topA\): if and only if
            \(\eta<\|A\|_{\rm{op}}^{-1}\). Then, from Proposition~\ref{prop:pFirmlyNonExpansive}, Assumption~\ref{assumption:asymptoticIdentification} \((i)\), ~\ref{assumption:rapidLocalConvergence} \((ii)\), and~\ref{assumption:finiteTimeIdentification} \((ii)\) are satisfied. Also, \cite[Theorem 1]{lu2022infimal} shows that
             \[
                \|z^{k+1}-z^{k}\|_P\leq \dist_{P^\dagger}(0,\sub(z^k))\leq k^{-\frac{1}{2}}\dist(z^0,\sol)\,.
             \]
             Note that \(\|z^k-\tilde{z}^k\|_P\leq \|z^k-z^{k-1}\|_P\), thus, Assumptions~\ref{assumption:rapidLocalConvergence} \((i)\) and~\ref{assumption:finiteTimeIdentification} \((i)\) are satisfied with \(\gamma=1\). Moreover, the PDHG dual update trivially satisfies Assumption~\ref{assumption:asymptoticIdentificationWeak} \((ii)\). Finally, by Lemma~\ref{prop:normEquivalence} Assumption~\ref{assumption:asymptoticIdentification} \((iii)\) holds with \(\Lxj\leq \|A\|_{\rm{op}}\lambdamin(P)^{-\frac{1}{2}}\) and \(\Lyj\leq \lambdamin(P)^{-\frac{1}{2}}\).

    \subsubsection{Proof of Proposition~\ref{prop:admm}} \label{appendix:ADMM}
            %\md{Pedro, we don't establish Assumption~\ref{assumption:rapidLocalConvergence} (ii), do we want to assume it holds?}\hnote{Assumption~\ref{assumption:rapidLocalConvergence} (ii) is not proved, right? Is there existing result establishing it? This should be double-checked for ADMM.}\pil{(Done.) Yes, we don't prove that assumption; I added it to the assumption of the proposition. I haven't been able to prove it using standard arguments, and I haven't found any existing results that establish it. There are some results in the literature that imply it, given stronger assumptions, such as \(A\) being full rank or \(\sol\) being compact.} 
            Our argument is based on the auxiliary primal problem
            \begin{align}
            \label{eq:admmPrimalProblem}
                    \min_{x\in\RR^n, u\in\RR^m} &\,\, f(x)+\iota_{\RR^m_+}(u)\\
                    \text{s.t.} &\,\, Ax + u =b\,.
            \end{align}
            % The minimax primal-dual problem related to the latter is
            % \begin{equation}
            % \label{eq:admmMinimaxProblem} \min_{x\in\R^n,y\in\R^m}\max_{y\in\R^m}f(x)+\iota_{\R^m_+}(u)+\langle y,Ax-b\rangle-\iota_{\R^m_+}(y)\,.
            % \end{equation}
            %\hnote{State the relationship between ADMM for solving (20) and the original problem, and why this is helpful intuitively here.}\pil{Done.} 
            This problem is a recast of~\eqref{eq:primalProblem} into the standard form of ADMM~\cite[Section 3.1]{boyd2011distributed}. Explicitly, we add the auxiliary variable \(u\geq 0\), which allows us to write the inequality constraint \(Ax\leq b\) as an equality constraint. We focus on it because we can leverage the results from~\cite{lu2023unified} to prove Assumptions~\ref{assumption:asymptoticIdentificationWeak} \((i)\),~\ref{assumption:rapidLocalConvergence}, and~\ref{assumption:finiteTimeIdentification} for the ADMM iterates for this problem. Subsequently, we can translate the results obtained for the auxiliary problem to the original problem. Thus, the proof follows from the next three steps.
            \begin{itemize}[leftmargin=1em]
                \item[] \emph{Step 1.} We prove that Assumptions~\ref{assumption:asymptoticIdentificationWeak} \((i)\),~\ref{assumption:rapidLocalConvergence} and~\ref{assumption:finiteTimeIdentification} hold for the {\em auxiliary} problem~\eqref{eq:admmPrimalProblem}. That is, they hold for the iterates \(w^k= \tilde{w}^k = (u^k,x^k,y^k)\), where we take $(u,x)$ as primal variables.% , the intermediates iterates \(\tilde{w}^k=w^k\), \ and the saddle subdifferential \(\widehat{\sub}=[\partial_u\widehat{\cL},\partial_x\widehat{\cL},\partial_y(-\widehat{\cL})]^\top\) where \(\widehat{\cL}(u,x,y)=f(x)+\iota_{\RR^{m}_+(u)}+\langle y,Ax+u-b\rangle\) is the Lagrangian function of the auxiliary problem~\eqref{eq:admmPrimalProblem}.
                \item[]\emph{Step 2.} Leveraging Step \(1\), we prove that Assumptions~\ref{assumption:asymptoticIdentificationWeak} \((i)\),~\ref{assumption:rapidLocalConvergence} and~\ref{assumption:finiteTimeIdentification} also hold for the {\em original} problem~\eqref{eq:primalProblem}. In this case, they hold for the iterates \(z^k= \tilde{z}^k = (x^k,y^k)\).%, the intermediates iterates \(\tilde{z}^k=z^k\), and the saddle subdifferential \(\sub=[\partial_x \cL,\partial_y(-\cL)]^\top\)
                % where \(\cL(x,y)=f(x)+\langle y,Ax-b\rangle-\iota_{\RR^m_+}(y)\) is the Lagrangian function of the original problem~\eqref{eq:primalProblem}.
                \item[] \emph{Step 3.} Finally, we show that items \((ii)\), \((iii)\), and \((iv)\) of Assumption~\ref{assumption:asymptoticIdentificationWeak} hold for the original problem~\eqref{eq:admmPrimalProblem} with iterates \(z^k= \tilde{z}^k = (x^k,y^k)\).%, the intermediate iterates \(\tilde{z}^k=z^k\), and the constraints \(G(x)=Ax-b\).
            \end{itemize}
            To establish these steps, we will use the Lagrangian of the original and auxiliary problems, that is,
            \begin{align*}
\cL(x,y)=f(x)+\langle y,Ax-b\rangle-\iota_{\RR^m_+}(y), \quad \text{and} \quad
\widehat{\cL}(u,x,y)=f(x)+\iota_{\RR^{m}_+(u)}+\langle y,Ax+u-b\rangle,
            \end{align*}
            respectively. As well as their corresponding minimax subdifferential
            \begin{equation}\label{eq:admm-subdifferentials}
                  \sub(x,y)=\begin{bmatrix}
                 \partial f(x)+A^\top y\\ b-Ax+N_{\RR^m_+}(y)
                 \end{bmatrix},\quad\text{and}\quad\widehat{\sub}(u,x,y)=\begin{bmatrix}
                 N_{\RR^m_+}(u)+y\\
                 \partial f(x)+A^\top y\\
                 b-Ax-u
                 \end{bmatrix}.
              \end{equation}
             \noindent\textbf{Step 1.} 
             To analyze the iterates \(w^k=\widehat{w}^k = (u^k,x^k,y^k)\) for the {\em auxiliary} problem~\eqref{eq:admmPrimalProblem}, define the matrix \(\widehat{P}\in\mathcal{S}^{m+n+m}\) such that
              \[
                \widehat{P}=\begin{bmatrix}
                    0 & 0 & 0\\
                    0 & \eta A^\top A & A^\top\\
                    0 & A & \frac{1}{\eta}I
                    \end{bmatrix}\,.
             \]
             A Schur-complement argument %if and only if \(\eta A^\topA-\eta A^\topA=0\succeq 0\), which is true unconditionally.
             shows that this matrix is PSD regardless of $\eta$. Invovking~\cite[Corollary 3]{lu2023unified}, we obtain that the ADMM update satisfies \begin{equation}\label{eq:whataninclusion}\widehat{P}(w^k-w^{k+1})\in\widehat{\sub}(w^{k+1})\,.\end{equation} Then, Assumption~\ref{assumption:finiteTimeIdentification} \((ii)\) holds follows from Proposition~\ref{prop:pFirmlyNonExpansive}. Further, there exists \(w^\star=(u^\star,\zs)\in\range(\widehat{P})\) such that \(\|w^k-w^\star\|_{\widehat{P}}\to 0\), whence Assumption~\ref{assumption:asymptoticIdentificationWeak} \((i)\) is also satisfied. Use \(\widehat{\cS}^\star\) to denote the set of primal dual solutions of the auxiliary problem~\eqref{eq:admmPrimalProblem}, i.e., \(\widehat{\cS}^\star=\widehat{\sub}^{-1}(0)\). Using~\cite[Theorem 1]{lu2022infimal} we get
             \[
                \|w^k-w^{k+1}\|_{\widehat{P}}\leq\dist_{\widehat{P}^\dagger}(0,\sub(w^k))\leq k^{-\frac{1}{2}}\dist(w^0,\widehat{\cS}^\star)\,.
             \]
             Therefore, Assumption~\ref{assumption:rapidLocalConvergence} \((i)\) and~\ref{assumption:finiteTimeIdentification} \((i)\) are satisfied with \(\gamma=1\).\\

             \noindent\textbf{Step 2.} To analyze the iterates \(z^k= \tilde{z}^k = (x^k,y^k)\), we now consider the matrix \(P\).
             By the definition of \(\widehat{P}\), for any \(w=(u,z)\) and \(\bar{w}=(\bar{u},\bar{z})\) in \(\mathbb{R}^{m+n+m}\) we have
             \begin{itemize}
             \item[\((a)\)]
                \(\|w-\bar{w}\|_{\widehat{P}}=\|z-\bar{z}\|_P\), and
            \item[\((b)\)] \(w\in\range(\widehat{P})\) if, and only if, \(z\in\range(P)\).
             \end{itemize}
             Therefore, we obtain \(\|z^k-\zs\|_P\to 0\) from \((a)\), and \(\zs\in\text{range}(P)\) from \((b)\). Thus, leveraging Step 1 we conclude that Assumption~\ref{assumption:asymptoticIdentificationWeak} \((i)\) holds. Furthermore, applying \((a)\) together with Step 1
            % Denote \(\wh{T}:\RR^{n+m+n}\to\RR^{n+m+n}\) and \(T:\RR^{n+m}\to\RR^{n+m}\) as the update operators for the iterates \(w^k\) and \(z^k\), respectively. 
            we derive that for any \(k\in\NN\),
            \[
                \|z^{k+1}-\zs\|_P=\|w^{k+1}-w^\star\|_{\wh{P}}\geq\|w^k-w^\star\|_{\wh{P}}=\|z^k-\zs\|_P\,.
            \]
            %  \begin{equation}
            %  \label{eq:admmFejer}
            %  \begin{aligned}
            %     \|z^{k+1}-\zs\|_P& =\|w^{k+1}-w^\star\|_{\wh{P}} & \text{(From \((a)\))}\\
            %     &\geq \|w^k-w^\star\|_{\wh{P}} & \text{(Step \((i)\))}\\
            %     &=\|z^k-\zs\|_P\,. & \text{(From \((a)\))}
            % \end{aligned}
            % \end{equation}
            Taking limits shows that item \((ii)\) of Assumption~\ref{assumption:finiteTimeIdentification} holds.
            
            To show Assumption~\ref{assumption:finiteTimeIdentification} \((i)\), we claim that 
            %By the definition of the original~\eqref{eq:primalProblem} and auxiliary~\eqref{eq:admmPrimalProblem} problems %, for all \((x,y)\in\mathbb{R}^{n+m}\) we have
            \((x,y)\in\sol\) if and only if \((u,x,y)\in\wh{\cS}^\star\) for \(u=Ax-b\in\RR^m_+\). Indeed, if \((x,y)\in\sol\) and \(u=Ax-b\), then immediately \(0=b-Ax-u\) and \(0\in \partial f(x)+A^\top y\), by~\eqref{eq:admm-subdifferentials}. Moreover, by feasibility and complementary slackness we have \(y\geq 0\), \(u\geq 0\), and \(\langle y,u\rangle=0\). It follows that \(-y\in N_{\R^m_+}(u)\), or equivalently, \(0\in N_{\R^m_+}(u)+y\). Thus, according to~\eqref{eq:admm-subdifferentials}, \((u,x,y)\in\wh{\cS}^\star\). On the other hand, if \((u,x,y)\in\sol\), then \(u=b-Ax\),  \(0\in\partial f(x)+A^\top y\), and \(-y\in\N_{\R^m_+}(u)\), by~\eqref{eq:admm-subdifferentials}. Hence, \(y\in\R^m_+\) and \(y_j>0\) implies \(u_j=0\) for any \(j\in[m]\), whence \(-u\in N_{\R^m_+}(y)\). Thus, accordin to~\eqref{eq:admm-subdifferentials}, \((x,y)\in\sol\), establishing the claim. Invoking \((a)\) we obtain
             \begin{equation}
             \label{eq:admmDistSolRelation}
                \dist_{\widehat{P}}(w^0,\widehat{\cS}^\star)=\dist_{P}(z^0,\sol)\,
            \end{equation}
            and \(\|w^k-w^\star\|_{\widehat{P}}=\|z^k-\zs\|_{P}\). Hence, Step 1 implies that Assumption~\ref{assumption:finiteTimeIdentification} \((i)\) holds with \(\gamma=1\). %To establish 
            
            Next we turn to Assumption~\ref{assumption:rapidLocalConvergence}. Combining Step 1 and~\eqref{eq:admmDistSolRelation} we derive that it suffices to show
             \begin{equation}\label{eq:connectingthedots} \dist_{P^\dagger}(0,\sub(z^k)\cap\range(P))\leq\dist_{\widehat{P}^\dagger}(0,\widehat{\sub}(w^k)\cap\range(\widehat{P}))
              \end{equation}
              to establish item $(i)$ of Assumption~\ref{assumption:rapidLocalConvergence}.
            Note that \(\widehat{P}=\widehat{M}\widehat{M}^\top\) and \(P=MM^\top\), where 
             \begin{equation}
             \label{eq:admmMatrices}
                \widehat{M}^\top=\begin{bmatrix}
                    0 &
                    \sqrt{\eta}A^\top &
                    \sqrt{\eta}^{-1}I
                \end{bmatrix}\quad\text{and}\quad M^\top=\begin{bmatrix}
                    \sqrt{\eta}A^\top &
                    \sqrt{\eta}^{-1}I
                \end{bmatrix}\,.
             \end{equation}
             The columns of \(\widehat{M}\) and \(M\) are linearly independent.
            Applying Lemma~\ref{prop:cholesky} \((i)\) we derive
             \begin{equation}
             \label{eq:lookatthoseranges}
             \begin{aligned}
                \range(\widehat{P})& =\range(\widehat{M})=\{(0,\eta A^\top v,v):v\in\RR^{m}\},\quad \text{and}\\ %\quad
               \range(P)&=\range(M)=\{(\eta A^\top,v):v\in\RR^{m}\}\,.
            \end{aligned}
            \end{equation}
             Let \(\wh{v}\in\widehat{\sub}(w^k)\cap\range(\wh{P})\) such that \(\|\wh{v}\|_{P^\dagger}=\dist_{P^\dagger}(0,\wh{\sub}(w^k)\cap\range(\wh{P}))\). Since \(\wh{v}\in\range(\wh{P})\), from~\eqref{eq:lookatthoseranges} there exists \(v\in\RR^m\) such that \(\wh{v}=\wh{M}v\). Moreover, since \(\wh{v}\in\wh{\sub}(w^k)\), the \(y\)-coordinates of \(\wh{M}v\)---recall~\eqref{eq:admmMatrices}---and of~\eqref{eq:whataninclusion} are equal: we obtain \(\sqrt{\eta}^{-1}v=b-Ax^k-u^k\). Then,
             \begin{align*}
                \dist_{P^\dagger}(0,\wh{\sub}(w^k)\cap\range(\wh{P}))&=\|\wh{M}v\|_{P^\dagger} & \text{(Definition of \(\wh{v}\) and \(v\))}\\
            &=\|v\|_2 & \text{(Lemma~\ref{prop:cholesky} \((ii)\))}\\
            &=\sqrt{\eta}\|b-Ax^k-u^k\|_2\,.
             \end{align*}
            %  Invoking \eqref{eq:whataninclusion} we obtain \(\widehat{\sub}(w^k)\cap\range(\widehat{P})\neq\emptyset\), and so by \eqref{eq:admm-subdifferentials} and~\eqref{eq:lookatthoseranges} we get \(-y^k\in N_{\RR^m_+}(u^k)\). Hence, \(y^k\in\RR^m_+\) and \(y^k_j>0\) implies \(u^k_j=0\) for any \(j\in [m]\), whence \(-u^k\in N_{\RR^m_+}(y)\).Furthermore, from \eqref{eq:lookatthoseranges} we obtain that there exists a $v\in \RR^m$ such that
            % \[
            %     \dist_{\widehat{P}^\dagger}(0,\widehat{\sub}(w^k)\cap\range(\widehat{P})) = \sqrt{\eta} \| \widehat{M} v \|_{\widehat{P}^\dagger}= \sqrt{\eta} \|v\|_2= \sqrt{\eta} \|b-Ax^k-u^k\|_2\,,
            % \]
            % where the second equality follows from Lemma~\ref{prop:cholesky} and the last equality follows from \eqref{eq:admm-subdifferentials} and the fact that $\sqrt{\eta} \wh{M} v \in \widehat \sub (w^k)$, which reduces to $v = b - Ax^k - u^k$ in the $y$-coordinate.
            Also, invoking \eqref{eq:whataninclusion} we obtain \(\widehat{\sub}(w^k)\cap\range(\widehat{P})\neq\emptyset\), and so by \eqref{eq:admm-subdifferentials} and~\eqref{eq:lookatthoseranges} we get \(-y^k\in N_{\RR^m_+}(u^k)\). Hence, \(y^k\in\RR^m_+\) and \(y^k_j>0\) implies \(u^k_j=0\) for any \(j\in [m]\), whence \(-u^k\in N_{\RR^m_+}(y^k)\).
            Using an analogous reasoning for $P$ and $\sub$ we derive
            \begin{align*}
              \dist_{P^\dagger}(0,\sub(z^k)\cap\range(P))&=\sqrt{\eta}\,\inf\{\|b-Ax^k+\zeta\|_2:\zeta\in N_{\RR_+^m}(y^k)\} & %\text{(Lemma~\ref{prop:cholesky} \((ii)\))}
              \\
                &\leq \sqrt{\eta}\, \|b-Ax^k-u^k\|_2 &  \text{(\(-u^k\in N_{\RR_+^m}(y^k)\))}\\
                &=  \dist_{\widehat{P}^\dagger}(0,\widehat{\sub}(w^k)\cap\range(\widehat{P}))\,.
            \end{align*}
            Thus, \eqref{eq:connectingthedots} is holds true and so Assumption~\ref{assumption:rapidLocalConvergence} \((i)\) holds with \(\gamma=1\).\\
            
            \noindent\textbf{Step 3.} Next we establish item $(ii)$ of Assumption~\ref{assumption:asymptoticIdentificationWeak}. Equivalentely, we wish to show that the dual update \(y^{k+1}\) can be interpreted as a projected dual ascent update on the {\em original problem}~\eqref{eq:minimaxProblem}, i.e. \(y^{k+1}=\proj_{\R_+^{m}}(y^{k}+\eta(Ax^{k}-b))\). Note that
            \begin{equation}
            \label{eq:admmKer}
                \ker(P)=\{(x,y)\in\RR^{n+m}:y+\eta Ax=0\}\,.
           \end{equation}
           Hence, recalling \eqref{eq:ohthatP} we get
             \begin{equation}
             \label{eq:admmUForm}
                u^{k+1}=\proj_{\RR^m_+}\left(-\eta^{-1}y^k-Ax^{k}+b\right)=\proj_{\RR^m_+}\left(-\eta^{-1}y_P^k-Ax_P^{k}+b\right)\,.
             \end{equation}
             Fix \(j\in[m]\). Then, we have two cases
            \begin{itemize}
                \item[\((a)\)] Suppose \(0> (y^{k}+\eta (Ax^{k}-b))_j\). Then, we derive from~\eqref{eq:admmUForm} that
                \[
                    u^{k+1}_j=\proj_{\R_+}(-\eta^{-1}y^k-(Ax^{k} -b))_j=(-\eta^{-1}y^k-(Ax^{k} -b))_j\,.
                \]
                Thus,
                \(
                    y^{k+1}_j=(y^{k}+\eta^{-1}(Ax^{k}-\eta y^{k}-Ax^{k}))_j=0\,.
                \)
                \item[\((b)\)] Suppose that \(0\leq(y^{k}+\eta(Ax^{k}-b))_j\). Then, we obtain invoking from~\eqref{eq:admmUForm} that
                \[
                    u_j=\proj_{\R_m^+}(-\eta^{-1}y^{k}-(Ax^{k}-b))_j=0\,.
                \]
                Therefore,
                \(
                    y^{k+1}_j=y^{k}_j+\eta (Ax^{k}-b)_j\,.
                \)
            \end{itemize}
            Thus, Assumption~\ref{assumption:asymptoticIdentificationWeak} \((ii)\) holds.

            Next, we will show that items $(iii)$ and $(iv)$ of
            Assumption~\ref{assumption:asymptoticIdentificationWeak} hold. To establish Assumption~\ref{assumption:asymptoticIdentificationWeak} \((iii)\), consider \(\YY=\R^n\times\R^m_+\) and let \(z=(x,y)\) and \(\wh{z}=(\wh{x},\wh{y})\) in \(\YY\). Since \(y_P+y_{P^\perp}\geq 0\), we obtain from~\eqref{eq:admmKer} that \(y_P\geq \eta Ax_{P^\perp}\). Then, for any \(j\in\indexN\) (recall the definition~\eqref{eq:indexSets}) we have
           \begin{align*}
                |A_jx-A_j\xs|&=|A_jx_P+A_jx_{P^\perp}-A_j\xs|\\
                &\leq |A_jx_P+\eta^{-1}(y_P)_j-(A_j\xs+\eta^{-1}\ys_j)| & \text{(\(\ys_j=0\) since \(j\in\indexN\))}\\
                & \leq \|A_j\|_{2}\|x_P-\xs\|_2+\eta^{-1}\|y_P-\ys\|_2 & \text{(Triangle Inequality)}\\
                &\leq (\|A_j\|_{2}+\eta^{-1})\|z_P-\zs\|_2 \\
                &\leq (\|A_j\|_{2}+\eta^{-1})\lambdaminp(P)\|z_P-\zs\|_P\,. & \text{(Proposition~\ref{prop:normEquivalence})}
           \end{align*}
           On the other hand, if \(\wh{z}_{P^\perp}=z_{P^\perp}\) then, from Proposition~\ref{prop:normEquivalence} we obtain
           \begin{align*}
             A_jx-A_j\wh{x}&=A_j(x_P-\wh{x}_P)\leq \|A_j\|_{2}\|x_P-\wh{x}_P\|_2\leq \|A_j\|_{2}\|z_P-\wh{z}_P\|_2\leq\|A_j\|_{2}\lambdaminp(P)\|z-\wh{z}\|_P
           \end{align*}
           Thus, Assumption~\ref{assumption:asymptoticIdentificationWeak} \((iii)\) holds with \(\Lxj\leq (\eta^{-1}+\|A_j\|_{2})\lambdaminp(P)^{-\frac{1}{2}}\).

Finally, we show that Assumption~\ref{assumption:asymptoticIdentificationWeak} $(iv)$ holds with $p = 1.$ Consider any initial iterates \(z^0=(x^0,y^0)\), \(\wh{z}^0=(\wh{x}^0,\wh{y}^0)\), and index \(j\in[m]\). Then,
             \begin{equation}
            \label{eq:admmYBound}
                \begin{aligned}
                    |y^1_j-\wh{y}^1_j|&= |\proj_{\RR_+}(y_j^{0}+\eta(A_jx^{0}-b_j))-\proj_{\RR_+}(\wh{y}^0_j+\eta(A_j\wh{x}^0-b_j))| & \text{(Assumption~\ref{assumption:asymptoticIdentificationWeak} \((ii)\))}\\
                    &= |\proj_{\RR_+}((y^{0}_P)_j+\eta(A_jx^{0}_P-b_j))-\proj_{\RR_+}((\wh{y}^0_P)_j+\eta(A_j\wh{x}^0_P-b))| & \text{(From~\eqref{eq:admmKer})}\\
                    &\leq |(y^{0}_P)_j+\eta(A_jx^{0}_P-b_j)-((\wh{y}^0_P)_j+\eta(A_j\wh{x}^0-b_j))| & \text{(Non-expansiveness)}\\
                    &\leq \|y^{0}_P-\wh{y}^0_P\|_2+\eta\|A_j\|_{\rm{op}}\|x^{0}_P-\wh{x}^0_P\|_2 & \text{(Triangle Inequality)}\\
                    &\leq (1+\eta\|A_j\|_{\rm{op}})\|z^{0}_P-\wh{z}^0_P\|_2\\
                    &\leq (1+\eta\|A_j\|_{\rm{op}})\lambdaminp(P)^{-\frac{1}{2}}\|z^{0}-\wh{z}^0\|_P & \text{(Proposition~\ref{prop:normEquivalence})}.\\
                \end{aligned}
            \end{equation}
            So Assumption~\ref{assumption:asymptoticIdentificationWeak} \((iv)\) holds with \(\step=1\) and \(\Lyj\leq(1+\eta\|A_j\|_{\rm{op}})\lambdaminp(P)^{-\frac{1}{2}}\).

\subsubsection{Proof of Proposition~\ref{prop:EGM}}\label{appendix:EGM}
%Since \(\eta<L^{-1}\),
Invoking \cite[Appendix I, Theorem 9]{cai2022finite}, we derive that for any \(k\in\NN\) the iterates of EGM satisfy
            \begin{equation}
            \label{eq:egdaSublinearRates}
                \max(\eta\dist_2(0,\partial\sub(z^k)),  2^{-1}\|z^k-z^{k+1}\|_2, \|z^k-\tilde{z}^{k}\|_2)\leq \frac{3}{\sqrt{1-(\eta L)^2}}\frac{\dist_2(z^0,\sol)}{\sqrt{k}}\,.
            \end{equation}
            Therefore, EGM satisfies Assumptions~\ref{assumption:rapidLocalConvergence} \((i)\) and~\ref{assumption:finiteTimeIdentification} \((i)\) with \(P=I\) and \(\gamma=3(1-(\eta L)^2)^{-\frac{1}{2}}\).
            To verify Asusmption~\ref{assumption:asymptoticIdentification} \((i)\), first note that the iterates of EGM are bounded~\cite[Appendix G, Corollary 1]{cai2022finite}.  Then, its set of accumulation points, denoted as \(\mathbb{A}\), is nonempty. We will see that \(\mathbb{A}\) is a singleton. By Assumption~\ref{assumption:problem} \((i)\), \(\sub\) has closed graph~\cite[Theorem 24.4]{rockafellar1997convex}. Then, using~\eqref{eq:egdaSublinearRates}, we derive \(0\in\sub(z)\) for all \(z\in\mathbb{A}\) and so
            \(\mathbb{A}\subseteq\sol\). Further, EGM is Fèjer monotone \cite[Appendix G, Lemma 1]{cai2022finite}, that is,
            \begin{equation}
            \label{eq:egdaFejer}
            \|z^{k+1}-z\|_2\leq \|z^k-z\|_2\quad\text{for all}\quad z\in\sol\,.
            \end{equation}
            Next, we show that there is only one accumulation point. Suppose seeking contradiction that there exist \(z, z'\in\mathbb{A}\) such that \(z \neq z'\), and define \(\Delta=\|z-z'\|_2/2\). Since \(z\in\mathbb{A}\) there exists \(k_0\in\NN\) such that \(\|z^{k_0}-z\|_2\leq \Delta\). Furthermore, from~\eqref{eq:egdaFejer} we obtain \(\|z^{k}-\tilde{z}^1\|_2\leq \|z^{k_0}-z\|_2\leq \Delta\) for all \(k\geq k_0\). Then, by the triangular inequality, for all \(k\geq k_0\),
            \[
              \|z' - z^k\|_2\geq \|z'-z\|_2-\|z-z^k\|_2\geq 2\Delta -\Delta=\Delta\,.
            \]
            It follows that \({z}'\) is not an accumulation point, yielding a contradiction. Thus, \(\mathbb{A}\) is a singleton and EGM satisfies Assumptions~\ref{assumption:asymptoticIdentification} \((i)\) and~\ref{assumption:rapidLocalConvergence} \((ii)\). %\hnote{Do we need Assumption 3 or 3* for EGM?}\pil{Yes, we need Assumption 3 not 3*. I changed it from 3* to 3. Done.}
            Further, EGM satisfies Assumption~\ref{assumption:finiteTimeIdentification} \((ii)\) thanks to~\eqref{eq:egdaFejer}. By definition, EGM satisfies Assumption~\ref{assumption:asymptoticIdentification} \((ii)\). To check Assumption~\ref{assumption:asymptoticIdentification} \((iii)\) note that, by Assumption~\ref{assumption:problem} \((i)\), the functions \(g_j\) are locally Lipschitz~\cite[Theorem 10.4]{rockafellar1997convex}. Then, since \(P=I\), Assumption~\ref{assumption:asymptoticIdentification} \((iii)\) holds with  \(\Lyj=1\) and \(\Lxj\) as a local Lipschitz constant of \(g_j\) over \(\ballt{\zs}{t}\).

            % As an instance of the Mirror-Prox method, EGDA is convergent for \(\eta<L^{-1}\)~\cite{nemirovski2004prox}, whence it satisfies Assumption~\ref{assumption:asymptoticIdentificationWeak} \((i)\).
\subsubsection{An auxiliary result}
        We prove an auxiliary result used in the proof of Proposition~\ref{prop:admm}.
        \begin{lemma}
            \label{prop:cholesky}
                Let \(S\in\mathcal{S}_+^{q}\) and \(M\in\RR^{q\times k}\) with linearly independent columns such that \(S=MM^\top\). Then, the following two hold true.
                \begin{itemize}
                    \item[\((i)\)] The ranges of these two matrices are the same \(\range(S)=\range(M)\).
                    \item[\((ii)\)] For any \(v\in\RR^k\) and \(u=Mv\) we have \(\|u\|_{S^\dagger}=\|v\|_2\).
                \end{itemize}
            \end{lemma}

            \begin{proof}
                 Since \(M\) has linearly independent columns, \(M^\top\) is surjective. Hence, \(\range(S)=\range(M)\), whence there exists \(v\in\RR^k\) such that \(u=Mv\). Since \(M\) has linearly independent columns, \(M^\dagger=(M^\top M)^{-1}M^\top\) notice that
                \begin{align*}
                    S^\dagger=(MM^\top)^\dagger=(M^\dagger)^\top M^\dagger=((M^\top M)^{-1}M^\top)^\top(M^\top M)^{-1}M^\top=M(M^\top M)^{-\top}(M^\top M)^{-1}M^\top\,.
                \end{align*}
                Therefore, for any \(v\in\RR^k\) and \(u=Mv\) we have
                \begin{align*}
                    \|u\|_{S^\dagger}^2&=\langle u,P^\dagger u\rangle_2
                    =\langle Mv,  M(M^\top M)^{-\top}(M^\top M)^{-1}M^\top Mv\rangle_2
                    % &=\langle (L^\top L)^{-1}(L^\top L)u,(L^\top L)^{-1}(L^\top L)u\rangle_2\\
                    =\langle v,v\rangle_2
                    = \|v\|_2^2\,.
                \end{align*}
            \end{proof}
\section{Missing proofs from Section~\ref{section:manifoldIdentification}}
\label{appendix:manifoldIdentification}

% \begin{definition}
%         \label{def:fixedPointMap}
%         Let \(T:\RR^{n+m}\to\RR^{n+m}\), \(T_x:\RR^{n+m}\to\RR^n\), and \(T_y:\RR^{n+m}\to\RR^m\) such that \(T(z^k)=(T_x(z^k),T_y(z^k))=(x^{k+1},y^{k+1})=z^{k+1}\) for all \(k\in\NN_0\) and any initial iterate \(z^0\in\RR^{n+m}\).
%         Furthermore, for any \(\step\in\NN\), \(T^\step(z)=(\Tx(z),\Ty(z))\) denotes the \(\step\)-times application of \(T\) over \(z\). The map \(T\) represents the primal-dual update introduced in~\eqref{eq:algorithm}.
%      \end{definition}
In this section, we prove our identification results, that is, Theorem~\ref{theorem:asymptoticIdentification},~\ref{theorem:finiteTimeIdentificationLinear}, and Theorem~\ref{theorem:finiteTimeIdentification}.
To establish these results we derive an algebraic charaterization of the radious of active-set stability.

\subsection{Radius of active-set stability}
\label{appendix:degeneracyMetric}
In this section, we show that the geometric definition of the radius of active-set stability introduced in~\eqref{eq:degeneracy-metric} can be characterized algebraically through the Lipschitz constants of the constraints and the primal-dual solution to which the algorithm converges.
To this end, we define the `optimal' Lipschitz moduli for a $P$-ball of radius $t$ as stated in Assumption~\ref{assumption:asymptoticIdentificationWeak}. Formally, for all $j \in [m]$ we define the {\em Lipschitz modulus} functions \(\Lx:(0,+\infty)\to \RR\) and \(\Ly:(0,+\infty)\to\RR\) given by
\begin{equation}
\label{eq:varphiPsi}
    \Lxt{t}=\sup_{z=(x,y)\in\Bt}\frac{\left(g_j(x)-g_j(\xs)\right)_+}{\|z-\zs\|_P}\quad\text{and}\quad \Lyt{t}=\sup_{z=(x,y)\in\Bt}\frac{(\ys_j-y_j^\step)_+}{\|z-\zs\|_P}\,,%\quad\text{and}\quad\Ly(t)=\supremum{z\in\cballP{\zs}{t}}\frac{|\ys_j-\Ty(z)|}{\|z-\zs\|_P}\,.
\end{equation}
%\hnote{What prevents $\|z-\zs\|_P=0$ for sd $P$?}\pil{I changed the definition of \(\YY_t\) to prevent it. Done.}
where $(\cdot)_+$ extracts the nonnegative part of its input and \(\Bt=(\cballP{\zs}{t}\cap\YY)\backslash(\{\zs\}+\ker(P))\). Recall \(\YY\) is defined as part of Assumption~\ref{assumption:asymptoticIdentificationWeak} \((iii)\).%\hnote{Where is $\cballP{\zs}{t}$ defined?}\pil{Done. I included the definiton of \(\ballP{z}{t}\) and \(\cballP{z}{t}\) in the preliminaries.}  
%The first bound of Assumption~\ref{assumption:asymptoticIdentificationWeak} \((iii)\) guarantees \(\Lxt{t}\) to be finite; the second constraint is technical and it serves to show that \(t\mapsto\Lyt{t}\) is continuous. %The relation of \(\Lyt{t}\) with Assumption~\ref{assumption:asymptoticIdentificationWeak} \((iv)\) is analogous.

Now we extend the definition radius defined in~\eqref{eq:degeneracy-metric} to account for the set \(\YY\) introduced in Assumption~\ref{assumption:asymptoticIdentificationWeak}. Define
\begin{equation}
\label{eq:delta_geo}
  \delta_{G} := \sup \Big\{t\in(0,\infty)\,\,\big|\, \,\text{For all } (x,y)\in\ballP{\zs}{t}\cap\YY 
  \text{ we have }\,G_\indexN(x) < 0, \text{ and } y_{\indexBs}^\step > 0\Big\}\,,
\end{equation}
We consider an alternative characterization of this radius, dubbed $\delta_A$, defined implicitly via
\begin{equation}
\label{eq:delta_alg}
    \delta_{A} = \min\left\{\min_{j\in\indexN}\frac{-g_j(\xs)}{\Lxt{\delta_A}},\min_{j\in\indexBs}\frac{\ys_j}{\Lyt{\delta_A}}\right\}\,,
\end{equation}
where $\Lxt{\cdot}$ and $\Lyt{\cdot}$ are the optimal Lipschitz modulus functions defined in \eqref{eq:varphiPsi}. This gives a well-defined notion of radius. We write \(\Lxjdo\) and \(\Lyjdo\) as shorthands for \(\Lxt{\delta_A}\) and \(\Lyt{\delta_A}\).
%\hnote{Please find ways to simplify the proof. You just need continuity of $\bar{L}$ in $t$ and it is a nonincreasing function to show the uniqueness. The proof is stated in an overly complicated way.}\pil{Done. I included a summary of the main idea of the proof at its beginning. Also, I simplified its last step.}
\begin{lemma}\label{claim:delta-well-defined}
    Suppose that Assumption~\ref{assumption:asymptoticIdentificationWeak} holds and that \(\delta_G\) is finite. %\pil{In a previous version of the paper we had a characterization of the finiteness of \(\delta_G\). Should we include it (not as part of this lemma but as a separate result)?} 
    Then, the function $\Delta\colon (0, + \infty) \to \RR \cup \{+\infty\}$ given by
    \begin{equation}
\label{eq:Delta}
     \Delta(t)=\min\left\{\min_{j\in \indexN}\frac{-g_j(\xs)}{\Lxt{t}},\min_{j\in\indexBs}\frac{\ys_j}{\Lyt{t}}\right\}\,
\end{equation}
has a unique fixed point. Thus, $\delta_A$ is well-defined and finite.
\end{lemma}

In turn, $\delta_A$ and $\delta_G$ are closely related. The following is the main result of this section. 
\begin{proposition} 
\label{prop:deltaEquivalence}
Suppose that Assumption~\ref{assumption:asymptoticIdentificationWeak} holds and that $\delta_G$ is finite. Then, 
$$\delta_A \le \delta_G.$$
Moreover, equality holds if Assumption~\ref{assumption:asymptoticIdentificationWeak} holds with \(\YY=\RR^{n+m}\) and $\step = 0$.
%with \(\step=0\) and that \(\delta_G<\infty\). Then, 
%$$\delta_G =\delta_A.$$ 
\end{proposition}
We remark that all algorithms satisfy Assumption~\ref{assumption:asymptoticIdentificationWeak} with \(\YY=\RR^{n+m}\) and $\step = 0$, except for ADMM, which instead requires \(\YY=\R^n\times\R^m_+\) and $\step = 1$. 
%Thus, slightly abusing notation, in subsequent sections we will use the symbol $\delta$ to denote the algebraic characterization \(\delta_A\)\pil{This is no longer true. We use \(\delta_A\).}.
The rest of this section is devoted to establishing Proposition~\ref{prop:deltaEquivalence}. We start with a proof of Lemma~\ref{claim:delta-well-defined}, since it contains some of the necessary ingredients for establishing Proposition~\ref{prop:deltaEquivalence}.

\begin{proof}[Proof of Lemma~\ref{claim:delta-well-defined}]

We need to show that there exists \(t^\star\in(0,\infty)\) such that \(\Delta(t^\star)=t^\star\). To prove this, first we show that \(t\mapsto \Delta(t)\) is continuous and {\em non-increasing} in its domain, which we prove is an unbounded open interval. Clearly, \(t\mapsto t\) is also continuous but {\em increasing} without bounds. Then, we show that there exists \(t\in \RR\) such that \(\Delta(t)\geq t\). Thus, by continuity and monotonicity,  \(\Delta\) and the identity function must intersect at some point greater than or equal to \(t\). To formalize this, we divide the proof into four steps. The first three steps derive properties of $\Lx, \Ly,$ and $\Delta$, and the final step uses these properties to derive the conclusion.
    \paragraph{Step 1: Continuity  of \(\Lx\) and \(\Ly\).} 
 We will see that, under Assumption~\ref{assumption:asymptoticIdentificationWeak}, \(\Lx\) is finite and locally Lipschitz-continuous if \(j\in\indexN\). The same applies to \(\Ly\) if \(j\in\indexN\cup\indexBs\) .

Finiteness follows directly from Assumption~\ref{assumption:asymptoticIdentificationWeak} \((iii)\) and \((iv)\). 
To establish the continuity of \(\Lx\) for any fixed \(j\in\indexN\), define the auxiliary function \(h\colon\mathbb{R}^{n+m}\backslash(\{\zs\}+\ker(P))\to\RR\) given by
    \[
        h(z)=\frac{g_j(x)-g_j(\xs)}{\|z-\zs\|_P} \quad \text{for any} \quad \; z=(x,y)\in \mathbb{R}^{n+m}\backslash(\{\zs\}+\ker(P))\;.
    \]
    % From Assumption~\ref{assumption:problem} \((i)\), \(h\) is continuous.
    Also, for any \(s_0<s_1\in(0,\infty)\), define the `{$P$-shell}' 
    \[
        \mathsf{S}(s_0,s_1)=\{z\in\YY:s_0\leq \|z-\zs\|_P\leq s_1\}\,.
    \]
    Fix \(t_0\in(0,\infty)\). First, we will show that \(h\) is locally Lipschitz continuous uniformly along affine sets parallel to \(\range(P)\).
    % From Assumption~\ref{assumption:problem} \((i)\), \(g_j\) is Lipschitz continuous in the compact set \(S(t_0-\zeta,t_0+\zeta)_P\). Then, there exists \(\xi\in(0,\infty)\) such that if \(|t-t_0|\leq\xi\) then \(|h(t)-h(t_0)|\leq 2^{-1}\eps\). 
    %and \(\eps>0\). 
    Define \(s_0\in(0,t_0)\) and \(s_1=2t_0-s_0\); note that \(t_0-s_0=s_1-t_0>0\).
     For any \(z,z'\in\mathsf{S}(s_0,s_1)\) such that \(z_{P^\perp}=z'_{P^\perp}\), define \(\omega=\|z-\zs\|_P\) and \(\zeta=g_j(x)-g_j(\xs)\) and analogously $\omega'$ and $\zeta'$. By the triangle inequality and Assumption~\ref{assumption:asymptoticIdentificationWeak} \((iii)\),
    \begin{align}
    \begin{split}
        \label{eq:varphiBounds}
         &s_0\leq \omega \leq s_1 , \qquad s_0 \leq \omega' \leq s_1,\qquad|\omega - \omega'|\leq\|z-z'\|_P, \\ &\max\{|\zeta|,
         |\zeta'|\}\leq L^x_{s_1 j} s_1,\quad\text{and}\quad|\zeta-\zeta'|\leq L^x_{s_1j}\|z-z'\|\,.
         \end{split}
    \end{align}
    Therefore,
    \begin{equation}
    \label{eq:hBounds}
        \begin{aligned}
            |h(z)-h(z')|
            %&=\left|\frac{\zeta_1}{\Delta_1}-\frac{\zeta_2}{\Delta_2}\right|\\
            &= (\omega\omega')^{-1}|\omega'\zeta-\omega\zeta'| & \text{(Definitions of \(h\), \(\omega\), and \(\zeta\))}\\
            &\leq (\omega\omega')^{-1}(\omega'|\zeta-\zeta'|+\zeta'|\omega-\omega'|) & \text{(Triangle inequality)}\\
            %&\leq s_0^{-2}(s_1\Lxj\|z^1-z^2\|_P+s_1\Lxj\|z^1-z^2\|_P)\\
            &\leq2 L^x_{s_1j} s_1s_0^{-2}\|z-z'\|_P\,. & \text{(From~\eqref{eq:varphiBounds})}
        \end{aligned}
    \end{equation}
    Hence \(h\) is uniformly Lipschitz continuous over %any affine subset of
    \(\mathsf{S}(s_0,s_1)\). %parallel to \(\range(P)\).
    Next, we establish the continuity of \(\Lx(\cdot)\).
    % In particular, for \(\xi=(2s_1\Lxj)^{-1}s_0^2\eps\), if \(\|z^1-z^2\|_P\leq \xi\) then \(|h(z^1)-h(z^1)|\leq \eps\). 
    Let \(t\in(0,\infty)\) such that \(|t-t_0|\leq |s_0-t_0|\). Denote \(\ot=\max\{t,t_0\}\), \(\ut=\min\{t,t_0\}\), and for any \(\eps>0\) define \(z^\eps\in \mathsf{S}(\ut,\ot)\) such that
    \(
        h(z^\eps)+\eps\geq\sup\{h(z):z\in\cS(\ut,\ot)\}\,.
    \)
    Then,
    \begin{equation}
    \label{eq:varphiSliceEps}
        \Lx(\ot)=\max\left\{\Lx(\ut),\sup_{z\in \cS(\ut,\ot)}h(z)\right\}\leq \max\left\{\Lx(\ut),h(z^\eps)\right\}+\eps\,,
    \end{equation}
    where the first equality follows by definition.
    %\md{Consider changing the notation for this (see comment on Page 34)}\pil{Done.} 
    Let \(u^\eps\in\cballP{\zs}{\ut}\) such that
    \(
        u^\eps=\zs+\ut\|z^\eps-\zs\|_P^{-1}(z_P^\eps-\zs)+z^\eps_{P^\perp}\,.
    \)
    Since \(\zs\in\range(P)\),
    \[
       \|z^\eps-u^\eps\|_P=(1-\ut\|z^\eps-\zs\|_P^{-1})\|z^\eps-\zs\|_P=\|z^\eps-\zs\|_P-\ut\leq \ot-\ut\,.
    \]
    By construction,  \(u^\eps_{P^\perp}=z^\eps_{P^\perp}\). Then, from~\eqref{eq:hBounds}, i.e. the Lipschitz continuity of \(h\), we obtain
    \begin{equation}
    \label{eq:hPEpsBound}
        |h(z^\eps)-h(u^\eps)|\leq 2s_1\Lxj s_0^{-2}\|z^\eps-u^\eps\|_P\leq 2\Lxj s_1s_0^{-2}(\ot-\ut)\,.
    \end{equation}
    Also, since \(u^\eps\in\cballP{\zs}{\ut}\), from \(\Lx\) definition we have \(h(u^\eps)\leq \Lx(\ut)\). It follows that
    \begin{align*}
        |\Lx(\ot)-\Lx(\ut)|&=\Lx(\ot)-\Lx(\ut) & \text{(Definition of \(\Lx\), \(\ot\), and \(\ut\))}\\
        &\leq\max\left\{\Lx(\ut), h(z^\eps)\right\}-\Lx(\ut)+\eps& \text{(From~\eqref{eq:varphiSliceEps})}\\
        &\leq \max\left\{0, h(z^\eps)-h(u^\eps)\right\}+\eps\\
        &\leq 2 L^x_{s_1j} s_1s_0^{-2}(\ot-\ut)+\eps\,. & \text{(From~\eqref{eq:hPEpsBound})}
    \end{align*}
    Since \(\eps>0\)  arbitrary, \(|\Lx(t)-\Lx(t_0)|\leq 2L^x_{s_1j} s_1s_0^{-2}|t-t_0|\). Moreover, since \(t_0\in(0,\infty)\), \(\xi\in(0,t_0)\), and \(t\in[t_0-\xi,t_0+\xi]\) are aribitrary, \(\Lx\) is locally Lipschitz continuous. The proof for the functions \(\Ly\) with \(j\in\indexN\cup\indexBs\) is completely analogous.
%\md{Split previous step into two steps: continuity and strict positivity of the Lipschitz functions.}\pil{Done.}
\paragraph{Step 2: Positivity of \(\Lx\) and \(\Ly\).}
The functions \(\Lx\) and \(\Ly\) take non-negative values and are non-decreasing by construction. Furthermore, at least one of them is strictly positive for all \(t\) sufficiently large. This is because, since \(\delta_G\) is finite, there exists \(z^0=(x^0,y^0)\in\RR^{n+m}\) and \(j\in\indexN\cup\indexBs\) such that \(g_j(x^0)>g_j(\xs)\) or \(y^\step_j<\ys_j\). Then, for \(t_{\rm{dom}}:=\|z^0-\zs\|_P\), \(\Lx(t_0)\) or \(\Ly(t_0)\) is strictly positive. Furthermore, by monotonicity, \(\Lx(t)\) or \(\Ly(t)\) is strictly positive for all \(t\geq t_{\rm{dom}}\).

\paragraph{Step 3: Properties of $\Delta$.} The function \(\Delta\) is positive, non-increasing, and locally Lipschitz-continuous on its domain, which is equal to an interval of the form \((t,\infty)\) with \(t\geq0\).
     To prove this, first recall that the functions \(\Lx\) and \(\Ly\) take non-negative real values.%\md{You need these functions to be strictly positive to conclude Lipschitzness. If the functions $g_{j}$ are constant this might fail.}\pil{Clarified. Done.} 
     Moreover, \(-g_j(\xs)>0\) for all \(j\in\indexN\) and \(\ys_j>0\) for all \(j\in\indexBs\). Hence, \(\Delta\) is positive. Also, the functions \(\Lx\) and \(\Ly\) non non-decreasing, whence \(\Delta\) is non-increasing. Define 
    \begin{equation}
    \label{eq:DeltaDomainBound}
        t_{\rm{sup}}:=\sup\{t\in(0,\infty):\Lx(t)=0 \text{ for all }j\in\indexN\text{ and }\Ly(t)=0\text{ for all }j\in\indexBs \}\,.
    \end{equation}
    From Step 2 we obtain that \(t_{\rm{sup}}\leq t_{\rm{dom}}<\infty\). Since \(\Delta\) is non-increasing, \(\Delta(t)<\infty\) for all 
    \(t> t_{\rm{sup}}\). Furthermore, by Step 1 the functions \(\Lx\) and \(\Ly\) are continuous, whence if the supremum is not equal to zero in~\eqref{eq:DeltaDomainBound}, then it is attained and \(\Delta(t_{\rm{sup}})=\infty\). It follows that \(\dom(\Delta)=(t_{\rm{sup}},\infty)\). Moreover, \(\Delta\) is locally Lipschitz-continuous on \(\dom(\Delta)\) since it is the minimum of compositions of locally Lipschitz-continuous functions, that is, \(\Delta= \min\{\min_{j\in\indexN}h^{\indexN}_j\circ\Lxjdo,\min_{j\in\indexBs} h^{\indexBs}_j\circ \Lyjdo\}\) where \(h^{\indexN}_j(t)=-g_j(\xs)/t\) and \(h^{\indexBs}_j(t)=\ys_j/t\).
\paragraph{Step 4: Fixed point existence and uniqueness.}
% \md{Pedro, in Step 3 you reused a bunch of the same notation from previous steps, but you endowed it with new semantic meaning. DO NOT DO THAT. The reader can easily get confused between the definition of the symbols at different steps.}
Let \(F\colon(0,\infty)\to[-\infty,\infty)\) be such that \(F(t)=t-\Delta(t)\). Note that \(\Delta(t^\star)=t^\star\) if and only if \(F(t^\star)=0\). The function \(F\) is increasing unboundedly and it is continuous, because \(\Delta\) is non-increasing and continuous (see Step 3). Then, by the Intermediate Value Theorem, to verify the existence of \(t^\star\in (0,\infty)\) such that  \(F(t^\star)=0\) it suffices to show that exists \(t\in(0,\infty)\) with \(F(t)\leq 0\), or equivalently, \(t\leq \Delta(t)\). 
If \(t_{\rm{sup}}=0\) (see~\eqref{eq:DeltaDomainBound}), then \(t\leq \Delta(t)\) for any \(t\leq \min\{1,\Delta(1)\}\), because \(\Delta\) is non-increasing by Step 3. If \(t_{\rm{sup}}>0\), then \(\Delta(t)\to \infty\) when \(t\to t_{\rm{sup}}\). This is because the supremum is attained in~\eqref{eq:DeltaDomainBound}, as shown in Step 3, and the functions \(\Lx\) and \(\Ly\) are continuous, as established in Step 1. Hence, for any \(\varepsilon>0\) sufficiently small and \(t=t_{\rm{sup}}+\varepsilon\) we obtain \(t\leq \Delta(t)\). Thus, there exists \(t^\star\in(0,\infty)\) such that \(F(t^\star)=0\). Its uniqueness follows from the strict monotonicity of \(F\). 
This concludes the proof of Lemma~\ref{claim:delta-well-defined}.

\end{proof}

Armed with this lemma, we are now ready to establish Proposition~\ref{prop:deltaEquivalence}.
\begin{proof}[Proof of Proposition~\ref{prop:deltaEquivalence}]

    We start with a claim that we will be useful for future arguments. 
    \begin{claim}
\label{prop:LipchitzGuarantee}
Suppose that Assumption~\ref{assumption:asymptoticIdentificationWeak} holds and that \(\delta_G\) is finite. Then, for any \(\theta\in(0,1)\), and primal-dual point \(z^0=(x^0,y^0)\in \cballP{\zs}{\theta\delta_A}\cap\YY\) we have
\[
  -g_j(x^0)\geq (1-\theta)\Lxjdo\delta_A\, \quad\text{for all }j\in\indexN,\quad\text{and}\quad y^\step_j\geq(1-\theta)\Lyjdo\delta_A\quad\text{for all }j\in\indexBs\,.
\]
In particular,  any primal-dual point \(z^0=(x^0,y^0)\in \ballP{\zs}{\delta_A}\cap\YY\) satisfies \(G_\indexN(x^0)<0\) and \(y^\step_\indexB>0\). 
\end{claim}

\begin{proof}[Proof of Claim~\ref{prop:LipchitzGuarantee}]
By Claim~\ref{claim:delta-well-defined}, \(\delta_A\in(0,\infty)\) is well defined. Let \(z=(x,y)\in\cballP{\zs}{\theta\delta_A}\cap\YY\) and \(j\in\indexN\). From~\eqref{eq:delta_alg}, we derive \(-g_j(\xs)\geq \Lxjdo\delta_A\). Then,
\begin{equation}
\label{eq:LipschtizBound}
    -g_j(x)=-g_j(\xs)-(g_j(x)-g_j(\xs))\geq -g_j(\xs)-\Lxjdo\|z-\zs\|_P\geq \Lxjdo\delta_A -\Lxjdo\theta\delta_A=(1-\theta)\Lxjdo\delta_A\,.
\end{equation}
The result for the dual bound follows analogously. Note that for any \(z^0=(x^0,y^0)\in\ballP{\zs}{\delta_A}\) there exists \(\theta\in(0,1)\) such that \(z^0\in\cballP{\zs}{\theta\delta_A}\). Then, from~\eqref{eq:LipschtizBound}, \(g_j(x)<0\)  if \(\Lxjdo>0\). Otherwise, if \(\Lxjdo=0\), \(g_j(x)\leq g_j(\xs)<0\), because \(j\in\indexN\).
\end{proof}
\noindent An immediate consequence of this claim is that $\delta_A \le \delta_G.$
    
    Next, we show that the converse inequality \(\delta_A\geq \delta_G\) also holds when \(\YY=\R^{n+m}\) and $\step = 0.$ Let \(\eps>0\) and consider the ball \(\ballP{\zs}{\delta_A+2\eps}\). Since \(\delta_A\) is the unique fixed point of \(\Delta\), and \(\Delta\) is non-increasing, then \(\Delta(\delta_A+\eps)<\delta_A+\eps\). Thus, either \((a)\)
    there exists \(j\in \indexN\) such that \(\displaystyle\delta_A+\eps>-g_j(\xs)/\Lx(\delta_A+\eps)\), or \((b)\) there exists \(j\in \indexBs\) such that \(\displaystyle\delta_A+\eps>\ys_j/\Ly(\delta_A+\eps)\). Suppose \((a)\) is true. Then, from the definition of \(\Lx\), there exists \(z\in \cballP{\zs}{\delta_A+\eps}\backslash(\{\zs\}+\ker(P))\) with
        \begin{equation}
        \label{eq:deltaAproofInequality}
            \frac{g_j(x)-g_j(\xs)}{\|z-\zs\|_P}> \frac{-g_j(\xs)}{\delta_A+\eps}\,.
        \end{equation}
        %Thanks to convexity, we can take a point $z$ such that \(\|z-\zs\|_P=\delta_A+\eps\).
        %\md{Pedro, please explain this in more detail}\pil{Done.}
        If \(\|z-\zs\|_P<\delta_A+\eps\), define  and \(z'=\zs + (\delta_A+\eps)(z-\zs)/\|z-\zs\|_P\). Note that \(z=\theta z' + (1-\theta)\zs\) for \(\theta = \|z-\zs\|_P/(\delta_A+\eps)\in(0,1)\). Also, \(\|x-\xs\|_2\neq 0\), otherwise the LHS of~\eqref{eq:deltaAproofInequality} would be equal to zero, contradicting the inequality. Furthermore, by the homogeneity of the seminorm,
        \begin{equation}
        \label{eq:deltaAproofNorms}
            \zeta:=\frac{\|z-\zs\|_P}{\|x-\xs\|_2}=\frac{\|\theta (z'-\zs)\|_P}{\|\theta(x'-\xs)\|_2}=\frac{\|z'-\zs\|_P}{\|x'-\xs\|_2}\,.
        \end{equation}
        On the other hand, Assumption~\ref{assumption:problem} \((i)\) gives us that \(g_j\) is convex. Then, considering that \(z\) is a convex combination of \(z'\) and \(\zs\), and that \(\|z'-\zs\|_P=\delta_A+\eps\), we obtain
        \begin{equation}
        \label{eq:deltaAproofExtending}
        \begin{aligned}
            \frac{g_j(x')-g_j(\xs)}{\delta_A+\eps}&=\frac{g_j(x')-g_j(\xs)}{\zeta\|x'-\xs\|_2} & \text{(From~\eqref{eq:deltaAproofNorms})}\\
            &\geq \frac{g_j(x)-g_j(\xs)}{\zeta\|x-\xs\|_2} & \text{(Convexity of \(g_j\))}\\
            &= \frac{g_j(x)-g_j(\xs)}{\|z-\zs\|_P} & \text{(From~\eqref{eq:deltaAproofNorms})}\\
            &> \frac{-g_j(\xs)}{\delta_A+\eps} & \text{(From~\eqref{eq:deltaAproofInequality})}\,.
        \end{aligned}
        \end{equation}
        Therefore \(g_j(x')-g_j(\xs)> -g_j(\xs)\) and, consequently, \(g_j(x')> 0\). If \((b)\) is true, by an analogous argument we get that there exists \(j\in\indexBs\) and \(z=(x,y)\in\cballP{\zs}{\delta_A+\eps}\) such that \(y_j< 0\). In either case, the constraints \(G_\indexN(x^0)<0\) and \(y>0\) are not satisfied at every point of \(\ballP{\zs}{\delta_A+2\eps}\). Thus, $\delta_{G} \leq \delta_{A}+2\eps$ and since $\eps$ was arbitrary, we obtain $\delta_{G} \leq \delta_{A}.$ This concludes the proof of Proposition~\ref{prop:deltaEquivalence}.

        \end{proof}

\subsection{Proof of Theorems~\ref{theorem:asymptoticIdentification} and~\ref{theorem:finiteTimeIdentificationLinear}}
\label{appendix:proof-asymptoticIdentification}

In this section, we prove generalized versions of Theorems~\ref{theorem:asymptoticIdentification} and~\ref{theorem:finiteTimeIdentificationLinear} using Assumption~\ref{assumption:asymptoticIdentificationWeak}. We transcribe the statement of these results here for completeness. To do this, first we extend the set defined in~\eqref{eq:YN} to account for the set \(\YY\) introduced with Assumption~\ref{assumption:asymptoticIdentificationWeak}. In the context of Assumption~\ref{assumption:asymptoticIdentificationWeak}, we denote
    \[
        \YN=\{(x,y)\in\YY:y_\indexN=0\}.
    \]
Note that this is equal to~\eqref{eq:YN} when \(\YY=\R^{n+m}\).
%\hnote{Add a roadmap of the proof here. Roughly speaking, (i) $z_k$ converges to $z^*$ in $P$ norm, (ii) $g_j$ is in the right direction, thus $y$ decreases by a certain amount whenever positive, (iii) this leads to finite-time identification. (please state it in a formal language). }\pil{Done. Following Mateo's suggestion, I wrote it after the statement of the Theorems, and right before the proof.}
\begin{theorem}[Generalization of Theorem~\ref{theorem:asymptoticIdentification}]
\label{theorem:asymptoticIdentificationG}
      Suppose Assumptions~\ref{assumption:problem} and~\ref{assumption:asymptoticIdentificationWeak} hold. Let $z^k$ be the \(k\)th iterate generated by update \eqref{eq:algorithm}. Then, there exists \(K\in\NN\) such that \(z^k\in \ballP{\zs}{\delta_{A}}\cap\YN\subseteq\cM\) for all \(k\geq K\).\footnote{\label{fn:delta}We emphasize that $\delta_A$ in this statement is the quantity defined in \eqref{eq:delta_alg}; by Proposition~\ref{prop:deltaEquivalence}, it coincides with $\delta_G$ as defined in \eqref{eq:degeneracy-metric} whenever Assumption~\ref{assumption:asymptoticIdentificationWeak} holds with $p=0$.}
\end{theorem}

\begin{theorem}[Generalization of Theorem~\ref{theorem:finiteTimeIdentificationLinear}]
\label{theorem:finiteTimeIdentificationLinearG}
     Suppose Assumptions~\ref{assumption:problem}, ~\ref{assumption:asymptoticIdentificationWeak},~\ref{assumption:rapidLocalConvergence}, and~\ref{assumption:finiteTimeIdentification} hold and that the \(k\)-th iterate \(z^k\) and the \(k\)-th intermediate iterate \(\tilde{z}^k\) of the meta-algorithm defined in~\eqref{eq:algorithm} are equal. Then, \(z^k\in\ballP{\zs}{\delta_{A}/2}\cap\YN\subseteq\cM\) provided
    \[
        k >\left\lceil e\left(\frac{\gamma\lambdamax(P)}{\alpha_G}\right)^{\irate}\right\rceil\left[1+{\irate}\ln\left(\frac{4\lambdamax(P)^{\frac{3}{2}}\dist_2(z^0,\sol)}{\alpha_G\delta_A}\right)\right]+\left\lceil \step+\max_{j\in\indexN}\frac{\Lyjd}{\eta C_j}\right\rceil, %\quad\text{where}\quad \rho=.
    \]
    where
    \[
        C_j=\begin{cases}
            \Lxjdo & \text{if }\Lxjdo>0\\
            -2g_j(\xs)/\delta_A & \text{otherwise.}\footref{fn:delta}
        \end{cases}
    \]
    % equal to \(\Lxjdo\) if \(\Lxjdo>0\) or equal to \(-2g_j(\xs)/\delta_{A}\) otherwise.
\end{theorem}
Let us make a couple of remarks about the statement of this result.
Recall that $\Lxjdo := \sup_{z \in \ballP{\zs}{\delta} \setminus \zs} \tfrac{g_j(x) - g_j(x^\star)}{\|z - \zs\|_P}$. If \(g_j\) is constant in \(\ballP{\zs}{\delta}\), then \(\overline\Lxjd=0\), whence the bound in Theorem~\ref{theorem:finiteTimeIdentificationLinear} becomes unrealizable. The constant \(C_j\) in Theorem~\ref{theorem:finiteTimeIdentificationLinearG} auxiliates this edge case. %To understand its construction, first note that the origin of \(\Lxjdo\) in the bound of Theorem~\ref{theorem:finiteTimeIdentificationLinear} relies on the proof of Proposition~\ref{prop:identificationFinalStep}. Specifically, it arises from the derivation of a strictly negative upper bound for \(g_j(x)\) proportional to \(-\delta_A\) when \(z=(x,y)\in\ballP{\zs}{\delta_A/2}\) and \(j\in\indexN\). If \(g_j\) is not constant in \(\ballP{\zs}{\delta}\), by means of Claim~\ref{prop:LipchitzGuarantee} we obtain \(g_j(x)\leq -\Lxjdo \delta_A/2\). On the other hand, if \(g_j\) is constant in \(\ballP{\zs}{\delta}\), we still have \(g_j(x)=g_j(\xs)=-\bar{C}_j\delta_A\) with \(\bar{C}_j=-g_j(\xs)/\delta_A\); note that \(\bar{C}_j<0\) when \(j\in\indexN\). From the definition~\eqref{eq:delta_alg}, we observe \(\Lxjdo\leq \bar{C}_j\); the bound with \(\bar{C}_j\) for the case  \(\Lxjdo=0\) is formally stronger than the one for the case \(\Lxjdo> 0\). It can be further improved with the two appearing in \(C_j\), which is permitted in the proof of Proposition~\ref{prop:identificationFinalStep} morally from the fact that required bound is for \(g_j(x)\) with \(z=(x,y)\) in \(\ballP{\zs}{\delta_A/2}\), not in \(\ballP{\zs}{\delta_A}\). \pil{I wrote this Remark after a comment of Haihao in the last meeting that asked me to explain what the intuition was behind the form of \(C_j\). This is the story, although I'm dissatisfied with not being able to explain it more succinctly.}
%\hnote{Move this remark to after the statement of Theorem B.5 before the proof. Connect to the algorithms (i.e., PDHG satisfies the $z^k=\tilde{z}^k$.)}\pil{Done.}
%\noindent{\em Remark (\(z^k=\tilde{z}^k\)).} 
It is natural to wonder whether the condition \(\tilde{z}^k= z^k\) for all \(k\in\NN\) is essential for the above result---note that ADMM and PPM satisfy it, but PDHG and EGM do not. This condition is not essential. The same conclusion holds provided that, for all $k\in\NN$,
\(
\dist_{P^\dagger}\!\bigl(0,\sub(\tilde z^{k})\bigr)\ \le\ \frac{\gamma\,\dist_{P}(\tilde z^{0},\sol)}{k^{\rate}}\,.
\)
Under this bound, linear convergence of $\tilde z^{k}$ follows by an analogous argument to the one used for $z^{k}$ (see Propositions~\ref{theorem:linearConvergence} and~\ref{prop:linearConvergence}), and this is sufficient to conclude Theorem~\ref{theorem:finiteTimeIdentificationLinearG}.
% can be proven for  in the same way as it is proven for \(z^k\).

The strategy we use to establish active set identification is straightforward. Recall that the identifiable set is characterized by the (in)equalities \(G_\indexN(x)<0\), \(y_{\indexBs}>0\), and \(y_\indexN=0\). Since the iterates converge to \(\zs\) in the \(P\)-norm, they remain inside \(\ballP{\zs}{\delta_A/2}\) after a sufficient number of steps. Then, Claim~\ref{prop:LipchitzGuarantee} gives us two conditions that any point inside this ball satisfies.
\begin{itemize}
    \item[\((i)\)] The magnitude of the nonactive constraints \(G_\indexN\) is negative and uniformly bounded away from zero, establishing the \(G_\indexN(x)<0\) identification inequality.
    \item[\((ii)\)] The multipliers of the active constraints \(y_{\indexBs}\) are positive and uniformly bounded away from zero, establishing the \(y_{\indexBs}>0\) identification inequality.
\end{itemize}
To prove the remaining equality \(y_\indexN=0\), note that any point inside \(\ballP{\zs}{\delta_A/2}\) also satisfies that the multipliers of the nonactive constraints \(y_\indexN\) are close to zero, since they are close to \(\ys_\indexN=0\). With this and \((i)\), the iterative application of the meta-algorithm's dual update---gradient ascent projected onto \(\RR^m_+\)---ensures that \(y_\indexN=0\) after a sufficient number of steps. Finally, active set identification can be guaranteed in a {\em finite number of steps} leveraging the explicit rates of convergence of the meta-algorithm and the metric subregularity condition of the problem.

% Now we show our first finite time asymptotic identification Theorem~\ref{theorem:finiteTimeIdentificationLinear}, which shows that when the intermediate step is redundant, we obtain finite time identification in \(O(\alpha_G^{-1}\log(\delta_A^{-1}))\) steps. The proof is divided into two parts; each one corresponding to the two summands in the definition of $K$ in Theorem~\ref{theorem:finiteTimeIdentificationLinear}. The first part, associated to the first summand, bounds the number of steps it takes for the iterates to land on the $P$-ball of radius $\delta_A/2$ around $\zs$. The second part corresponds to the additional number of iterations needed to reach $\YN$. This one is a direct application of Proposition~\ref{prop:localSystemFiniteIdentification}.
% We start with auxiliary propositions that will be used in both proofs. The following proposition establishes a sufficient condition for iterates to belong to the union of manifolds \(\cM\).
We begin with a few auxiliary propositions that will be used in both proofs. The next proposition provides a sufficient condition ensuring that the iterates eventually lie in the union of manifolds \(\cM\). 
%Recall that, in the context of Assumption~\ref{assumption:asymptoticIdentificationWeak}, \(\YN=\{(x,y)\in\YY:y_\indexN=0\}\).
% \begin{proposition}
% \label{prop:manifoldInclusion}
%     Suppose Assumption~\ref{assumption:asymptoticIdentificationWeak} hold. Then, \(\ballP{\zs}{\delta_A}\cap\YN\subseteq\cM\).
% \end{proposition}

% \begin{proof}
%      Let \(z=(x,y)\in\cballP{\zs}{\theta \delta_A}\cap\YN\). Since, \(z\in \cballP{\zs}{\delta_A}\), from Claim~\ref{prop:LipchitzGuarantee} we have \(G_\indexN(x)<0\) and \(y_{\indexBs}>0\). Also, since \(z\in\YN\), we have \(y_{\indexBw}\geq 0\) and \(y_\indexN=0\). Thus, \(z\in\cM\).
% \end{proof}
\begin{proposition}
\label{prop:manifoldInclusion}
    Suppose Assumption~\ref{assumption:asymptoticIdentificationWeak} holds. For any initial iterate \(z^0\), if \(z^0\in\ballP{\zs}{\delta_{A}}\) and \(z^{\step}\in\ballP{\zs}{\delta_{A}}\cap\YN\), then \(z^{\step}\in\cM\). In particular, if $\step = 0,$ then $\ballP{\zs}{\delta_A}\cap\YN\subseteq\cM.$\footref{fn:delta} %\md{Can you remind us what $\YN$ is?}\pil{Done. I wrote its `generalized' definition at the beginning of the subsection.}
\end{proposition}

\begin{proof}
     Since, \(z^0\in \ballP{\zs}{\delta_A}\), from Claim~\ref{prop:LipchitzGuarantee} we have \(y_{\indexBs}^{p}>0\). Similarly, since \(z^{\step}\in \ballP{\zs}{\delta_A}\), we have \(G(x^{\step})>0\). Also, since \(z^{\step}\in\YN\), we have \(y^\step_{\indexBw}\geq 0\) and \(y^\step_\indexN=0.\) Thus, \(z^{\step}\in\cM\).
\end{proof}

 The following Proposition shows that once the iterates enter \(\ballP{\zs}{\delta_{A}/2}\), it takes a few more steps to reach the active set \(\cM\).
\begin{proposition}
    \label{prop:identificationFinalStep}
    Suppose Assumption~\ref{assumption:asymptoticIdentificationWeak} holds and that the iterates satisfy \(z^k,\tilde{z}^k\in \cballP{\zs}{\delta_{A}/2}\cap\YY\) for all \(k\in\NN_0\). Then, \(z^k\in\cM\) for all
    \[
        k\geq\step+\left\lceil \max_{j\in\indexN}\frac{\Lyjd}{\eta C_j}\right\rceil=:K,\quad\text{where}\quad C_j=\begin{cases}
            \Lxjdo &\text{if } \Lxjdo>0\\
            -2g_j(\xs)/\delta_{A} & \text{otherwise}.\footref{fn:delta} 
        \end{cases}
    \]
    % In particular, if \(\theta\in(0,1/2]\) we have \(K\geq \step+\left\lceil \max_{j\in\indexN}\frac{\Lyjdo}{\eta\Lxjdo}\right\rceil\).
    %Furthermore, if \(\theta\leq 1/2\) then \(K\leq\lceil \max_{j\in\indexN}\frac{\Lyjdo}{\eta \Lxjdo}\rceil\).
\end{proposition}

\begin{proof}
     Fix \(j\in\indexN\). First, for any \(k\geq p\) we aim to prove
    \begin{equation}
    \label{eq:induction}
        y^{k}_j=\proj_{\RR_+}\left(y^\step_j+\sum_{\ell=\step}^{k-1}\eta g_j(\tilde{x}^\ell)\right)\,.
    \end{equation}
    We proceed by induction. For \(k=\step\), the result is trivial. Now, assume~\eqref{eq:induction} holds for some \(k\geq\step\). 
    To exploit the induction hypothesis we claim that for any \(a\in\RR\) and \(b\leq 0\) we have
    \begin{equation}
    \label{eq:claimProj}
            \proj_{\RR_+}((\proj_{\RR_+}a)+b)=\proj_{\RR_+}(a+b)\,.
    \end{equation}
    To prove this, suppose \(a\geq 0\), then \(\proj_{\RR_+}a=a\) whence the result follows trivially. Now suppose \(a\leq 0\), then \((\proj_{\RR_+}a)+b=b\leq 0\) and \(a+b\leq 0\), whence
    \(
            \proj_{\RR_+}((\proj_{\RR_+}a)+b)=0=\proj_{\RR_+}(a+b)\,.
    \)
    Since \(\tilde{z}^{k}\in \cballP{\zs}{\delta_{A}/2}\cap\YY\), invoking Claim~\ref{prop:LipchitzGuarantee} we derive \(g_j(\tilde{x}^k)\leq -\Lxjdo\delta_{A}/2\leq 0\). Thus,
    \begin{align*}
        y^{k+1}&=\proj_{\RR_+}(y_j^k+\eta g_j(\tilde{x}^k)) & \text{(Assumption~\ref{assumption:asymptoticIdentificationWeak} \((ii)\))}\\
        &=\proj_{\RR_+}\left(\proj_{\RR_+}\left(\sum_{\ell=\step}^{k-1}y^\step_j+\eta g_j(\tilde{x}^\ell)\right)+\eta g_j(\tilde{x}^k)\right) & \text{(Induction hypothesis)}\\
        &=\proj_{\RR_+}\left(y^\step_j+\sum_{\ell=\step}^{k}\eta g_j(\tilde{x}^\ell)\right) & \text{(Using~\eqref{eq:claimProj})}.
    \end{align*}
    Thus,~\eqref{eq:induction} holds for any \(k\geq\step\). Now, let \(k\geq K\). Combining \(z^\step\in\cballP{\zs}{\delta_{A}/2}\cap\YY\) and Assumption~\ref{assumption:asymptoticIdentificationWeak} \((iv)\) we obtain
    \[
        y^\step_j=y^\step_j-\ys_j\leq \Lyjd\|z^0-\zs\|_P\leq 
        \Lyjd \delta_{A}/2\,.
    \]
    Moreover, using Claim~\ref{prop:LipchitzGuarantee} in tandem with the fact that \(\tilde{z}^\ell\in \cballP{\zs}{\delta_{A}/2}\) for all \(\ell\in\{0,\dots ,k-1\}\), we derive \(g_j(\tilde{x}^\ell)\leq -\Lxjdo\delta_{A}/2\) for all \(\ell\in\{0,\dots,k-1\}\). If \(\Lxjdo=0\), then \(g_j(\tilde{x}^\ell)=g_j(\xs)<0\), whence, \(g_j(\tilde{x}^\ell)\leq -C_j\delta_{A}/2<0\). Consequently,
    \[
        y_j^\step+\sum_{\ell=\step}^{k-1}\eta g_j(\tilde{x}^\ell)\leq \Lyjd\delta_A/2 -(k-\step)\eta C_j\delta_A/2\leq 0\,.
    \]
    The last inequality holds because \(k\geq K\). Thus \(y_j^{k}=0\), from~\eqref{eq:induction}. Since \(j\in\indexN\) is arbitrary, \(y^k_\indexN=0\). Moreover, from Assumption~\ref{assumption:asymptoticIdentificationWeak} \((ii)\) we have \(y^k_{\indexB}\geq 0\). Thus, \(z^{k-\step}\in\ballP{\zs}{\delta_A}\) and \(z^k\in\ballP{\zs}{\delta_A}\cap\YN\), whence \(z^k\in\cM\), from Proposition~\ref{prop:manifoldInclusion}. 
    % The last statement follows directly from the following inequality:
    % any \(a,b\in\RR_+\) and \(\theta\in[0,1/2]\) hold
    % \[
    %     \frac{\theta a}{(1-\theta)b}\leq \frac{a}{b}\,.
    % \]
    % If \(a=0\) or \(b=0\) the inequality holds directly. Then, suppose \(a,b>0\) and let \(h:[0,1/2]\to\RR\) such that 
    % \(
    %     h(\theta)= \theta a((1-\theta)b)^{-1}\,.
    % \)
    % The function \(h\) is differentiable; for any \(\theta\in[0,1/2]\) we have
    % \(
    %     h'(\theta)=a((1-\theta)^2b)^{-1}>0\,.
    % \)
    % Therefore \(h\) is strictly increasing, whence \(h(\theta)\leq h(1/2)=ab^{-1}\).
\end{proof}
% \noindent\textit{Remark.} If \(\Lxjdo=0\) for all \(j\in\indexN\) then \(K=\infty\) in Proposition~\ref{prop:identificationFinalStep}. However, in that case \(g(x)_j\leq g_j(\xs)\) for all \(x\in\cballP{\zs}{\delta_A}\cap\YY\). Therefore, to get a valid bound, \(\Lxjdo\) can be replaced by \(-\delta_A^{-1}g_j(\xs)>0\). Moreover, in general the bound obtained with this replacemenent is not worse than the one offered in Proposition~\ref{prop:identificationFinalStep}: from~\eqref{eq:delta_alg}, \(\Lxjdo\leq -\delta_A^{-1}g_j(\xs)\) for any \(j\in\indexN\).

Now we have the main ingredients for the proofs of Theorems~\ref{theorem:asymptoticIdentificationG} and~\ref{theorem:finiteTimeIdentificationLinearG}.

\begin{proof}[Proof of Theorem~\ref{theorem:asymptoticIdentificationG}]
    From Assumption~\ref{assumption:asymptoticIdentificationWeak} \((i)\) there exists \(M\in\NN\) such that \(z^k,\tilde{z}^k\in \ballP{\zs}{\delta_{A}/2}\) for all \(k\geq M\). From Proposition~\ref{prop:identificationFinalStep} it suffices to take \(K\geq M+\step+\lceil\max\{\Lyjd(\eta C_j)^{-1}:j\in\indexN\}\rceil\).
\end{proof}

\begin{proof}[Proof of Theorem~\ref{theorem:finiteTimeIdentificationLinearG}]
     Let \(\eps>0\) and \(\zstilde\in\sol\) such that \(\dist_P(z^k,\sol)\geq\|z^k-\zstilde\|_P+\eps\). Denote \(\nu=\gamma\lambdamax(P)/\alpha_G\).
     \(\rho = \lceil e\gamma \lambdamax(P)/\alpha_G\rceil\). Then,
    \begin{align*}
        \|z^k-\zs\|_P
        &\leq \|z^k-\zstilde\|_P+\|\zstilde-\zs\|_P & \text{(Triangle inequality)}\\
        &\leq 2\|z^k-\zstilde\|_P & \text{(Assumption~\ref{assumption:finiteTimeIdentification} \((ii)\))}\\
        &\leq 2\dist_P(z^k,\sol)+\eps & \text{(Defintion of \(\zstilde\))}\\
        % &\leq 2e^{-\frac{k\rate}{\rho}}\nu\dist_{P}(z^0,\sol)+\eps& \text{(Proposition~\ref{prop:linearConvergence})}\\
        &\leq 2\lambdamax(P)^\frac{1}{2}\sqrt{e}\nu\exp\left(-\frac{1}{2}\frac{k}{\lceil e\nu^2\rceil}\right)\dist_{2}(z^0,\sol)+\eps & \text{(Proposition~\ref{prop:linearConvergence}).}
        % &=\frac{2\gamma \lambdamax(P)^{\frac{3}{2}}\sqrt{e}}{\alpha_G}\exp\left(-\frac{1}{2}\frac{k}{\lceil e\nu^2\rceil}\right)\dist_{2}(z^0,\sol)+\eps
    \end{align*}
    Since \(\eps>0\) is arbitrary, \(\|z^k-\zs\|_P\leq  2\gamma \lambdamax(P)^{\frac{3}{2}}\sqrt{e}/\alpha_G\cdot \exp(-k/2\lceil e(\gamma\lambdamax(P)/\alpha_G)^2\rceil)\cdot \dist_{2}(z^0,\sol)\). Thus, \(\tilde{z}^k=z^k\in\ballP{\zs}{\delta_A/2}\) whenever 
    \[
        k>\left\lceil e\left(\frac{\gamma\lambdamax(P)}{\alpha_G}\right)^2\right\rceil \left[1+2\ln\left(\frac{4\gamma \lambdamax(P)^{\frac{3}{2}}\dist_2(z^0,\sol)}{\alpha_G\delta_A}\right)\right]\,.
    \]
    By Proposition~\ref{prop:identificationFinalStep} it suffices at most \(\lceil\step+\max\{\Lyjd(\eta C_j)^{-1}:j\in\indexN\}\rceil\) more steps to identify \(\cM\).
\end{proof}

\subsection{Proof of Theorem~\ref{theorem:finiteTimeIdentification}}\label{appendix:proof-finiteTimeIdentification}

%Concerning the finite time identification in \(O(\alpha_L^{-1}\delta_A^{-\rate})\) steps, the first observation we make is that any point \(z\in\sols\) sufficiently close to \(\zs\) is also an optimal solution, as shown in Proposition~\ref{prop:reducedSystemIdentification}. 

% In this section we prove our stronger identification result, Theorem~\ref{theorem:finiteTimeIdentification}. The overall structure of the proof considers two parts, similarly to the one of   Theorem~\ref{theorem:finiteTimeIdentificationLinear}; first we bound the number of steps required for the iterates to land on \(\ballP{\zs}{\delta_A/2}\) and then we show it takes a few more steps to them to arrive to \(\cM\). The main contribution of Theorem~\ref{theorem:finiteTimeIdentification} is that, in constrast with Theorem~\ref{theorem:finiteTimeIdentificationLinear}, the bound of the first part does not depend on the global constant \(\alpha_G\), but only on the local one \(\alpha_L\).

In this section, we prove a generalized version of Theorem~\ref{theorem:finiteTimeIdentification} using Assumption~\ref{assumption:asymptoticIdentificationWeak}. We transcribe the statement of this result here for completeness.

\begin{theorem}[Generalization of Theorem~\ref{theorem:finiteTimeIdentification}]
\label{theorem:finiteTimeIdentificationG}
    Under Assumptions~\ref{assumption:problem},~\ref{assumption:asymptoticIdentificationWeak},~\ref{assumption:rapidLocalConvergence}, and ~\ref{assumption:finiteTimeIdentification}, the \(k\)-th iterate satisfy \(z^k\in\cM\) whenever
    \[
        k> \left(\max\left\{1,\frac{1}{\alpha_L}\right\}\frac{8\lambdamax(P)^{\frac{3}{2}}\dist_P(z^0,\sol)}{\delta_{A}}\right)^{\irate}+\left\lceil\step+ \max_{j\in\indexN}\frac{\Lyjd}{\eta C_j} \right\rceil\,,
    \]
    where
    \[
        C_j=\begin{cases}
            \Lxjdo & \text{if }\Lxjdo>0\\
            -2g_j(\xs)/\delta_A & \text{otherwise.}\footref{fn:delta}
        \end{cases}
    \]
\end{theorem}

We start with some notation and auxiliary results that will help us handle taking projections despite the lack of invertibility of $P$.
Recall that, for any point \(u\in\RR^{n+m}\), we let 
    \[
        u=u_P+u_{P^\perp}\quad\text{where}\quad u_P\in\range(P),\ u_{P^\perp}\in\range(P)^\perp=\ker(P)\,.
    \]
    Similarly, for any set \(S\subseteq \RR^{n+m}\), we denote
    \[
        S_P=\{u_P:u\in S\}\quad\text{and}\quad S_{P^\perp}=\{u_{P^\perp}:u\in S\}\,.
    \]
    If \(S\) is a nonempty closed and convex set, \(S_P\) is also a nonempty and convex set that might not be closed. Thus, the \(P\)-projection of a point \(z\in\RR^n\) over \(S\), \(\argmin\{\|u-z\|_P:u\in S\}\), is not well defined in general. 
% \begin{definition}
%     For any positive semidefinite matrix \(P\in\RR^{n\times n}\), nonempty closed convex set \(S\subseteq\RR^n\) and point \(x\in\RR^n\), define the \(P\)-projection of \(x\) over \(S\) as
%     \[
%         \pproj_Sx=\argmin_{u\in S}\|x-u\|_P\,.
%     \]
% \end{definition}
% \begin{lemma}
%     \label{lemma:pprojectionWellDefined}
%     Let \(P\in\RR^{n\times n}\) be positive semidefinite matrix, \(S\subseteq\RR^n\) a nonempty closed convex set, and \(x\in\RR^n\). Then, if \(S_P\) is closed,  \(\pproj_Sx\) is nonempty and \((\pproj_Sx)_P\) is a singleton.
% \end{lemma}
% \begin{proof}
%     Note that \(\|z\|_P=\|z_P\|_P\) for any \(z\in\RR^n\). Therefore,
%     \[
%         \min_{u\in S}\|x-u\|_P=\min_{u\in S_P}\|x_P-u\|_P\,.
%     \]
%     Also, \(S\) is nonempty, convex, and closed, whence \(S_P\) is nonempty, convex, and closed (by hypothesis). Then, since \(\|\cdot\|_P\) is a norm in \(\range(P)\), there exists a unique \(\bar{w}\in S_P\) such that 
%     \[
%         \bar{w}=\argmin_{u\in S_P}\|x_P-u\|_P\,.
%     \]
%     Furthermore, since \(\bar{w}\in S_P\), by the construction of \(S_P\) there exists \(\bar{u}\in S\) such that \(\bar{u}_P=\bar{w}\). Thus, \(\bar{u}\in \pproj_Sx\), whence \(\pproj_Sx\) is nonempty. Furthermore, for any \(\hat{u}\in\pproj_Sx\) we have that \(\hat{u}_P\in\argmin_{u\in S_P}\|x_P-u\|_P=\{\bar{w}\}\), whence \((\pproj_Sx)_P\) is a singleton.
% \end{proof}
    To overcome this difficulty, for any nonempty, convex \(S\subseteq\RR^n\) set, and $\eps > 0$ we define the \(P\)-closure-projection and the \((P,\eps)\)-almost-projection as the point and set-valued map given by %\md{Define $\bar z$}\pil{\(\bar{z}_S^P?\). Done}
    \[
        \cPo^P_S(z)=\argmin_{u\in\cl(S_P)}\|u-z\|_P\quad\text{and}\quad\cP^P_{S,\eps} (z)=S\cap\cballP{\cPo_S^P(z)}{\eps},\quad\text{respectively}.
    \]
%    \hnote{Do $\peproj_S (z)$ here and later?}\pil{Done.}
    %The vector \(\bar{z}\) \md{should this be a $\bar z$?}\pil{Yes. Done.}
    The \(P\)-closure-projection is well defined as a point instead of a set since \(\cl(S_P)\) is nonempty, convex, and closed, and $\| \cdot \|_P$ is a norm when restricted to \(\range(P)\). We will use the following auxiliary result. %equipped with the \(P\)-seminorm and \(P\)-inner-product is a Hilbert space, whence the argmin is a singleton.
    %Also, note that \(\peproj_Sx\) is nonempty: since \(\bar{u}\in\cl(S_P)\), for any \(\eps>0\) there exists \(u\in S_P\) such that \(\|u-\bar{u}\|_P\leq\eps\), whence there exists \(u\in S\) such that \(\|u-\bar{u}\|_P\leq\eps\). 
    % \[
    %     \peproj_Sx=\{u\in S:\|x-u\|_P\leq p^\star+\eps\},\quad\text{where}\quad p^\star=\inf_{u\in S}\|x-u\|_P\,.
    % \]
    % Clearly, \(\peproj_Sx\) is nonempty for any \(\eps>0\). 

\begin{lemma}
    \label{lemma:pprojection}
     %Given any $u$, let $\bar u \in \peproj_S u$ denote any element in the almost-projection. Then, the following hold true.
     %Let \(P\in\RR^{n\times n}\) be positive semidefinite matrix, 
     Let \(S\subseteq\RR^{n + m}\) a nonempty convex set, and \(z,w\in\RR^n\). Then, for any  \(u\in\cP_{S,\varepsilon}^P(z)\) %=S\cap\cballP{\cPo^P_{S}(z)}{\eps}\)
     and \(v\in\cP_{S,\varepsilon}^P(w)\), %=S\cap\cballP{\cPo^P_S(w)}{\eps}\).
     the following hold true. 
     \begin{itemize}
            \item[\((i)\)] (Convergence) Any sequence \(z^k\in\cP_{S,\eps_k}^{P}(z)\) with \(\eps_k\to 0\) satisfies \(\|z^k-\cPo^P_S(z)\|_P\to 0\).
        \item[\((ii)\)] (Non-expansiveness) \(\|\cPo_S^P(z)-\cPo_S^P(w)\|_P\leq \|z-w\|_P\) and \(\|u-v\|_P\leq \|z-w\|_P+2\eps\,.\)
        \item[\((iii)\)] (Fixed-points) For any \(z\in S\) we have \(z = \cPo_S^P(z)\in \cP_{S,\eps}^P(z)\).
        \item[\((iv)\)] (Distance to \(S\)) For any $\eps > 0$ and $u \in \cP_{S,\eps}^{P}(z)$ we have \[\dist_P(z,S)=\|z-\cPo_S^P(z)\|_P\leq\|z-u\|_P\leq \dist_P(z,S)+\eps\,.\]
     \end{itemize}
\end{lemma}

\begin{proof}
    Items \((i)\) and \((iii)\) follow directly by definition. Item $(iv)$ follows easily from the Triangle Inequality. 
    %By the continuity of \(\|\cdot\|_P\),
    % \[
    %     \|x-\bar{u}\|_P=\inf_{w\in S}\|x-w\|_P\quad\text{and}\quad \|y-\bar{v}\|_P=\inf_{w\in S}\|y-w\|_P\,.
    % \]
    % By the Triangle Inequality, it follows that
    % \begin{align*}
    %     &\|x-\bar{u}\|_P\leq\|x-u\|_P\leq \|x-\bar{u}\|_P+\|u-\bar{u}\|_P\leq \|x-\bar{u}\|_P+\eps,\quad\text{and similarly,}\\
    %     &\|y-\bar{v}\|_P\leq\|y-v\|_P\leq \|y-\bar{v}\|_P+\|v-\bar{v}\|_P\leq\|y-\bar{v}\|_P+\eps\,.
    % \end{align*}
    % Therefore \((iv)\) holds. 
    Finally, since \(\|\cdot\|_P\) is a norm in \(\range(P)\) and (norm)-projections over nonempty closed convex sets are nonexpansive, and so
    \(
        \|\cPo^P_S(z)-\cPo^P_S(w)\|_P=\|{u}_P-{v}_P\|_P\leq \|x_P-y_P\|_P=\|x-y\|_P.
    \)
    Therefore, by the triangle inequality, we recover \((ii)\)
    \[
        \|u-v\|_P\leq\|\cPo_S^P(z)-\cPo_S^P(w)\|_P+\|u-\cPo_S^P(z)\|_P+\|v-\cPo_S^P(w)\|_P\leq \|x-y\|_P+2\eps.
    \]
\end{proof}

% \begin{proposition}
%     \label{prop:finiteTimeBall}
%     Under Assumptions~\ref{assumption:problem},~\ref{assumption:asymptoticIdentificationWeak},~\ref{assumption:rapidLocalConvergence}, and ~\ref{assumption:finiteTimeIdentification}, for any \(\theta\in (0,1]\), the \(k\)-th iterate and the \(k\)-th intermediate iterate satisfy \(z^k,\tilde{z}^k\in\ballP{\zs}{\theta\delta_A}\) whenever
%     \[
%         k > K:=\left(\max\left\{1,\frac{\lambdamax(P)}{\alpha_L}\right\}\frac{4\gamma \lambdamax(P)^{\frac{1}{2}}\dist_2(z^0,\sol)}{\theta\delta_A}\right)^{\irate}\,.
%     \]
% \end{proposition}

Now we have all the ingredients for the proof of Theorem~\ref{theorem:finiteTimeIdentificationG}.
\vspace{0.32 cm}
\noindent\textit{Proof of Theorem~\ref{theorem:finiteTimeIdentificationG}.}
% \pil{I faked the Proof environment to be able to wrap a figure inside.}
% \md{Please change $\delta_A$ to $\delta_{A}$ in all instances of this proof. Also write $\delta_{A}/2$ as opposed to $2^{-1}\delta_{A}.$}\pil{Done.}
Denote \(\xi=\delta_A/2\) and 
    \[
        K:= \left(\max\left\{1,\frac{1}{\alpha_L}\right\}\frac{8\lambdamax(P)^{\frac{3}{2}}\dist_P(z^0,\sol)}{\delta_A}\right)^{\irate}\,.
    \]
    For any \(z\in\RR^{n+m}\) and \(\eps>0\), we write \(\cPo (z):=\argmin\{\|u-z\|_P:u\in\cl((\sols)_P)\}\),
    % \hnote{$\cPo(z)$ here and later please, same for $\cPe (z)$ and other projections.}\pil{Done.} 
    and \(\cPe (z)\) to denote some fixed element of \(\cP^P_{\sols,\eps} (z)=\sols\cap\cballP{\cPo (z)}{\eps}\)
    % \md{Why not introducing this notation from before? Also in page 28 you define similar notation, with different meaning. Please change it there (right after (34)).}\pil{Done.} 
    Any time we write \(\cPe (z)\), we refer to the same element of \(\cP^P_{\sols,\eps} (z)\).  The proof consists of six steps. 
    % \md{As general rule of thumb, only use powers of $2$ for $2^{7}$ or higher.}\pil{Fixed. Done.}
    \begin{itemize}
        \item[\(1.\)] We show that for any \(k> K\), we have
        \begin{equation}\label{eq:always-label}
            \|z^k-z^{k+1}\|_P< \xi/4\quad\text{and}\quad\dist_P(z^k,\sols)< \xi/4\quad\text{for all}\quad k>K\,.
        \end{equation}
        %The proof follows from the sublinear rates given by Assumptions~\ref{assumption:rapidLocalConvergence} \((i)\) and~\ref{assumption:finiteTimeIdentification} \((i)\). This result implies that, for sufficiently large iterations, each step $z^{k+1}-z^k$ makes only small progress and the iterates $z^k$ remain close to $\sols$.
        \item[\(2.\)] We prove that for any \(k> K\), {\em if} \(\cPo (z^k)\in\cballP{\zs}{\xi/2}\) {\em then} \(z^k\in\ballP{\zs}{\xi/2}\). %The proof follows from step 1 and star non-expansiveness (Assumption~\ref{assumption:finiteTimeIdentification} \((ii)\)). This result means that, for sufficiently large iterations, the iterates $z^k$ must be close to the optimal solution $\zs$ whenever their projections $\cPo (z^k)$ are.
        \item[\(3.\)] We establish that projections $\cPo (z^k)$ are indeed close to the optimal solution $\zs$ for large enough iterations. Formally, for any \(k> K\) we have that \(\cPo (z^{k})\in\cballP{\zs}{\xi/2}\). %We prove this result by contradiction: if \(\cPo z^{k}\not\in\cballP{\zs}{\xi/2}\), then there exists \(n\geq k\) such that \(z^{n}\) and \(z^{\ell+1}\) belong to two parallel halfspaces that are at distance \(\xi/4\) (see Fig.~\ref{fig:proof}), contradicting step \(1\). The fundamental cause is star non-expansiveness (Assumption~\ref{assumption:finiteTimeIdentification} \((ii)\)).
        \item[\(4.\)] Combining steps \(2\) and \(3\), we conclude that any \(k> K\) holds \(z^k\in\ballP{\zs}{\xi/2}\).
        \item[\(5.\)] We use Step 4 to show that for any \(k> K\) we have \(\tilde{z}^k\in\ballP{\zs}{\xi}\).
        \item[\(6.\)] Leveraging Proposition~\ref{prop:identificationFinalStep} we conclude that it takes \(\lceil\step+\max\{\Lyjd(\eta C_j)^{-1}:j\in\indexN\}\rceil\) additional steps for the iterates to reach \(\cM\).
    \end{itemize}
    %\hnote{The above is great!}
    Next we execute these steps in detail.
\paragraph{Step 1.} For the first inequality in~\eqref{eq:always-label}, we have 
    \begin{align*}
       \|z^{k+1}-z^k\|_{P} &\leq \frac{\gamma}{k^{\rate}}\dist_P(z^0,\sol)%/k^{\rate} 
       &\text{(Assumption~\ref{assumption:finiteTimeIdentification} \((i)\))}\\
       &\leq \frac{\gamma }{k^{\rate}}\lambdamax(P)^{\frac{1}{2}}\dist_2(z^0,\sol) & \text{(Proposition~\ref{prop:normEquivalence})}\,.
    \end{align*}
    Therefore, \(\|z^{k+1}-z^k\|_P < \xi/4\) provided that \(k>K\geq (4\gamma\lambdamax(P)^{\frac{1}{2}}\dist_P(z^0,\sol)/\xi)^{2}\).
    For the second inequality~\eqref{eq:always-label},
    \begin{align*}
        \dist_P(z^k, \sols) &\leq \frac{1}{\alpha_L}\lambdamax(P) \dist_{P^\dagger}(0, \sub(z^k)\cap\range(P)) & \text{(Definition~\eqref{eq:alphaL} \& Proposition~\ref{prop:metricSubregularityPnorm})}\\
        &\leq \frac{\gamma}{\alpha_L k^{\rate}} \lambdamax(P)\dist_P(z^0,\sol)  & \text{(Assumption~\ref{assumption:rapidLocalConvergence} \((i)\))}\\
        &\leq \frac{\gamma}{\alpha_L k^{\rate}}\lambdamax(P)^{\frac{3}{2}}\dist_2(z^0,\sol)  & \text{(Proposition~\ref{prop:normEquivalence})}.
    \end{align*}
    Therefore \(\dist_{P}(z^k, \sols) < \xi/4\) whenever \(k > K\geq(4\gamma\lambdamax(P)^{\frac{3}{2}}\dist_P(z^0,\sol)/\alpha_L\xi)^{2}\).
    \paragraph{Step 2.}  We make use of the following simple lemma.
    \begin{lemma}
    \label{prop:reducedSystemIdentification}
    Suppose Assumption~\ref{assumption:asymptoticIdentificationWeak} holds. Then, \(\sols\cap \ballP{\zs}{\delta_A}\subseteq\sol\). 
\end{lemma}

\begin{proof}[Proof of Lemma~\ref{prop:reducedSystemIdentification}]
    Let \(z\in\sols\cap\ballP{\zs}{\delta_A}\). Note that \(\sols\subseteq\YY\). Since \(z\in \ballP{\zs}{\delta_A}\cap\YY\), from Claim~\ref{prop:LipchitzGuarantee} we obtain
    \(
        G(x)_\indexN\leq 0.
    \)
    Since \(z\in\sols\), \(G(x)_\indexB\leq 0\), \(y\geq 0\), and \(f(x)-h(y)\leq 0\).
    % \(z\) satisfies the system of inequalities~\eqref{eq:reducedSystemIdentification}. 
    Hence,
    %\(z\) satisfies the system of inequalities~\eqref{eq:primalDualSolutions}. That is,
    \(z\in\sol\). 
\end{proof}
    Let \(k\geq K\) and suppose \(\|\cPo (z^k)-\zs\|_P\leq \xi/2\). Then, there exists \(\bar{\eps}>0\) such that \(\cPe (z^k)\in\cballP{\zs}{\xi}\subseteq\ballP{\zs}{\delta_A}\) for all \(\eps\in(0,\bar{\eps}]\). 
    By Lemma~\ref{prop:reducedSystemIdentification}, \(\cPe (z^k)\in\sol\) for any \(\eps\in(0,\bar{\eps}]\), whence
    \begin{align*}
        \|z^k-\zs\|_{P}
        &\leq \|z^k-\cPe (z^k)\|_{P}+\|\zs-\cPe (z^k)\|_{P} & \text{(Triangle inequality)}\\
        &\leq 2\|z^k-\cPe (z^k)\|_{P}& \text{(Assumption~\ref{assumption:finiteTimeIdentification} $(ii)$)}\\
        &\leq 2\dist_P(z^k,\sols)+2\eps & \text{(Lemma~\ref{lemma:pprojection})}
    \end{align*}    
    Since \(\eps\in(0,\bar{\eps}]\) is arbitrary, %\md{What is the definition of the upper bound?}\pil{Done. I moved the Lemma to the beginning of Step 2 to clarify it.}
    \( \|z^k-\zs\|_{P}\leq 2\dist_P(z^k,\sols)\). Then, from Step \(1\) we get \(z^k\in\ballP{\zs}{\xi/2}\).

    \begin{wrapfigure}{r}{0.42\textwidth}
    \vspace{-0.5 cm}
    \centering
    \includegraphics[width=0.42\textwidth]{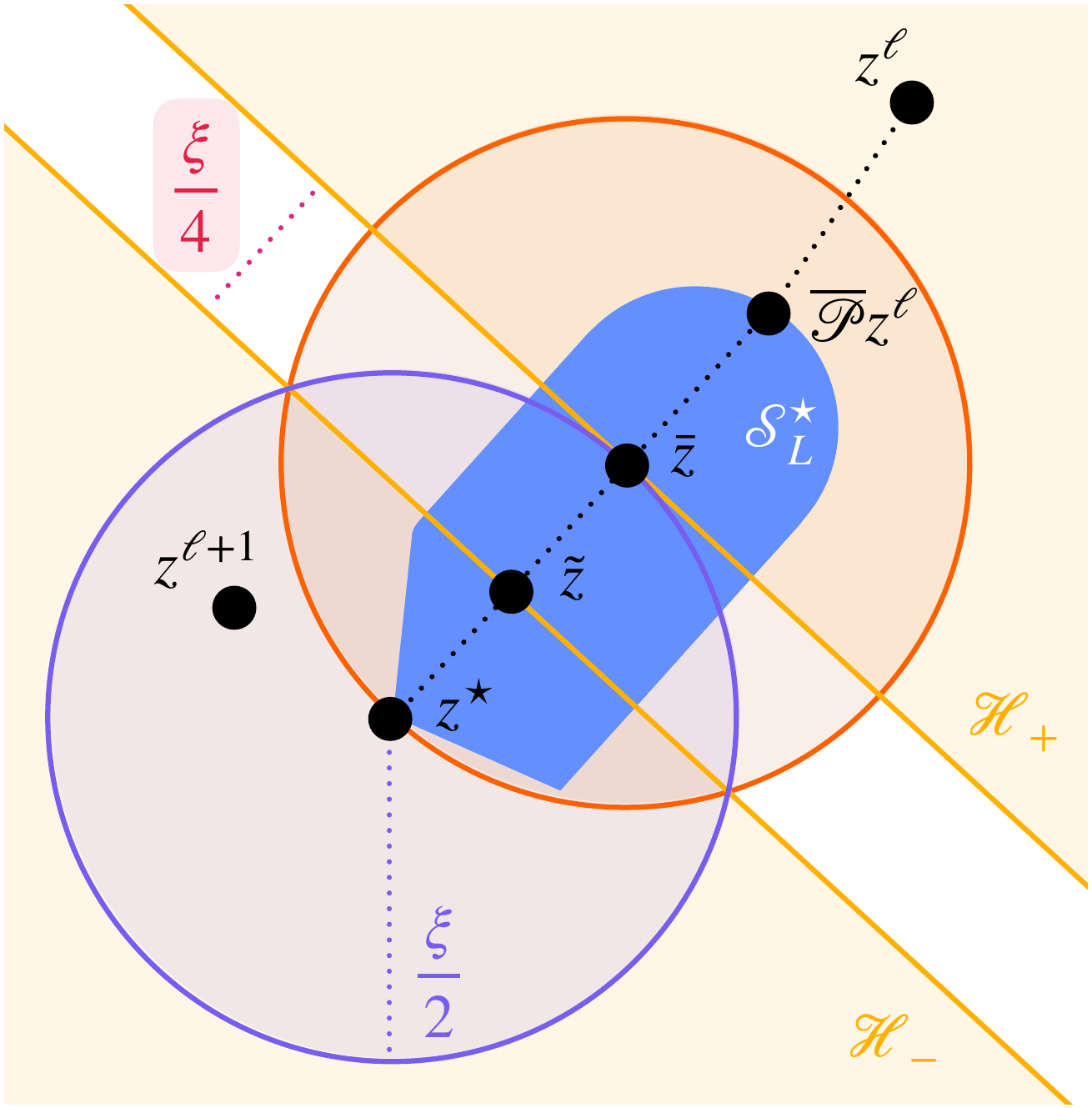}
    \caption{Visualization of Step \(3\). The orange ball \(\cballP{\bar{z}}{\xi/2}\) depicts a forbidden region for the iterates, given by star non-expansiveness (Assumption~\ref{assumption:finiteTimeIdentification} \((ii)\)). The blue solid region depicts the set \(\sols\).}
    \vspace{-1 cm}
     \label{fig:proof}
    \end{wrapfigure}
   \paragraph{Step 3.} 
  %\hnote{State this is proved by contradiction. Try to give a geometric viewpoint before diving into details for this step with Figure 8 (Figure 8 is great with clear intuition), and link different part with Figure 8. }\pil{Done.}
  We aim to prove that \(\cPo(z^k)\) belongs to \(\ballP{\zs}{\xi/2}\) for any \(k>K\) by contradiction.
   % Then, there exists \(\bar{\eps},\omega_k>0\) such that 
   % \(
   %  \|\zs-\cPe z^{k}\|_P\geq \omega_k>\xi/2\text{ for all }\eps\in(0,\bar{\eps}]\,.
   % \)
   Let \(k> K\) and suppose that \(\|\cPo (z^k)-\zs\|_P>\xi/2\). By Assumption~\ref{assumption:asymptoticIdentificationWeak} \((i)\), there exists a maximum \(\ell\in\NN\) such that \(\|\zs-z^\ell\|_P>\xi/2\). Our goal is to show that \((a)\) \(\ell>K\), and \((b)\) \(z^\ell\) and \(z^{\ell+1}\) belong to two parallel halfspaces that are at distance \(\xi/4\). These together contradict the bound on consecutive iterates derived in Step 1.  
   
   To prove that \(\ell>K\), we claim that \(\zs\in\cl(\sol_P)\). To prove this, fix \(\eps>0\). From Assumption~\ref{assumption:asymptoticIdentificationWeak} \((i)\) and Proposition~\ref{prop:linearConvergence} there exists \(\bar{k}\in\NN\) such that \(\|z^{\bar{k}}-\zs\|_P\leq \eps/3\) and \(\dist_P(z^{\bar{k}},\sol)\leq \eps/3\), respectively. Let \(\zstilde\in\sol\) such that \(\dist_P(z^{\bar{k}},\sol)\geq \|z^{\bar{k}}-\zstilde\|_P+\eps/3\). Then,
   \begin{align*}
        \dist_P(\zs,\sol)&\leq \|\zs -\zstilde\|_P\\
        &\leq \|\zs - z^{\bar{k}}\|_P+\|z^{\bar{k}}- \zstilde\|\\
        &\leq   \|\zs - z^{\bar{k}}\|_P+\dist_{P}(z^{\bar{k}},\sol)+\eps/3\\
        &\leq \eps\,.
   \end{align*}
    % % \begin{wrapfigure}{r}{0.5\textwidth}
    % % \vspace{-1.1 cm}
    % \begin{figure}
    % \begin{center}
    % \includegraphics[width=0.42\textwidth]{figures/proof_sketch.pdf}
    % \end{center}
    % % \vspace{-0.6 cm}
    % \caption{Visualization of Step \(3\) of the proof of Theorem~\ref{theorem:finiteTimeIdentificationG}. The orange ball \(\cballP{\bar{z}}{\xi/2}\) depicts a forbidden region for the iterates, given by star non-expansiveness (Assumption~\ref{assumption:finiteTimeIdentification} \((ii)\)). The blue solid region depicts the set \(\sols\).\pil{TODO: Wrap this figure.}}
    %  \label{fig:proof}

    % % \end{wrapfigure}
    % \end{figure}
   \clearpage
   \noindent Since \(\eps>0\) is arbitrary, \(\dist_P(\zs,\sol)=0\), whence \(\zs\in\cl(\sol_P)\). Then, since \(\cl(\sol_P)\subseteq\cl((\sols)_P)\), from Lemma~\ref{lemma:pprojection} (ii) we obtain
    \begin{align*}
        \|\zs-z^{k}\|_P &\geq\|\zs-\cPo (z^{k})\|_P>\xi/2\,.
    \end{align*}
   % Then, for any \(\eps\in(0,2^{-1}(\omega_k-\xi/2))\),
   %  \begin{align*}
   %      \|\zs-z^{k}\|_P &\geq\|\zs-\cPe z^{k}\|_P-2\eps & \text{(Lemma~\ref{lemma:pprojection})}\\
   %      &\geq\omega_k-2\eps\\
   %      &>\xi/2\,. &\text{(\(\eps<2^{-1}(\omega_k-\xi/2)\))}
   %  \end{align*}
    Therefore  \(\ell\geq k> K\), by maximality.
    
    Now we delve into the proof of \(\|z^\ell-z^{\ell+1}\|_P\geq \xi/4\). The strategy is to construct a point that is both in \(\sols\) and in the boundary of \(\ballP{\zs}{\xi/2}\) that faces \(z^\ell\); in particular,  by Lemma~\ref{prop:reducedSystemIdentification} it has to be also in \(\sol\). From star non-expansiveness we obtain that \(z^{\ell+1}\) has to be far away from that point, and from convexity we get that it has to be also far from \(z^{\ell}\)---see Figure~\ref{fig:proof} for an illustration of the geometry of the argument. To formalize this, first note that, since \(\|\zs-z^\ell\|_P>\xi/2\), from Step 2 we obtain \(\|\zs- \cPo (z^\ell)\|_P>\xi/2\). Then, there exists \(\bar{\eps}>0\) such that 
    \begin{equation}
    \label{eq:projDistanceII}
        \inf_{\eps\in(0,\bar{\eps}]}\|\zs-\cPe  (z^\ell)\|_P>\xi/2\,.
    \end{equation}
    Also, since \(\zs\in\cl(\sol_P)\), for all \(\eps>0\) there exists \(z^{\star,\eps}\in\sol\) such that \(\|\zs-z^{\star,\eps}\|_P\leq\eps\). Then, define
    \begin{equation}
    \label{eq:barZdef}
        \bar{z}=\zs+\frac{\xi}{2}\frac{\cPo (z^\ell)-\zs}{\|\cPo (z^\ell)-\zs\|_P}\quad\text{and}\quad \bar{z}^\eps=z^{\star,\eps}+\frac{\xi}{2}\frac{\cPe (z^\ell)-z^{\star,\eps}}{\|\cPe (z^\ell)-z^{\star,\eps}\|_P}\quad\text{for any \(\eps\in (0,\bar{\eps}]\)}\,.
    \end{equation}
    From~\eqref{eq:projDistanceII}, \(\bar{z}^\eps\) is a convex combination of \(z^{\star,\eps}\in\sol\subseteq\sols\) and \(\cPe (z^{\ell+1})\in\sols\), and from Assumption~\ref{assumption:problem} \((i)\), \(\sols\) is convex, whence \(\bar{z}^\eps\in\sols\). Then, since \(\bar{z}^\eps\in\cballP{\zs}{\xi/2}\), from Lemma~\ref{prop:reducedSystemIdentification} we obtain \(\bar{z}^\eps\in\sol\). Denoting \(\delta_A=\|\cPo (z^\ell)-\zs\|_P\) and \(\Delta_\eps=\|\cPe (z^\ell)-\zs\|_P\), by the triangle inequality we have \(|\delta_A-\Delta_\eps|\leq \|\cPo (z^\ell)-\cPe (z^\ell)\|_P\leq\eps\) and \(\delta_A-\eps\leq \Delta_\eps\leq \delta_A+\eps\). Hence, for any \(\eps\in(0,\min\{\bar{\eps},1,\delta_A/2\})\),
    \begin{equation}
    \label{eq:zbarzeps}
        \begin{aligned}
            \|\bar{z}-\bar{z}^\eps\|_P&\stackrel{\text{(I)}}{\leq} \xi/2 \|\Delta_\eps^{-1}(\cPe (z^\ell)-z^{\star,\eps})-\delta_A^{-1}(\cPo (z^\ell)-\zs)\|_P+\|\zs-z^{\star,\eps}\|_P\\
            &\stackrel{\text{(II)}}{\leq} \xi/2 \|\Delta_\eps^{-1}(\cPe (z^\ell)-z^{\star,\eps})-\delta_A^{-1}(\cPo (z^\ell)-\zs)\|_P+\eps\\
            &=\xi(2\delta_A\Delta_\eps)^{-1}\|\delta_A(\cPe (z^\ell)-z^{\star,\eps})-\Delta_\eps(\cPo (z^\ell)-\zs)\|_P+\eps\\
            &=\xi(2\delta_A\Delta_\eps)^{-1}\|\delta_A(\cPe (z^\ell)-\cPo (z^\ell))+(\delta_A-\Delta_\eps)\cPo (z^\ell)+\delta_A (\zs-z^{\star,\eps})+(\Delta_\eps-\delta_A)\zs\|_P\\
            &\hspace{2.9 cm}+\eps\\
            &\stackrel{\text{(III)}}{\leq} \xi (2\delta_A\Delta_\eps)^{-1} (\delta_A \|\cPe (z^\ell)-\cPo (z^\ell)\|_P+|\Delta_\eps-\delta_A|\|\cPo (z^\ell)\|_P+\delta_A\|\zs-z^{\star,\eps}\|_P\\
            &\hspace{2.9 cm}+|\delta_A-\Delta_\eps|\|\zs\|_P)+\eps\\
            &\stackrel{\text{(IV)}}{\leq} \xi (2\delta_A\Delta_\eps)^{-1} \eps(2\delta_A+\|\cPo (z^\ell)\|_P+\|\zs\|_P)+\eps\\
            &\stackrel{\text{(V)}}{\leq} \eps(1+\xi \delta_A^{-2}(2\delta_A+\|\cPo (z^\ell)\|_P+\|\zs\|_P))\\
            &\to 0\quad\text{as}\quad \eps\to 0\,,
        \end{aligned}
    \end{equation}
    where (I) follows from the triangle inequality and~\eqref{eq:barZdef}; (II) from the definition of \(z^{\star,\eps}\); (III) from the Triangle Inequality and homogeneity; (IV) from the definitions of \(z^{\star,\eps}\), \(\Delta_\eps\), and \(\cP_\eps z^\ell\); and (V) from the inequalities \(\Delta_\eps\geq \delta_A-\eps\geq\delta_A/2\). In particular, \(\bar{z}_P\in\cl(\sol_P)\). 
    %By the construction of \(\bar{z}\), \(\cPo (z^\ell)-\zs\in \cN_{\cballP{\zs}{\xi/2}}(\bar{z})\).
    Then, consider the halfspace
    \begin{align*}
        \cH_+&=\{z\in\RR^{n+m}:\langle z-\bar{z},\bar{z}-\zs\rangle_P \geq0\}\,.
    \end{align*}
    % \textit{Claim:} The point \(z^\ell\) belongs to \(\cH\). \\
    % \noindent\textit{Proof of claim:}
    We claim that $z^\ell\in \cH_+$. Suppose to the contrary that \(z^\ell\not\in\cH_+\). Note that, by construction
    \(
        \cPo (z^\ell)-\zs=2\delta_A/\xi\cdot(\bar{z}-\zs)
    \), whence \(\cPo (z^\ell)-\bar{z}=(\cPo (z^\ell)-\zs)-(\bar{z}-\zs)=(2\delta_A\xi^{-1}-1)(\bar{z}-\zs)\), and then \(\langle z^\ell-\bar{z},\cPo (z^\ell)-\bar{z}\rangle_P < 0\) since \(z^\ell\not\in\cH^+\) and \(2\delta_A/\xi>1\). It follows that
    \begin{align*}
        \|z^\ell- \cPo (z^\ell)\|_P^2&= \|z^\ell-\bar{z}-(\cPo (z^\ell)-\bar{z})\|_P^2\\
        &=\|z^\ell-\bar{z}\|_P^2-2\langle z^\ell-\bar{z},\cPo (z^\ell)-\bar{z}\rangle_P+ \|\cPo (z^\ell)-\bar{z}\|_P^2\\
        &>\|z^\ell-\bar{z}\|_P^2\,.
    \end{align*}
    This is contradiction to the definition of \(\cPo (z^\ell)\), since \(\bar{z}_P\in\cl(\sol_P)\subseteq\cl((\sols)_P)\). Therefore, \(z^\ell\in \cH_+\). Now, let \(\tilde{z}=(\bar{z}+\zs)/2\), and consider the following halfspace
    \[
        \cH_-=\{z\in\RR^{n+m}:\langle z-\tilde{z},\tilde{z}-\zs\rangle_P \leq 0\}\,.
    \]
    We will show that \(z^{\ell+1}\in\cH_-\). From the definition of \(n\), \(z^{\ell+1}\in \cballP{\zs}{\xi/2}\). Furthermore,
    \begin{align*}
        \|z^{\ell+1}-\bar{z}\|_P&\geq \|z^{\ell+1}-\bar{z}^\eps\|_P-\|\bar{z}-\bar{z}^\eps\|_P & \text{(Triangle inequality)}\\
    &\geq \|\zs-\bar{z}^\eps\|_P-\|\bar{z}-\bar{z}^\eps\|_P & \text{(Assumption~\ref{assumption:finiteTimeIdentification} \((ii)\))}\\
    &\geq \|\zs-\bar{z}\|_P-2\|\bar{z}-\bar{z}^\eps\|_P & \text{(Triangle inequality)}\\
    &\geq \xi/2 -2\|\bar{z}-\bar{z}^\eps\|_P & \text{(Definition of \(\bar{z}\))}\\
    &\to\xi/2\quad\text{as}\quad\eps\to 0 & \text{(From~\eqref{eq:zbarzeps})}\,.
    \end{align*}
    Since \(\eps\) is arbitrarily small, \(z^{\ell+1}\not\in\ballP{\bar{z}}{\xi/2}\). Then, to conclude that \(z^{\ell+1}\in\cH_-\) it suffices to show that \(\cballP{\zs}{\xi/2}\cap\cH_-^c\subseteq\ballP{\bar{z}}{\xi/2}\). 
    % \begin{lemma}
    % \label{lemma:ballIntersection}
    %     Let \(P\in\RR^{n\times n}\) be  a positive semidefinite matrix and \(x,y\in\RR^n\) points such that \(\|x-y\|_P=\xi>0\). Define the halfspace
    %     \[
    %         \cH=\{z\in\RR^n:\langle z-2^{-1}(x+y),2^{-1}(x-y)\rangle_P\leq 0\}\,.
    %     \]
    %     Then, \(\cballP{y}{\xi}\cap\cH^c\subseteq\ballP{x}{\xi}\).
    % \end{lemma}
    To prove this, let \(u=2/\xi\cdot (\bar{z}-\zs)\). Note that \(\tilde{z}-\zs=(\bar{z}-\zs)/2=\xi/4\cdot  u\) and \(z-\tilde{z}=z-(\bar{z}+\zs)/2=(z-\zs)-(\bar{z}-\zs)/2=\xi/2\cdot(2/\xi\cdot (z-\zs) - u/2)\). Hence, \(\langle z,-\tilde{z},\tilde{z}-\zs\rangle_P\leq 0\) if and only if \(\langle 2/\xi\cdot(z-\zs)-u/2,u\rangle_P\leq 0\). It follows that 
    \begin{align*}
        \cH_-&=\zs+\xi/2\cdot\{z\in\RR^n:\langle z-u/2,u\rangle_P \leq 0\}\\
        \cballP{\zs}{\xi/2}&=\zs+\xi/2\cdot \cballP{0}{1},\quad\text{and}\\
        \ballP{\bar{z}}{\xi/2}&=\zs+\xi/2\cdot \ballP{u}{1}\,.
    \end{align*}
    Then, by the possibility of translating by \(-\zs\) and then scaling by \(2/\xi\), assume without loss of generality \(\xi/2=1\) and \(\zs=0\), implying \(u=\bar{z}\) in the above representation.
    Let \(z\in \cballP{0}{1}\cap\cH_-^c\). 
    % Since \(z\in\cballP{0}{1}\), \(\|x-z\|_P\leq \|x\|_P+\|z\|_P\leq 2\). 
    Since \(z\in\cH_-^c\), \(2\langle z,\bar{z}\rangle_P > \|\bar{z}\|_P^2=1\). Then,
    \(
            \|\bar{z}-z\|_P^2=\|\bar{z}\|_P^2-2\langle \bar{z},z\rangle_P+\|z\|_P^2< 2-1=1\,.
    \)
    Therefore, \(z\in\ballP{\bar{z}}{1}\)\,. Since \(z\in\cballP{0}{1}\cap\cH_-^c\) is arbitrary, \(\cballP{0}{1}\cap\cH_-^c\subseteq\ballP{\bar{z}}{1}\). Thus, \(z^{\ell+1}\in\cH_-\). Now, we claim that \(
        \dist_P(\cH_+,\cH_-)=\xi/4\,.
    \)
%     \begin{lemma}
% \label{lemma:halfspaceSeparation}
%     Let \(P\in\RR^{n\times n}\) be  a positive semidefinite matrix and \(x,y\in\RR^n\) points such that \(\|x-y\|_P=\xi>0\). Define the halfspaces
%     \[
%         \cH_-=\{z\in\RR^n:\langle z-2^{-1}(x+y), 2^{-1}(x-y)\rangle_P\leq 0\}\quad\text{and}\quad \cH_+=\{z\in\RR^n:\langle z-x, x-y\rangle_P\geq 0\}\,.
%     \]
%     Then, \(\dist_P(\cH_-,\cH_+)=\xi/2\).
% \end{lemma}
    Similarly as before, for \(u=2/\xi\cdot(\bar{z}-\zs)\) we have
    \begin{align*}
        % \cH_-&=\zs+\xi/2\cdot\{z\in\RR^n:\langle z-2^{-1}u,u\rangle_P \leq 0\}\,.
        % ,\quad\text{and}\\
        \cH_+=\zs+\xi/2\cdot\{z\in\RR^n:\langle z-u,u\rangle_P\geq0\}\,.
    \end{align*}
    Hence, assume without loss of generality \(\xi/2=1\) and \(\zs=0\).
    Let \(z^-=\bar{z}/2\in\cH_-\) and \(z^+=\bar{z}\in\cH_+\). Clearly, \(\|z^--z^+\|_P=\|\bar{z}\|_P/2=1/2\). Therefore \(\dist_P(\cH_-,\cH_+)\leq 1/2\). For the reverse inequality, let \(z^-\in\cH_-\) and \(z^+\in\cH_+\).  Denote
    \[
        z^\pm=z^\pm_{\bar{z}}+z^\pm_{\bar{z}^\perp}\quad\text{where}\quad z^\pm_{\bar{z}}=\bar{z}\left\langle z^\pm,\|\bar{z}\|_P^{-1}\bar{z}\right\rangle_P\quad\text{and}\quad z^\pm_{\bar{z}^\perp}=z^\pm-z^\pm\,,
    \]
    where $\pm$ represents \(+\) or \(-\). Then, by orthogonality, we have
    \begin{align*}
        \|z^+-z^-\|_P^2&=\|z^+_{\bar{z}}-z^-_{\bar{z}}\|_P^2+\|z^+_{\bar{z}^\perp}-z^-_{\bar{z}^\perp}\|_P^2\\
        &\geq \|z^+_{\bar{z}}-z^-_{\bar{z}}\|_P^2\\
        &=\left|\langle z^+-z^-,\|\bar{z}\|_P^{-1}\bar{z}\rangle_P\right|\|\bar{z}\|_P\\
        &\geq|\langle z^+,\bar{z}\rangle_P|-|\langle z^-,\bar{z}\rangle_P|\\
        &\geq \|\bar{z}\|_P^2(1-1/2)\\
        &=1/2\,.
    \end{align*}
    Therefore, \(\dist_P(\cH_-,\cH_+)\geq 1/2=\xi/4\). It follows that \(\|z^{\ell+1}-z^\ell\|\geq \xi/4\). But, from  Step \(1\) we have \(\|z^\ell-z^{\ell+1}\|_P<\xi/4\), a contradiction.
    \paragraph{Step 4.}  Combining Steps 2 and 3, it holds for any \(k> K\) that \(z^k\in\ballP{\zs}{\xi/2}\).
    \paragraph{Step 5.} From Assumption~\ref{assumption:finiteTimeIdentification} \((i)\) we have
    \begin{align*}
        \|\tilde{z}^k-\zs\|_P\leq \gamma\dist_P(z^0,\sol)/k^{\rate}\,.
    \end{align*}
    Then,  whenever \(k> K\geq (\gamma\dist_P(z^0,\sol)/\xi)^{\irate}\) we obtain \(\|\tilde{z}^k-z^k\|_2< \xi/2\), whence
    % \[
    %     \|\tilde{z}^k-z^k\|_2\leq \frac{c\dist_P(z^0,\sol)}{k^\rate}
    % \]
    \begin{align*}
        \|\tilde{z}^k-\zs\|_P\leq \|\tilde{z}^k-z^k\|_P+\|z^k-\zs\|_P< \xi/2+\xi/2=\xi\,.
    \end{align*}
    \paragraph{Step 6.}   Using Steps \(4\) and \(5\) we have that \(z^{k},\tilde{z}^k\in\ballP{\zs}{\xi}\) for all \(k>K\). Then, by Proposition~\ref{prop:identificationFinalStep} the iterates need at most \(\lceil \step+\max\{\Lyjd(\eta C_j)^{-1}:j\in\indexN\}\rceil\) more steps to reach \(\cM\).\hfill \qed

    % In this section, we provide a short proof of Theorem~\ref{theorem:twoStageConvergence}. The proof of Theorem~\ref{theorem:linearConvergenceFast} is completely analogous to that of Theorem~\ref{theorem:linearConvergence}, whence it is omitted. 
    \subsection{Proof of Theorem~\ref{theorem:twoStageConvergence}}
    \label{appendix:proof-twoStageConvergence}
    From Proposition~\ref{theorem:finiteTimeIdentification}, the \(k\)th iterate generated by update~\eqref{eq:algorithm} satifies \(z^k\in\ballP{\zs}{\delta_A/2}\cap\cM\) whenever
    \[
           k >K:=\left(\max\left\{1,\frac{1}{\alpha_L}\right\}\frac{8\lambdamax(P)^{\frac{3}{2}}\dist_2(z^0,\sol)}{\delta_A}\right)^{\irate}+\left\lceil \max_{j\in\indexN}\frac{\Lyjd}{\eta \overline \Lxjd} \right\rceil\,.\footref{fn:delta} 
    \]
    Then, from Proposition~\ref{prop:linearConvergence}, for all \(k>K\) we have
    \[
        \dist_2(z^k,\sol)\leq \sqrt{e}\nu \exp\left(-\frac{1}{2}\frac{k}{\lceil e\nu^2 \rceil}\right)\frac{\delta_A}{\lambdamax(P)^{\frac{1}{2}}}\quad\text{where}\quad \nu =\frac{\gamma\lambdamax(P)}{\alpha_G}\,.
    \]
    Rearranging, we recover the result of Theorem~\ref{theorem:twoStageConvergence}.

\subsection{Proof of Proposition~\ref{prop:constantComparison}}
        \paragraph{\((\alpha_\indexM\geq\alpha_G)\)} Clearly \(\cD_{\delta_A}\cap\cM\subseteq\cD\),\footref{fn:delta} whence 
        \[
             \alpha_G=\inf_{z\in\cD}\frac{\dist_2(0,\sub(z))}{\dist_2(z,\sol)}\leq \inf_{z\in\cD_{\delta_A}\cap\cM}\frac{\dist_2(0,\sub(z))}{\dist_2(z,\sols)}=\alpha_{\indexM}\,.
        \]
        \paragraph{\((\alpha_L\geq\alpha_G)\)} By construction, \(\sol\subseteq\sols\), and then \(\dist_2(z,\sols)\leq \dist_2(z,\sol)\) for all \(z\in\RR^{n+m}\). Hence,
        \[
            \alpha_G=\inf_{z\in\cD}\frac{\dist_2(0,\sub(z))}{\dist_2(z,\sol)}\leq \inf_{z\in\cD}\frac{\dist_2(0,\sub(z))}{\dist_2(z,\sols)}=\alpha_L\,.
        \]
        \paragraph{\((\alpha_\indexM\geq\alpha_L)\)} Let \(z\in\ballP{\zs}{\delta_A/2}\). Let \(\zstilde\in\sols\) such that \(\|z-\zstilde\|_2=\dist_2(z,\sols)\). Suppose to the contrary that \(\zstilde\not\in\ballP{\zs}{\delta_A}\). In particular,  \(\zstilde\not\in\cballP{\zs}{\zeta}\) for \(\zeta=\|z-\zs\|_P\). Let \(\zeta_P=\lambdamin(P)^{-\frac{1}{2}}\zeta/2\). 
        %For any \(z\in\RR^{n+m}\) denote \(\cP z:=\proj_{\cballt{\zs}{\zeta_P}}z\). 
        From Proposition~\ref{prop:normEquivalence} and since \(\kappa(P)\leq 4\),
        \[
            \|\zstilde-\zs\|_2\geq\lambdamax(P)^{-\frac{1}{2}}\|\zstilde-\zs\|_P>\lambdamax(P)^{-\frac{1}{2}}\zeta=\kappa(P)^{-\frac{1}{2}}\lambdamin(P)^{-\frac{1}{2}}\zeta\geq \lambdamin(P)^{-\frac{1}{2}}\zeta/2=\zeta_P\,.
        \]
        Thus, \(\zstilde\not\in\cballt{\zs}{\zeta_P}\), whence \(\proj_{\cballt{\zs}{\zeta_P}}\zstilde\neq \zstilde\). Conversely, from Proposition~\ref{prop:normEquivalence} we obtain
        \[
            \|z-\zs\|_2\leq \lambdamin(P)^{-\frac{1}{2}}\|z-\zs\|_P\leq \lambdamin^{-\frac{1}{2}}\zeta/2=\zeta_P\,.
        \]
        Then, \(z\in\cballt{\zs}{\zeta_P}\), whence \(\proj_{\cballt{\zs}{\zeta_P}} z=z\). From the firm-non-expansiveness of the projection, we obtain
        \begin{align*}
            \|\zstilde-z\|_2^2&\geq \|\proj_{\cballt{\zs}{\zeta_P}}\zstilde-\proj_{\cballt{\zs}{\zeta_P}}z\|_2^2+\|(\zstilde-\proj_{\cballt{\zs}{\zeta_P}}\zstilde)-(z-\proj_{\cballt{\zs}{\zeta_P}} z)\|_2^2\\
            &=\|z-\proj_{\cballt{\zs}{\zeta_P}}\zstilde\|_2^2+\|\zstilde-\proj_{\cballt{\zs}{\zeta_P}}\zstilde\|_2^2\\
            &>\|z-\proj_{\cballt{\zs}{\zeta_P}}\zstilde\|_2^2\,.
        \end{align*}
        This contradicts that, by construction, \(\zstilde=\argmin\{\|z-z'\|_2:z'\in\sols\}\).
        Therefore,  \(\zstilde\in\ballP{\zs}{\delta_A}\). Moreover, from Lemma~\ref{prop:reducedSystemIdentification}, \(\zstilde\in\sol\). Hence, \( \dist_2(z,\sol)\leq\dist_2(z,\sols)\) for any \(z\in\ballP{\zs}{\delta_A/2}\). Then, since \(\cD_{\delta_A}\subseteq\ballP{\zs}{\delta_A/2}\), we obtain
        \[
            \alpha_{L}\leq\inf_{z\in\cD_{\delta_A}}\frac{\dist_2(0,\sub(z))}{\dist_2(z,\sols)}\leq \inf_{z\in\cD_{\delta_A}}\frac{\dist_2(0,\sub(z))}{\dist_2(z,\sol)}\leq\inf_{z\in\cD_{\delta_A}\cap\cM}\frac{\dist_2(0,\sub(z))}{\dist_2(z,\sol)}=\alpha_{\indexM}\,.
        \]
        Completing the proof of Proposition~\ref{prop:constantComparison}. \hfill \qed

    \subsection{Derivation of the results from Example~\ref{ex:rotatedHouse}}
    \label{appendix:example}

    %\hnote{Please make a formal statement here that we show for that example $$\alpha_G\leq\min\{c_1,c_2\},\quad0.037\leq\alpha_L\leq 0.44,\quad\text{and}\quad\alpha_{\indexM}=1.$$
    %What is stated below is more than what is needed to support the main paper. It is unclear to me whether they are relavant. My suggestion is to keep only the relavant proof showing 
    %$$\alpha_G\leq\min\{c_1,c_2\},\quad0.037\leq\alpha_L\leq 0.44,\quad\text{and}\quad\alpha_{\indexM}=1.$$
    %}
    In this section, we prove that the problem of Example~\ref{ex:rotatedHouse} satisfies
    \begin{equation}
    \label{eq:exampleBounds}
        \alpha_G\leq\min\{c_1,c_2\},\quad0.037\leq\alpha_L\leq 0.44,\quad\text{and}\quad\alpha_{\indexM}=1.
    \end{equation}
    To do this, rewrite the problem as
    \[
        \min_{x\in\R^2} \langle c,x\rangle\quad
        \text{s.t.}\quad Ax\leq  b,\quad\text{where}\quad A=-\begin{bmatrix}
            c_1 & c_2\\
            1 & 0\\
            0 & 1
        \end{bmatrix},\quad b=-\begin{bmatrix}
            \|c\|_1\\
            0\\
            0
        \end{bmatrix}\,.
    \]
    By inspection, it is easy to see that
    \[
        \sol=\sol_x\times\sol_y\quad\text{where}\quad\sol_x=\{x\in\RR^{2}_+:\langle c,x\rangle =\|c\|_1\}\quad\text{and}\quad\sol_y=\{(1,0,0)\}\,.
    \]
    We now show an explicit form for the distance to zero of the saddle subdifferential. Let \(z=(x,y)\in\RR^n\times \RR^m_+\) and \(I=\{j\in[m]:y_j=0\}\). Using Proposition~\ref{prop:saddleSubForm}, we obtain
    \begin{equation}
        \label{eq:subDistance}
        \begin{aligned}
        \dist_2^2(0,\sub(z))&=\|c+A^\top y\|_2^2+\|\max(0, (Ax-b)_I)\|_2^2+\|(Ax-b)_{I^c}\|_2^2\,.\\
        &=\|(1-y_1)c-y_2e_2-y_3e_3\|_2^2+\left\|\left(\begin{bmatrix}\langle c,x\rangle-\|c\|_1\\
        x_1\\
        x_2
        \end{bmatrix}_I\right)_-\right\|_2^2+\left\|\begin{bmatrix}\langle c,x\rangle-\|c\|_1\\
        x_1\\
        x_2
        \end{bmatrix}_{I^c}\right\|_2^2
        \end{aligned}
    \end{equation}
    where \(e_i\) is the \(i\)-th canonical vector.
    
    The following analysis holds for \(P=I\) and any solution \(\zs\in\relint(\sol)\).
    \paragraph{Estimation of \(\boldsymbol{\alpha_{\indexM}}\).} Let \(z=(x,y)\in\ballt{\zs}{\delta}\cap\cM\). Note that
    \[
        \cM=\{(x,y)\in\mathbb{R}^{n+m}:x_1,x_2>0,\,y_1>0,\,y_{2},y_3=0\}\,.
    \]
    Then \(I=\{2,3\}\) and~\eqref{eq:subDistance} becomes
    \[
        \dist_2^2(0,\sub(z))=\|c\|_2^2(1-y_1)^2+\left(\langle c,x\rangle -\|c\|_1\right)^2.
    \]
    Furthermore, since \(z\in\ballt{\zs}{\delta}\cap\cM\),
    \[
        \dist_2^2(z,\sol)=\dist_2^2(y,\sol_y)+\dist_2^2(x,\sol_x)=(1-y_1)^2+\left(\frac{\langle c,x\rangle-\|c\|_1}{\|c\|_2}\right)^2\,.
    \]
    But \(\|c\|_2=1\), hence
    \[
        \alpha_{\indexM}=\inf_{z\in\cD_{\delta}\cap\cM}\frac{\dist_2(0,\sub(z))}{\dist_2(z,\sol)}=1\,.
    \]
    \paragraph{Estimation of \(\boldsymbol{\alpha_G}\).} We provide upper bound estimates for \(\alpha_G\) in terms of \(\tau\) and \(\min\{c_1,c_2\}\).
    {\em Dependence on \(\min\{c_1,c_2\}\).} Assume without loss of generality that \(c_1\geq c_2\). For any \(\eps\in(0,\tau)\) let \(\bar{z}=(\bar{x},\bar{y})\in\mathbb{R}^{n+m}\) such that \(\bar{x}=(0,(1+\frac{c_1}{c_2})+\eps)\) and \(\bar{y}=(1,0,0)\). Note that \(\bar{z}\in\cD\). Moreover \(c+A^\top \bar{y}=0\) and \(I=\{2,3\}\), whence~\eqref{eq:subDistance} becomes
    \[
        \dist_2^2(0,\sub(\bar{z}))=|(A\bar{x}-b)_1|^2=((c_2+c_1)+c_2\eps-\|c\|_1)^2=(c_2\eps)^2
    \]
    On the other hand, we have
    \(
        \dist_2(\bar{z},\sol)=\eps
    \). Gathering up, we obtain that
    \[
        \alpha_G=\inf_{z\in\cD}\frac{\dist_2(0,\sub(z))}{\dist_2(z,\sol)}\leq \frac{\dist_2(0,\sub(\bar{z}))}{\dist_2(\bar{z},\sol)}=\min\{c_1,c_2\}\,.
    \]
    {\em Dependence on \(\tau\).} For any \(\eps>0\) let \(\tau_\eps=\tau-\eps\) and \(\bar{z}^\eps =(\bar{x},\bar{y})\in\RR^{n+m}\) such that
    \(
        \bar{x}^\eps=\xs+\tau_\eps c
    \) and \(\bar{y}=0\). Note that \(\bar{z}\in\cD\).  Moreover, in view of~\eqref{eq:subDistance}, we have \(I=\{1,2,3\}\). Hence,
    \[
        \dist_2^2(0,\sub(\bar{z}))=\|c\|_2^2=1\,.
    \]
    Furthermore, we have
    \[
        \dist_2^2(\bar{z},\sol)=\dist_2^2(\bar{x},\sol_x)+\dist_2^2(\bar{y},\sol_y)=\left(\frac{\langle c,\bar{x}\rangle-\|c\|_1}{\|c\|_2}\right)^2+1=1+\tau_\eps^2\,.
    \]
    Gathering up, we obtain that \(\alpha_G=O(\tau^{-1})\):
    \[
        \alpha_G=\inf_{z\in\cD}\frac{\dist_2(0,\sub(z))}{\dist_2(z,\sol)}\leq \frac{\dist_2(0,\sub(\bar{z}))}{\dist_2(\bar{z},\sol)}=\frac{1}{\sqrt{1+\tau_\eps^2}}\xrightarrow[]{\eps\to0} \frac{1}{\sqrt{1+\tau^2}}\,.
    \]
    \paragraph{Estimation of \(\boldsymbol{\alpha_L}\).}
    % Note that, in this case, the system of equations~\eqref{eq:reducedSystemIdentification} becomes
    % \[
    %     \begin{cases}
    %         \langle c,x\rangle+\langle b,y\rangle&\leq 0\\
    %         c+A^\top y&=0\\
    %         (Ax-b)_1&\leq 0\\
    %         y_2,\, y_3&\geq 0
    %     \end{cases}  \quad\Leftrightarrow\quad \begin{cases}
    %        \langle c,x\rangle & \leq \|c\|_1\delta_A y_1\\
    %         \begin{bmatrix}
    %             c_1(1-y_1)-y_2\\
    %             c_2(1-y_1)-y_3
    %         \end{bmatrix}&=0\\
    %         \langle c,x\rangle &\geq \|c\|_1\delta_A \\
    %         y_2,\, y_3&\geq 0
    %     \end{cases}
    % \]
    % From the RHS equations, we deduce the following: the first and the third equations imply \(y_1\geq 1\). This and the second equation give \(y_1,y_2\leq 0\). Then, we obtain \(y_2=y_3=0\) from the fourth equation. Therefore, \(y_1=1\) from the second equation. Furthermore, this and the first and third equations imply \(\langle c,x\rangle=\|c\|_1\delta_A\). 
    Invoking Proposition~\ref{prop:localSystemFiniteIdentification} we obtain
    \[
        \sols=\mathcal{S}_{\zs,x}\times\mathcal{S}_{\zs,y\quad}\text{where}\quad\mathcal{S}_{\zs,x}=\{x\in\RR^{2}:\langle c,x\rangle=\|c\|_1\}\quad\text{and}\quad \mathcal{S}_{\zs,y}=\{(1,0,0)\}\,.
    \]
    Hence, following the analysis of \(\alpha_G\), we derive 
    \(
        \alpha_L\leq 1/\sqrt{1+\tau^2}\,.
    \) In particular, when \(\tau=2\) we obtain \(\alpha_L\leq 0.44\)
    However, we can not follow our previous analysis for \(\alpha_G\) to obtain \(\alpha_L\leq \min\{c_1,c_2\}\), as the \(\bar{z}\) chosen in that case satisfies 
    \[
        \dist_2^2(\bar{z},\sols)=\left(\frac{\langle c,\bar{x}\rangle-\|c\|_1}{\|c\|_2}\right)^2=((c_2+c_1)+c_2\eps-\|c\|_1)^2=(c_2\eps)^2=\dist_2^2(0,\sub(\bar{z}))\,.
    \]
    Nevertheless, in the following we will show that \(\alpha_L\gg\min\{c_1,c_2\}\) when \(\min\{c_1,c_2\}\) is small.
    \paragraph{Numerical lower bound for \(\boldsymbol{\alpha_L}\).} 
    % In our derivation of the upper bound \(\alpha_G\leq \min\{c_1,c_2\}\), we have used points that are far from the solution point \(\zs\) in instance \((a)\), requiring setting \(R\geq \delta_A(1+\frac{c_1}{c_2})+1\gg 1\). To decouple the effects of \(\min\{c_1,c_2\}\) being small and \(R\) being big in the numerical experiments, we numerically evaluate instance \((b)\), which does not suffer from the aforementioned problem.
    To check that \(\alpha_L\gg \min\{c_1,c_2\}\) when \(\min\{c_1,c_2\}\) is small, let  \[
        G=\{z\in\RR^{n+m}:Hz\leq 0\},\quad\text{where}\quad H=\begin{bmatrix}
            A_1 & 0_{1\times 3}\\
            0_{2\times 2} & A^\top  \\
            0_{2\times 2} & -A^\top  \\
            0_{2\times 2} & I_3\\
            \frac{1}{R}c^\top  & \frac{1}{R}b^\top 
        \end{bmatrix}\quad\text{and}\quad I_3=\begin{bmatrix}
            0 & -1 & 0\\
            0 & 0 & -1
        \end{bmatrix}\,.
    \]
 From
    %~\cite{Nearly optimal linear convergence of stochastic primal-dual methods for linear programming}
    ~\cite[Lemma 3 (c)]{lu2024geometry} and the invariance to translations of the sharpness constants, for any \(\zstilde\in\sol\) we have
    \[
        \inf_{z\in\cballt{\zstilde}{\tau}}%\cap\YY}
        \frac{\dist_2(0,\sub(z))}{\dist_2(z,\sol)}\geq\inf_{z\in\RR^{n+m}}\frac{\dist_2(Hz,\RR^{m}_-)}{\dist_{2}(z,G)}=:\cH\,,
    \]
    where \(\cH\) is the {\em sharpness constant} (the inverse of the {\em Hoffman constant}) of the homogeneous linear system \(Hz\leq 0\). Then, taking into account that \(\cD=\sol+\tau\ball\),
    %(\sol+\tau\ball)\cap\YY\)
    we obtain
    \[
        \alpha_L=\inf_{z\in\cD}\frac{\dist_2(0,\sub(z))}{\dist_2(z,\sol)}\geq \inf_{\zstilde\in\sol}\inf_{z\in\cballt{\zstilde}{\tau}}%\cap\YY}
        \frac{\dist_2(0,\sub(z))}{\dist_2(z,\sol)}\geq \cH\,.
    \]
    Furthermore, using the procedure proposed by Peña~\cite{pena2024easily}, we can numerically lower bound \(\mathcal{H}\). This way, when \(\min\{c_1,c_2\}=10^{-7}\) we obtain numerically \(
        \alpha_L\geq \cH\geq 0.036.
    \)
    \paragraph{Auxiliary results.} We prove two auxiliary results used in the derivation of Example~\ref{ex:rotatedHouse}. The first gives an exact form for the distance to zero of the saddle subdifferential \(\sub\) of the minimax problem~\eqref{eq:minimaxProblem}. The second relates the set \(\sols\) defined in~\eqref{eq:reducedSystemIdentification} with the solution set of a reduced optimization problem.
\begin{proposition}
    \label{prop:saddleSubForm}
    For any \(z=(x,y)\in\RR^{n+m}\), the distance to zero of the saddle subdifferential of the minimax problem~\eqref{eq:minimaxProblem} has the following form
    \[
        \dist_2^2(0,\sub(z))=\begin{cases}
            \dist_2^2(0,\partial f(x)+ J_G^\top (x)y)+\|(G(x)_{I})_+\|_2^2+\|G(x)_{[m]\backslash I}\|_2^2 & \text{if } y\geq0\\
            \infty & \text{otherwise}
            \end{cases}\,.
    \]
    where \(I=\{j\in[m]:y_j=0\}\) and \(J_G(x)_j=\partial g_j(x)\) for all \(j\in[m]\).
\end{proposition}

\begin{proof}
    Let \(z=(x,y)\in\RR^{n+m}\). Recall that
    \[
        \sub(x,y)=\begin{bmatrix}\partial_x \cL(x,y),\\
        \partial_y (-\cL)(x,y)\end{bmatrix}
        \quad\text{where}\quad \cL(x,y)=f(x)+\langle y,G(x)\rangle-\iota_{\RR^m_+}(y)\,.
    \]
    By subdifferential calculus, we obtain
    \begin{align*}
        \partial_x \cL(x,y)&=\partial f(x)+ J_G^\top (x)y\quad\text{and}\quad
        \partial_y (-\cL)(x,y)=\begin{cases}
         -G(x)+N_{\RR^m_+}(y) & \text{if }y\geq 0\\
         \emptyset & \text{otherwise}
         \end{cases}\,.
    \end{align*}
    Therefore, if \(y_j<0\) for some \(j\in[m]\) then \(\sub(z)=\emptyset\), whence \(\dist_2(0,\sub(z))=\infty\). Then, suppose \(y\geq 0\) and
    fix \(j\in[m]\). Note that  \(N_{\RR_+^m(y)}=\bigtimes_{j=1}^mN_{\RR_+(y_j)}\).
    Moreover,
    \[
        \min_{t\in N_{\RR_+}(y_j)}|G(x)_j-t|^2=\begin{cases}
            \max\{0, G(x)_j\}^2 & \text{if }y_j=0\\
            |G(x)_j|^2 & \text{otherwise}
        \end{cases}\,.
    \]
    It follows that
    \begin{align*}
        \dist_2^2(0,\partial_y (-\cL)(x,y))&=\min_{z\in-N_{\RR^m_+}(y)}\sum_{j=1}^m|G(x)_j+z_j|^2\\
        &=\sum_{j=1}^m\min_{z_j\in-N_{\RR_+}(y_j)}|G(x)_j+z_j|^2\\
        &=\|(G(x)_{I})_+\|_2^2+\|G(x)_{[m]\backslash I}\|_2^2\,.
    \end{align*}
    Gathering up, we obtain
    \begin{align*}
        \dist_2^2(0,\sub(z))&=\dist_2^2(0,\partial_x \cL(x,y))+\dist_2^2(0,\partial_y (-\cL)(x,y))\\
        &=\dist_2^2(0,\partial f(x)+ J_G^\top (x)y)+\|(G(x)_{I})_+\|_2^2+\|G(x)_{[m]\backslash I}\|_2^2\,.
    \end{align*}
\end{proof}

 Consider the simplified primal problem that only incorporates the constraints indexed by \(\indexB\),
  \begin{equation}
    \label{eq:reducedPrimalProblem}
        p^\star_{\indexB} = \begin{cases}\displaystyle\min_{x\in\R^n} & f(x)\\
        \text{s.t.} &  g_j(x)\leq 0 \quad \text{for all }j \in\indexB\, .
        \end{cases}
    \end{equation}
Denote  \(\sol_\indexB\subseteq\RR^{n+q}\), where \(q=|\indexB|\), as the set of primal-dual solutions to the minimax problem associated to~\eqref{eq:reducedPrimalProblem}
\[ 
    \min_{x\in\RR^n}\max_{y\in\RR^q_+}f(x)+\langle y,G_\indexB(x)\rangle\,.
\]
 The following proposition establishes that the set \(\sols\) is equal to \(\sol_\indexB\times \{0_\indexN\}\) up to dual coordinates reordering.
 
\begin{proposition}
\label{prop:localSystemFiniteIdentification}
        Assume that \(\indexB=\{1,\dots, q\}\). Then, \(\sols=\sol_\indexB\times\{0_\indexN\}\).
    \end{proposition} 

\begin{proof}
 Let \(\bar{\cS}_B\subseteq\RR^{n+m}\) be the solution to the following system
        \begin{align*}
            f(x)-h_{\indexB}(y)\leq 0,\quad
            G_{\indexB}(x) \leq 0,\quad
            y_{\indexB} \geq 0\,\quad\text{and}\quad y_\indexN=0\,,
        \end{align*}
        where  \(h_\indexB:\RR^m_+\to\RR\cup\{+\infty\}\) is such that \(h_{\indexB}(y)=\min_{x\in\RR^n}f(x)+\sum_{j\in \indexB}y_jg_j(x)\) is, modulo \(y_\indexN\), the dual function of the reduced problem~\eqref{eq:reducedPrimalProblem}. By construction, \(\bar{S}_\indexB=\sol_\indexB\times\{0_\indexN\}\).
 Let \(z=(x,y)\in\sols\). If \(z\in\cballP{\zs}{\delta_A/2}\),  from Propositions~\ref{prop:reducedSystemIdentification} and~\ref{prop:LipchitzGuarantee} we obtain \(z\in\sol\) and \(G_N(x)<0\), respectively. By complementary slackness, \(y_\indexN=0\). If \(z\not\in\cballP{\zs}{\delta_A/2}\), let \(\theta=\delta_A/(2\|z-\zs\|_P)\in(0,1)\) and
    \[
        \bar{z}=(\bar{x},\bar{y}):=(1-\theta)\zs+\theta z=\zs+\frac{\delta_A}{2}\frac{z-\zs}{\|z-\zs\|_P}\in\cballP{\zs}{\delta_A}\,.
    \] 
    Since \(\sols\) is convex, \(\bar{z}\in\sols\). Therefore, as argued in the first case, \(\bar{y}_\indexN=0\). Hence, since \(\theta\neq 0\),
    \(
        y_\indexN=\theta^{-1}(\bar{y}_\indexN-(1-\theta)\ys_\indexN)=0\,.
    \)
    Furthermore,
    \[
        h_B(y)=\min_{x\in\RR^n} f(x)+\sum_{j\in\indexB}y_jg_j(x)=\min_{x\in\RR^n}f(x)+\sum_{j\in[m]}y_jg_j(x)=h(y)\,.
    \]
    Then, \(z\in\bar{\cS}_\indexB\), whence \(\sols\subseteq\bar{\cS}_\indexB\). On the other hand, any \(z=(x,y)\in\bar{\cS}_\indexB\) satisfies \(y_\indexN=0\), whence \(h(y)=h_\indexB(y)\) and then \(z\in\sols\). Therefore \(\bar{\cS}_\indexB\subseteq\sols\).
\end{proof}

\end{document}